\documentclass{amsart}

\usepackage{amsmath,amssymb,amsthm}

\usepackage{hyperref}
\usepackage{graphicx}

\newcommand{\Log}{\mathop{\mathrm{Log}}\nolimits}
\newcommand{\splitting}{\mathcal{M}}

\newcommand{\arr}{\longrightarrow}
\newcommand{\R}{\mathbb{R}}
\newcommand{\C}{\mathbb{C}}
\newcommand{\Z}{\mathbb{Z}}
\newcommand{\N}{\mathbb{N}}

\newcommand{\X}{\mathcal{X}}

\newcommand{\wt}{\widetilde}

\newcommand{\hide}[1]{}

\newtheorem{theorem}{Theorem}[section]
\newtheorem{proposition}[theorem]{Proposition}

\newtheorem{lemma}[theorem]{Lemma}

\theoremstyle{definition}
\newtheorem{defi}[theorem]{Definition}
\newtheorem{question}[theorem]{Question}
\newtheorem*{remark}{Remark}

\title[Locally connected Smale spaces]{Locally connected Smale spaces, pinched spectrum, and infra-nilmanifolds}
\author{Volodymyr Nekrashevych}
\thanks{Support by the Institut Mittag-Leffler (Djursholm, Sweden) and
  NSF grants DMS1006280 and DMS1709480 is gratefully acknowledged}

\begin{document}
\maketitle

\begin{abstract}
We show that if $(\X, f)$ is a locally connected Smale space (e.g., a basic set of an Axiom-A diffeomorphism) such that the local product
structure on $\X$ can be lifted by a covering with virtually nilpotent
group of deck transformations to a global direct product, then $(\X,
f)$ is topologically conjugate to a hyperbolic infra-nilmanifold
automorphism. We use this result to give a generalization to Smale
spaces of a theorem of M.~Brin and A.~Manning on Anosov
diffeomorphisms with pinched spectrum, and to show that every locally
connected codimension one Smale space is topologically conjugate to a
hyperbolic automorphism of a torus.
\end{abstract}

\tableofcontents

\section{Introduction}

Smale spaces were introduced by D.~Ruelle (see~\cite{ruelle:therm}) as generalizations
of Anosov diffeomorphisms and basic sets of Axiom-A diffeomorphisms. They were also extensively studied
before as ``spaces with hyperbolic canonical coordinates'' by
R.~Bowen~\cite{bowen:markovpartition,bowen:periodicmeasures}.

A Smale space is a compact
metric space $\X$ with a homeomorphism $f:\X\arr\X$ such that there
exists a local direct product structure on $\X$ with respect to which
$f$ is expanding in one and contracting in the other direction.
For example, every Anosov diffeomorphism
of a compact manifold is a Smale space. Restrictions of Axiom-A homemorphisms to the basic sets are also examples of Smale spaces. For more on Smale spaces, see~\cite{ruelle:therm,putnam,putnam:smalespaces}.

Smale spaces are classical objects of the theory of dynamical
systems, but many basic questions about them (and even about
Anosov diffeomorphisms) remain to be open.

For example, it seems that the following question is open.
\begin{question}
\label{quest}
Is it true that if $(\X, f)$ is a Smale space such that $\X$ is
connected and locally connected, then $(\X, f)$ is topologically
conjugate to an Anosov diffeomorphism?
\end{question}

Many well studied examples of Smale spaces are such that
one or both of the factors of the local direct product
structure are totally disconnected, e.g., the shifts of finite type, the Smale solenoid
(see~\cite[Section 1.9]{brin:book}), Williams attractors, etc..  See more 
examples in~\cite{williams:attractors,anderputn}. See also~\cite{wieler},
where it is proved that all such Smale spaces are inverse 
limits of iterations of one self-map, i.e., are natural
generalizations of solenoids.

Note that a question similar to Question~\ref{quest} for expanding
maps has a positive answer.
Namely, the following theorem is proved in
\cite[Theorem~6.1.6]{nek:book} and
\cite[Theorem~5.9]{nek:models} using Gromov's theorem on groups of polynomial growth~\cite{gro:gr}.
\begin{theorem}
\label{th:expanding}
Let $f:\X\arr\X$ be a self-covering map of a locally connected and
connected compact metric
space. Suppose that there exists a covering map $\pi:\wt\X\arr\X$ such
that $f$ can be lifted by $\pi$ to an expanding homeomorphism of
$\wt\X$. (This is true, for example, if $f:\X\arr\X$ is locally
expanding and $\X$ is semi-locally simply connected.) Then $(\X, f)$ is
topologically conjugate to an expanding infra-nilmanifold endomorphism.
\end{theorem}

Here an \emph{infra-nilmanifold endomorphism} is a map
$\phi:G\backslash L\arr G\backslash L$, where $L$ is a simply
connected nilpotent Lie group, $G$ is a subgroup of the affine group
$\mathop{\mathrm{Aut}}(L)\ltimes L$ acting on $L$ freely, properly,
and co-compactly, and $\phi$ is induced by an automorphism $\Phi:L\arr
L$. If the automorphism $\Phi$ is expanding, i.e., if all eigenvalues
of $D\Phi$ have absolute value greater than one, then we say that the
corresponding endomorphism $\phi$ is \emph{expanding}. If $D\Phi$ has no
eigenvalues of absolute value one, then we say that $\phi$ is \emph{hyperbolic}.

Note, that the case when $f:\X\arr\X$ in Theorem~\ref{th:expanding} is
an expanding endomorphism of a Riemannian manifold, is a result of
M.~Gromov~\cite{gro:gr} (based on results of M.~Shub~\cite{shub2}).

All known examples of Anosov diffeomorphisms, and hence apparently all
known examples of locally connected Smale spaces are hyperbolic automorphisms
of infra-nilmanifolds. See~\cite{Smale:diffdynsyst}, and Problem~3 in
the additional list of problems in~\cite{smale:problems}. It was
proved by A.~Manning in~\cite{manning:nilmanifolds} that every Anosov
diffeomorphism of an infra-nilmanifold is topologically conjugate to
a hyperbolic automorphism of an infra-nilmanifold. Another
result in this direction is a theorem of J.~Franks~\cite{franks} and
S.E.~Newhouse~\cite{newhouse:codimone} stating that if $(\X, f)$ is an Anosov
diffeomorphism such that stable or unstable manifolds of $\X$ are
one-dimensional, then $(\X, f)$ is topologically conjugate to a hyperbolic
linear automorphism of the torus $\R^n/\Z^n$.

One of the main obstacles for proving that every Anosov
diffeomorphism is an automorphism of an infra-nilmanifold
is showing that the foliations of $\X$ into stable and unstable
manifolds when lifted to the universal covering $\wt\X$ of $\X$ come from a
direct product decomposition of $\wt\X$.

\begin{defi}
We say that a Smale space $(\X, f)$ is \emph{splittable} if there
exists a covering map $\pi:\splitting\arr\X$ and a direct product decomposition
$\splitting=A\times B$ of $\splitting$ such that $\pi$ maps plaques $\{a\}\times B$
and $A\times\{b\}$ of the direct product decomposition of $\splitting$
bijectively to stable and unstable leaves of $\X$.
\end{defi}

Here a \emph{stable} (resp.\ \emph{unstable}) \emph{leaf} of $\X$ is
the equivalence class
with respect to the equivalence relation $\lim_{n\to+\infty}d(f^n(x),
f^n(y))=0$ (resp.\ $\lim_{n\to-\infty}d(f^n(x), f^n(y))$).

If $\X$ is locally connected and connected, then every splitting is a
Galois covering with a finitely generated group of deck
transformations, see Proposition~\ref{pr:finitegeneration}.

One of the main results of our paper is the following
theorem.

\begin{theorem}
\label{th:main}
Let $(\X, f)$ be a Smale space such that $\X$ is connected and locally
connected. Suppose that it has a splitting with a virtually nilpotent
group of deck transformations. Then $(\X, f)$ is topologically
conjugate to a hyperbolic infra-nilmanifold automorphism.
\end{theorem}

Note that we do not assume in Theorem~\ref{th:main}  that
$\X$ is even locally simply connected. On the other hand, when restricted to the class of Anosov diffeomorphisms, it is a weaker statement than the result of~\cite{manning:nilmanifolds}. It is not clear what should be the statement generalizing A.~Manning's result in the class of locally connected Smale spaces.

M.~Brin in~\cite{brin:dan1,brin:dan2} gave a ``pinching'' condition on the Mather
spectrum of an Anosov diffeomorphism $(\X, f)$
(i.e., spectrum of the operator induced by $f$ on the Banach space of
vector fields on $\X$) ensuring that $(\X, f)$ has a splitting with a
virtually nilpotent group of deck transformations. In the case of
Anosov diffeomorphisms the
splitting map $\pi:\splitting\arr\X$ is necessarily the universal covering map,
so that the group of deck transformations is the fundamental group of
$\X$. M.~Brin and A.~Manning proved then in~\cite{brinmanning} that
all Anosov diffeomorphisms satisfying the Brin's pinching condition
are hyperbolic automorphisms of infra-nilmanifolds.

We generalize the results of M.~Brin and A.~Manning. Of course, we can
not use the original pinching condition, since we do not have vector
fields on Smale spaces. We find, however, a purely topological
condition, which follows from Brin's condition in the case of Anosov
diffeomorphism. In fact, we even improve the Brin's spectral
pinching condition for Anosov diffeomorphisms.

Here is an informal description of our condition. 
Consider a finite covering $\mathcal{R}$ of $\X$ by sufficiently small open
rectangles (i.e., such that their diameters are smaller than the expansivity constant). The covering
will induce coverings of the stable and unstable leaves by the
plaques of the elements of $\mathcal{R}$.
Define, for an stable leaf $V$ and $x,
y\in V$, the combinatorial distance
$d_{\mathcal{R}}(x, y)$ equal to the smallest length
$m$ of a chain $x\in R_0, R_1, \ldots, R_m\ni y$, $R_i\cap
R_{i+1}\ne\emptyset$, of plaques of the elements of $\mathcal{R}$ (which can be infinite).
Then $d_{\mathcal{R}}(f^{-n}(x), f^{-n}(y))$ grows exponentially for $x\ne y$, if it is finite. We say
that $\alpha_0>0$ and $\alpha_1>0$ are \emph{stable lower} and \emph{upper}
\emph{exponents} if there exists $C>1$ such that
\[C^{-1}e^{\alpha_0 n}\le d_{\mathcal{R}}(f^{-n}(x), f^{-n}(y))\le
Ce^{\alpha_1 n}\]
for all stably equivalent $x, y$ such that
the distance between $x$ and $y$ inside their stable
leaf belongs to some fixed interval $[\epsilon_1, \epsilon_2]$ for
$0<\epsilon_1<\epsilon_2$. \emph{Stable upper and lower critical
  exponents} are the infimum and the supremum of all stable upper and
lower exponents, respectively. We prove that
the stable critical exponents are uniquely
determined by the topological conjugacy class of the Smale space and
are positive and finite (if the Smale space is locally connected and connected). The unstable
upper and lower critical exponents are defined in the similar way (they
are stable upper and lower critical exponents of $(\X, f^{-1})$). For
more details, see Sections~\ref{s:lowerexponents}
and~\ref{s:upperexponents}.

Note that if the stable or unstable leafs of $(\X, f)$ are not locally
connected, then the corresponding upper exponents are infinite, since
the metric $d_{\mathcal{R}}(x, y)$ will be infinite for some $x, y$
(as we assume that the covering $\mathcal{R}$ consists of small rectangles).

\begin{theorem}
\label{th:brin}
Suppose that $(\X, f)$ is a Smale space such that $\X$ is connected
and locally connected. Let $a_0, a_1, b_0, b_1$ be the stable lower
and upper, and the unstable lower and upper critical exponents,
respectively. If \[\frac{a_0}{a_1}+\frac{b_0}{b_1}>1\]
then $(\X, f)$ is topologically conjugate to a hyperbolic
infra-nilmanifold automorphism.
\end{theorem}

We show that the Brin's pinching condition on the Mather spectrum of an
Anosov diffeomorphism implies our condition on the critical
exponents. 

As another application of Theorem~\ref{th:main}, we show that the theorem
of J.~Franks and S.E.~Newhouse on co-dimension one Anosov diffeomorphisms is true for all locally connected Smale spaces.

\begin{theorem}
\label{th:codimone}
Let $(\X, f)$ be a Smale space such that $\X$ is connected and locally
connected, and either stable or unstable leaves of $(\X, f)$ are
homeomorphic (with respect to their intrinsic topology)
to $\R$. Then $(\X, f)$ is topologically conjugate to a hyperbolic
linear automorphism of a torus $\R^n/\Z^n$.
\end{theorem}

Here intrinsic topology of a leaf is the direct limit topology coming
from decomposition of a leaf into the union of plaques of rectangles
of $\X$.

Theorem~\ref{th:codimone}, for example, rules out basic sets of Axiom-A diffeomorphisms such that the stable leaves are homeomophic to $\R$, while the unstable leaves are locally connected but not homeomorphic to manifolds (e.g., are locally homeomorphic to the Sierpinski carpet). 

\begin{remark}
A more general notion of an endomorphism of an infra-nilmanifold is
discussed in~\cite{dekimpe:whatis,dekimpe:shouldbe}. It is also
noted there that some of the results of~\cite{franks} and~\cite{shub2} are
based on a false result. The proof of Theorem~\ref{th:main} 
shows that it is enough to
consider the narrower notion of an automorphism of an
infra-nilmanifold in the classification of Smale spaces and Anosov
diffeomorphisms up to topological conjugacy. We do not use the 
results of~\cite{franks} (except for his proof of Theorem~2.2, which
we repeat for our setting). 
The results of~\cite{shub2} are not used in the proof of
Theorem~\ref{th:expanding}, where also the narrower notion of an
endomorphism of an infra-nilmanifold is used,
see~\cite[Theorem~6.1.6]{nek:book}.
\end{remark}

\subsection*{Structure of the paper}

In Section~\ref{s:Smalespaces}, we collect basic facts and definitions
related to Smale spaces, and fix the related notations.

We study lower exponents of a Smale
space, and a family of metrics associated with lower exponents in
Section~\ref{s:lowerexponents}.
We also recall there properties of the SRB measures on leaves of Smale
spaces.

Locally connected Smale spaces are studied in
Section~\ref{s:upperexponents}.
We show that the following conditions for a Smale space $(\X,
f)$ are equivalent (see Theorem~\ref{th:connectedness}):
\begin{enumerate}
\item The space $\X$ is locally connected.
\item All stable and unstable leaves of $\X$ are locally connected.
\item All stable and unstable leaves of $\X$ are connected.
\item $\X$ has finite stable and unstable upper exponents.
\end{enumerate}

In Section~\ref{s:splitting} we study splittings of locally connected 
Smale spaces. We show that for any splitting $\pi:\splitting\arr\X$ of a
locally connected Smale space there exists a well defined group of
deck transformations $G$,
that $G$ is finitely generated, and that there exists a lift $F:\splitting\arr\splitting$ of
$f$ to $\splitting$, which is unique up to compositions with elements of $G$.

The lift $F$ defines then, for any point $x_0\in\splitting$, an automorphism
$\phi$ of $G$ by the rule $F(g(x_0))=\phi(g)(F(x_0))$.

\begin{question}
Does the pair $(G, \phi)$ uniquely determine the topological conjugacy
class of $(\X, f)$?
\end{question} 

We do not know the answer to this question, but we show that we can
reconstruct $(\X, f)$ after adding an extra piece of information to $(G,
\phi)$.

\begin{defi}
Let $\pi:\splitting\arr\X$ be a splitting of a locally connected and
connected Smale space $(\X, f)$. Let $W_+$ and $W_-$ be stable and
unstable plaques of a fixed point $x_0$ of a lift of $f$. We say that
$\Sigma_+, \Sigma_-\subset G$ are \emph{coarse stable and unstable
  plaques} if the Hausdorff distances between $\Sigma_+(x_0)$ and
$W_+$ and between $\Sigma_-(x_0)$ and $W_-$ are finite.
\end{defi}

Here the distance in $\splitting$ is measured with respect to a $G$-invariant
metric.

\begin{theorem}
\label{th:conjugacycriterionintro}
The quadruple $(G, \phi, \Sigma_+, \Sigma_-)$ uniquely determines the
topological conjugacy class of $(\X, f)$.
\end{theorem}

We prove Theorem~\ref{th:conjugacycriterionintro} by representing
$W_+$ and $W_-$ as boundaries of Gromov hyperbolic graphs constructed
using the quadruple. These graphs are quasi-isometric to Cayley
graphs of the \emph{Ruelle groupoids} associated with the Smale
space. A general theory of Cayley graphs of \emph{hyperbolic
  groupoids} is developed in~\cite{nek:hyperbolic}. We hope that these new techniques will be helpful in future studies of hyperbolic dynamics.

We get the following corollary of
Theorem~\ref{th:conjugacycriterionintro}.

\begin{theorem}
\label{th:semiconjintro}
Let $(\X_i, f_i)$ for $i=1, 2$ be connected and locally connected
Smale spaces. Let $\pi_i:\splitting_i\arr\X_i$ be splittings, and let
$F_i:\splitting_i\arr\splitting_i$ be lifts of $f_i$. Suppose that
$F_i$ have fixed points, and that the groups of deck transformations
of $\pi_i$ are both isomorphic to a group $G$. If there exists a continuous
map $\Phi:\splitting_1\arr\splitting_2$ such that $\Phi(g(x))=g(\Phi(x))$ and
$\Phi(F_1(x))=F_2(\Phi(x))$ for all $x\in\splitting_1$ and $g\in G$,
then the Smale spaces $(\X_1, f_1)$ and $(\X_2, f_2)$ are
topologically conjugate.
\end{theorem}

Section~\ref{s:nilpotentsplitting} is devoted to the proof of
Theorem~\ref{th:main}. As the first step we prove the following.

\begin{proposition}
\label{pr:hyperbolicandfixed}
Let $\pi:\splitting\arr\X$ be
a splitting of a locally connected and connected Smale space $(\X,
f)$ such that the group $G$ of deck transformations is torsion free
nilpotent. Then $f$ has a
fixed point, and the associated automorphism $\phi:G\arr G$ is
hyperbolic (i.e., its unique extension $\Phi$ to a simply connected
nilpotent Lie group containing $G$ as a lattice is hyperbolic).
\end{proposition}

We prove both statements of Proposition~\ref{pr:hyperbolicandfixed}
by induction on the nilpotency class of
$G$. 
We show at first that the automorphism $\phi$ induces a
hyperbolic automorphism of the center $Z(G)\cong\Z^n$ of $G$.
Then we construct
an action of $\R^n$ on $\splitting$ naturally extending the action of
$\Z^n$, using the direct product structure on $\splitting$.
The action induces an action of the torus $\R^n/\Z^n$ on $\X$
and agrees with the local product structure, metric on $\X$, and the
dynamics, in such a way that the map induced by $f$ on $(\R^n/\Z^n)\backslash
\X$ is a Smale space with a splitting with the group of deck
transformations isomorphic to $G/Z(G)$. This provides us the
necessary inductive steps to prove Proposition~\ref{pr:hyperbolicandfixed}.
We use after that the arguments of~\cite[Theorem~2.2]{franks}, Theorem~\ref{th:semiconjintro}, and some additional algebraic arguments to prove Theorem~\ref{th:main}.

Theorem~\ref{th:brin} generalizing the Brin's pinching condition to
Smale spaces is proved in Section~\ref{s:pinched}. We prove
at first that every Smale space satisfying conditions of
Theorem~\ref{th:brin} has a splitting
(Theorem~\ref{th:developable}). Then we prove that the group of deck
transformations of the splitting 
is virtually nilpotent (Theorem~\ref{th:growth}) using
Gromov's theorem on groups of polynomial growth. Both proofs are
similar to the original proofs of M.~Brin, except that in the proof of
Theorem~\ref{th:developable} we use results of
Section~\ref{s:lowerexponents} on lower exponents of a Smale space,
which allows us to get a better pinching condition, and to
prove the theorem for all locally connected Smale space, and not only
for Anosov diffeomorphisms.

In Section~\ref{s:mather}, we show how our condition on critical
exponents is related to M.~Brin's pinching condition on the Mather
spectrum of a diffeomorphism. We show that M.~Brin's condition implies
the condition of Theorem~\ref{th:brin}.

Section~\ref{s:codimone} is devoted to the proof
Theorem~\ref{th:codimone} on co-dimension one Smale spaces. We prove it using Theorem~\ref{th:main} and the ideas of the proof of Theorem~\ref{th:brin}. 

\subsection*{Acknowledgments}
I started writing this paper during a visit to
Institut Mittag-Leffler (Djursholm, Sweden) 
as a participant of the semester ``Geometric and
Analytic Aspects of Group Theory'' in March of 2012.  I am very grateful to the
Institute for excellent conditions for work and to the
organizers of the semester for inviting me. I am also grateful to the
anonymous referee, who found an error in the first version of the paper and helped to improve the exposition.

This paper is based upon work supported by the National Science
Foundation under grants DMS1006280 and DMS1709480.

\section{Smale spaces}
\label{s:Smalespaces}

\subsection{Local product structures}

\begin{defi}
\label{def:directproduct}
A \emph{direct product structure} on a topological space $R$ is
defined by a continuous map $[\cdot, \cdot]:R\times R\arr R$
satisfying
\begin{enumerate}
\item $[x, x]=x$ for all $x\in R$;
\item $[[x, y], z]=[x, z]$ and $[x, [y, z]]=[x, z]$ for all $x, y,
  z\in R$.
\end{enumerate}
We call a space with a direct product structure on it a \emph{rectangle}.
\end{defi}

If $R=A\times B$ is a decomposition of $R$ into a direct product of
two topological spaces, then the corresponding direct product
structure is given by the operation
\begin{equation}
\label{eq:directproduct}
[(x_1, y_1), (x_2, y_2)]=(x_1, y_2).
\end{equation}

Let $R$ be a rectangle. For $x\in R$ the corresponding \emph{plaques} are the sets
\begin{equation}\label{eq:plaques}
P_1(R, x)=\{y\in R\;:\;[x, y]=x\},\qquad P_2(R, x)=\{y\in R\;:\;[x,
y]=y\}.
\end{equation}
See Figure~\ref{fig:rectangle}. If $R=A\times B$ with the
corresponding direct product structure~\eqref{eq:directproduct}, then
the plaques are given by
\[P_1(R, (a, b))=A\times\{b\},\qquad P_2(R, (a, b))=\{a\}\times B.\]

The map $P_1(R, x)\times P_2(R, x)\arr R$ given by
\[(y_1, y_2)\mapsto [y_1, y_2]\]
is a homeomorphism. 

For any pair $x, y\in R$ the natural maps
$P_1(R, x)\arr P_1(R, y)$ and $P_2(R, x)\arr P_2(R, y)$
given by $z\mapsto [y, z]$ and $z\mapsto [z, y]$, respectively, are
called \emph{holonomy maps} inside $R$, and are homeomorphisms.

These homeomorphism agree with the homeomorphisms
$P_1(R, x)\times P_2(R, x)\arr R$, so that we get a canonical
decomposition of $R$ into the direct product
of two spaces $A$ and $B$, which can be
identified with $P_1(R, x)$ and $P_2(R, x)$, respectively.

\begin{figure}
\includegraphics{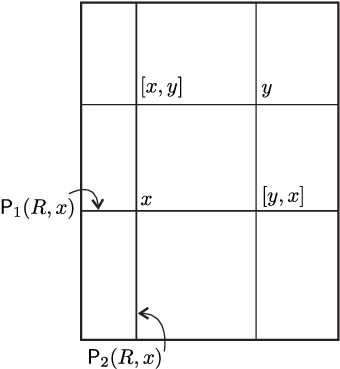}
\caption{Rectangle}
\label{fig:rectangle}
\end{figure}

\begin{defi}
\label{def:locprod}
Let $\X$ be a topological space. A \emph{local product
  structure} on $\X$ is given by a covering $\mathcal{R}$ of $\X$ by open sets $R$
with a direct product structure $[\cdot, \cdot]_R$ on each of them,
such that for any $R_1, R_2\in\mathcal{R}$, and for every $x\in\X$
there exists a neighborhood $U$ of $x$ such that $[y_1,
y_2]_{R_1}=[y_1, y_2]_{R_2}$ for all $y_1, y_2\in U\cap R_1\cap R_2$.

Two coverings of $\X$ by open rectangles define the same local product
structures if their union defines a local product structure, i.e.,
satisfied the above compatibility condition.
\end{defi}

If $\X$ is a space with a local direct product structure, then an open
subset $R\subset\X$ with a direct product structure $[\cdot, \cdot]$
is a \emph{(sub-)rectangle} of $\X$ if the union of $\{R\}$ with a covering
defining the local product structure satisfies the compatibility
conditions of Definition~\ref{def:locprod}.

\begin{defi}
\label{def:preserveslocprod}
We say that a continuous map $f:\X_1\arr\X_2$ between spaces with
local product structures \emph{preserves the local product structures}
if every point of $\X_1$ has a rectangular neighborhood $U$ such that
$f(U)$ is a rectangle of $\X_2$, and $f([x, y]_U)=[f(x), f(y)]_{f(U)}$
for all $x, y\in U$.
\end{defi}

\begin{defi}
\label{def:prodmetric}
Let $\X$ be a space with a local product structure. We say that a
metric $d$ on $\X$ \emph{agrees with the local product structure} if
for every point $x\in\X$ there exists an open rectangular neighborhood
$R=A\times B$ of $x$ and metrics $d_A$ and $d_B$ on $A$ and $B$,
respectively, such that the restriction of $d$ to $R$ is bi-Lipschitz
equivalent to the metric
\[d_R((x_1, y_1), (x_2, y_2))=d_A(x_1, x_2)+d_B(y_1, y_2).\]
\end{defi}

If a metric $d$ agrees with the local product structure, then for
every point $x\in\X$ there exists a rectangular neighborhood $R$ of
$x$ such that all holonomy maps inside $R$ are bi-Lipschitz with respect
to the metric $d$ with a fixed Lipschitz constant (depending only on
$R$). Conversely, it is easy to see that 
a metric $d$ agrees with the local product
structure if for every $x\in\X$ there exists a rectangular neighborhood
$R$ of $x$ such that the holonomies inside $R$ are uniformly
bi-Lipschitz, and $d(y, z)$ is bi-Lipschitz equivalent to $d([x, y],
[x, z])+d([y, x], [z, x])$.

\subsection{Smale spaces}

\begin{defi}
\label{def:smalespaces}
A \emph{Smale space} is a compact metrizable space $\X$ together with a
homeomorphism $f:\X\arr\X$ such that there exists a metric $d$ on
$\X$, constants $\lambda\in (0, 1)$ and $C>0$, and a local product structure on
$\X$ such that $f$ preserves the local product structure and for every
$x\in\X$ there exists a rectangular neighborhood $R$ of $x$ such that
for all $n\ge 0$ and $y, z\in P_1(R, x)$ we have
\[d(f^n(y), f^n(z))\le C\lambda^n d(y, z),\]
and for all $n\ge 0$ and $y, z\in P_2(R, x)$ we have
\[d(f^{-n}(y), f^{-n}(z)\le C\lambda^n d(y, z).\]
\end{defi}

We will denote $P_1(R, x)=P_+(R, x)$ and $P_2(R, x)=P_-(R, x)$.

Examples of Smale spaces are Anosov diffeomorphisms of \emph{compact}
manifolds, restrictions of Axiom-A diffeomorphisms to their basic
sets, shifts of finite type, spaces of substitutional tilings,
etc.. See~\cite{putnam:smalespaces} for more examples. Note that pseudo-Anosov
diffeomorphisms are not Smale spaces. 

\begin{defi}
\label{def:expansive}
A homeomorphism $f:\X\arr\X$ of a compact space $\X$ is said to be
\emph{expansive} if there exists a neighborhood $U$ of the diagonal in
$\X\times\X$ such that $(f^n(x), f^n(y))\in U$ for all $n\in\Z$ implies
$x=y$.
\end{defi}

Note that if $U$ satisfies the conditions of the definition, then
$\{(x, y)\in\X^2\;:\;(x, y), (y, x)\in U\}$ also satisfies the
conditions of the definition. Consequently, we may assume that $U$ is
symmetric.

\begin{proposition}
Every Smale space is an expansive dynamical system.
\end{proposition}

\begin{proof}
We can find a finite covering $\mathcal{R}$ of $\X$ by rectangles
satisfying the conditions of Definition~\ref{def:smalespaces}. Let
$\epsilon>0$ be a Lebesgue's number of the covering. There exists
$\delta>0$ such that for any points $x, y\in\X$ such that $d(x,
y)<\delta$ and any rectangle $R\in\mathcal{R}$ such that $x, y\in
R$ we have $d(x, [x, y])<C^{-1}\epsilon$ and $d(y, [x, y])<C^{-1}\epsilon$.

Let $x, y\in\X$ be such that $d(f^n(x), f^n(y))<\delta$ for all $n\in\Z$. Then for
every $n\in\Z$ there exists a rectangle $R_n\in\mathcal{R}$ such that
$f^n(x), f^n(y)\in R_n$. Then $d(f^n(x), [f^n(x),
f^n(y)])<C^{-1}\epsilon$. Note that $f^n(x)$ and $[f^n(x), f^n(y)]$ belong
to one plaque $P_-(R_n, f^n(x))$. It follows that
$d(f^{n-k}(x), [f^{n-k}(x), f^{n-k}(y)])\le C\lambda^k d(f^n(x), [f^n(x),
f^n(y)])<C\lambda^k C^{-1}\epsilon=\lambda^k\epsilon$ for all $k\ge 0$ and all
$n\in\Z$. In particular, $d(x, [x, y])<\lambda^k\epsilon$ for all
$k\ge 0$, i.e., $x=[x, y]$. It is shown in the same way that $y=[x,
y]$, which implies that $x=y$. Therefore, the set
$U\subset\X\times\X$ equal to the set of pairs $(x, y)$ such that
$d(x, y)<\delta$ satisfies the conditions of Definition~\ref{def:expansive}.
\end{proof}

\begin{defi}
\label{def:logscale}
 A \emph{log-scale} on a set $X$ is a function
 $\ell:\X\times\X\arr\R\cup\{\infty\}$ satisfying the following
 conditions:
\begin{enumerate}
\item $\ell(x, y)=\ell(y, x)$ for all $x, y\in X$;
\item $\ell(x, y)=\infty$ if and only if $x=y$;
\item there exists $\Delta>0$ such that
\[\ell(x, z)\ge\min\{\ell(x, y), \ell(y, z)\}-\Delta\]
for all $x, y, z\in X$.
\end{enumerate}
\end{defi}

We say that two log-scales $\ell_1, \ell_2$ are \emph{bi-Lipschitz
  equivalent} if the difference $|\ell_1(x, y)-\ell_2(x, y)|$ is
uniformly bounded for all $x\ne y$.

Let us  describe the natural class of metrics on expansive dynamical
systems defined in~\cite{fried:naturalmetric}, using log-scales.

Let $(\X, f)$ be an expansive dynamical system. Let $U$ be a symmetric
neighborhood of the diagonal, satisfying the conditions of Definition~\ref{def:expansive}.
Define $\ell(x, y)$ for $x, y\in\X$ to be maximal $n$ such that
$(f^k(x), f^k(y))\in U$ for all $k\in [-n, n]$.

\begin{lemma}
\label{lem:standlogscale}
The defined function $\ell$ is a log-scale. It does not depend, up to
bi-Lipschitz equivalence, on the choice of $U$.
\end{lemma}

We call $\ell$ \emph{the standard log-scale} of the expansive dynamical system.

\begin{proof}
We have $\ell(x, y)=\ell(y, x)$, since we assume that $U$ is
symmetric. We also have $\ell(x, y)=\infty$ if and only if $x=y$, by
Definition~\ref{def:expansive}.

It remains to show that there exists
$\Delta$ such that $\ell(x, z)\ge\min\{\ell(x, y), \ell(y,
z)\}-\Delta$ for all $x, y, z\in\X$.

Since a compact set has a unique uniform structure consisting of all
neighborhoods of the diagonal (see~\cite{bour:top}), there
exists a neighborhood of
the diagonal $V\subset\X^2$ such that $(x, y), (y, z)\in V$ implies
$(x, z)\in U$.

Note that the sets $U_n=\{(x, y)\;:\;\ell(x, y)\ge
n\}=\bigcap_{k=-n}^nf^k(U)$ are neighborhoods of the diagonal,
$U_{n+1}\subseteq U_n$ for all $n$, and $\bigcap_{n\ge 1}U_n$ is equal
to the diagonal. In particular, by compactness of $\X$, there exists
$\Delta>0$ such that $U_\Delta\subset V$.

Denote by $V_n=\bigcap_{k=-n}^nf^k(V)$ the set of pairs $(x, y)$ such
that $(f^k(x), f^k(y))\in V$ for all $k=-n, \ldots, n$. Then $(x, y),
(y, z)\in V_n$ implies $(x, z)\in U_n$.

Then for every $n>\Delta$ we have
$U_n\subset V_{n-\Delta}$,
since the conditions that $(f^k(x), f^k(y))\in U$ for all
$|k|\le n$ implies $(f^k(x), f^k(y))\in U_\Delta\subset V$ for
all $|k|\le n-\Delta$.

Let $\min\{\ell(x, y), \ell(y, z)\}=m$. Then $(x, y), (y,
z)\in U_m\subset V_{m-\Delta}$, hence $(x, z)\in U_{m-\Delta}$, i.e.,
$\ell(x, z)\ge m-\Delta$.

Let us show that $\ell$ does not depend on the choice of $U$. Let $U'$
and $U''$ be two neighborhoods of the diagonal, satisfying the
conditions of Definition~\ref{def:expansive}. Then, as above, there
exists $C>0$ such that $U'_C\subset U''$ and $U''_C\subset U'$. By the
same arguments as above, we conclude that $U'_{n+C}\subset U''_n$
and $U''_{n+C}\subset U'_n$ for all $n\ge 0$. But this implies that
the values of the log-scales defined by $U'$ and $U''$ differ from
each other not more than by $C$.
\end{proof}

It is proved in~\cite[Lemma~5.4.2]{nek:hyperbolic} that for every Smale space $(\X, f)$ the log-scale $\ell$ agrees with the local product structure on $\X$.

\begin{defi}
\label{def:leaves}
Let $(\X, f)$ be an expansive dynamical system, and let $\ell$ be the
standard log-scale. We say that $x, y\in\X$ are
\emph{stably equivalent} (denoted $x\sim_+ y$) if $\ell(f^n(x), f^n(y))\to+\infty$ as
$n\to+\infty$. They are \emph{unstably equivalent} (denoted $x\sim_-
y$) if $\ell(f^{-n}(x), f^{-n}(y))\to+\infty$ as $n\to+\infty$.
We call stable and unstable equivalence classes \emph{stable and
  unstable leaves}.
\end{defi}

Note that $\ell(x_n, y_n)\to\infty$ is equivalent to $d(x_n, y_n)\to
0$ for any pair of sequences $x_n, y_n\in\X$ and for any metric $d$ on $\X$.

Two points $x, y\in\X$ are stably equivalent if and only if $(f^n(x),
f^n(y))\in U$ for all $n$ big enough. Denote, for $x\in\X$ and $n\in\Z$,
by $W_{n, +}(x)$ the set of points $y\in\X$ such that $(f^k(x),
f^k(y))\in U$ for all $k\ge -n$. Similarly, we denote by $W_{n, -}(x)$
the set of points $y\in\X$ such that $(f^k(x), f^k(y))\in U$ for all
$k\le n$.

Then $W_{n, +}(x)$ and $W_{n, -}(x)$ are decreasing sequences of sets, and
$W_+(x)=\bigcup_{n\in\N}W_{-n, +}(x)$ and
$W_-(x)=\bigcup_{n\in\N}W_{-n, -}(x)$ are
equal to the stable and unstable leaves of $x$, respectively.

Note that for all $n\in\N$, $x\in\X$, $*\in\{+, -\}$, and
$y_1, y_2\in W_{n, *}(x)$ we have $\ell(y_1, y_2)\ge n$.

If $W$ is a stable leaf, then we denote, for $y_1, y_2\in W$, by
$\ell_+(y_1, y_2)$ or $\ell_W(y_1, y_2)$ the biggest
$n_0$ such that $(f^n(y_1), f^n(y_2))\in U$ for all $n\ge -n_0$.

The following properties of $\ell_+$ follow directly from the
definitions.
\begin{itemize}
\item $\ell(y_1, y_2)\le\ell_+(y_1, y_2)$ for all stably equivalent
  $y_1, y_2$;
\item if $y_1, y_2$ are stably equivalent and $\ell_+(y_1, y_2)>0$,
  then $\ell(y_1, y_2)=\ell_+(y_1, y_2)$;
\item for all stably equivalent $y_1, y_2$ we have
\begin{equation}
\label{eq:ellplusf}
\ell_+(f(y_1), f(y_2))=\ell_+(y_1, y_2)+1.
\end{equation}
\end{itemize}

Similarly, if $W$ is an unstable leaf, then $\ell_-(y_1,
y_2)=\ell_W(y_1, y_2)$, for $y_1, y_2\in W$, is the biggest
$n_0$ such that $(f^n(y_1), f^n(y_2))\in U$ for all $n\le n_0$.
We also have $\ell(y_1,
y_1)\le\ell_-(y_1, y_2)$, $\ell(y_1, y_2)=\ell_-(y_1, y_2)$ if
$\ell_-(y_1, y_2)>0$, and
\begin{equation}
\label{eq:ellminusf}
\ell_-(f(y_1), f(y_2))=\ell_-(y_1, y_2)-1
\end{equation}
for all pairs $y_1, y_2$ of unstably equivalent points.

\begin{lemma}
\label{lem:ellplusellminus}
Let $W$ be a stable or unstable leaf. Then the
corresponding function $\ell_+$ or $\ell_-$ is a log-scale on $W$.
\end{lemma}

\begin{proof}
If $\ell_+(x, y), \ell_+(y, z), \ell_+(x, z)$ are all positive, then
they are equal to the corresponding values of $\ell$, hence, by
Lemma~\ref{lem:standlogscale},
$\ell_+(x, z)\ge\min(\ell_+(x, y), \ell_+(y, z))-\Delta$, for $\Delta$ not
depending on $x, y, z$.

If they are not positive, then we can find $n\ge 0$ such that
$\ell_+(f^n(x), f^n(y))=\ell_+(x, y)+n$, $\ell_+(f^n(y),
f^n(z))=\ell_+(y, z)+n$, and $\ell_+(f^n(x), f^n(z))=\ell_+(x, z)+n$
are positive, and applying the above argument for $f^n(x), f^n(y)$,
and $f^n(z)$ conclude that $\ell_+(x, z)+n\ge\min(\ell_+(x, y)+n,
\ell_+(y, z)+n)-\Delta$, which is equivalent to $\ell_+(x,
z)\ge\min(\ell_+(x, y), \ell_+(y, z))-\Delta$.
\end{proof}

We call the log-scales $\ell_+$ and $\ell_-$ the \emph{internal}
log-scales on the respective leaf. 

The \emph{internal topology} on a
leaf is the topology defined by the corresponding log-scale $\ell_+$
or $\ell_-$. Here topology defined by a log-scale $\ell$ on a set $X$
is given by the the basis $B(n, x)=\{y\in
X\;:\;\ell(x, y)\ge n\}$ of neighborhoods of points $x\in X$. Note
that $B(n, x)$ is not necessarily open or closed.

Equivalently, the internal topology of a leaf $W$ is equal to the
direct limit topology of representation of $W$ as the union of the
sequence $W_{-n, *}(x)$ for $n\in\N$ and $x\in W$.

Note also that leaves of a Smale space are locally compact, since
neighborhoods of points of a leaf are continuous images of
neighborhoods of points of $\X$.


\subsection{Irreducible Smale spaces}

Let $(\X, f)$ be a Smale space. A point $x\in\X$ is said to be \emph{non-wandering} if for every neighborhood $U$ of $x$ there exists a positive integer $n$ such that $f^n(U)\cap U\ne\emptyset$. The set of non-wandering points is obviously $f$-invariant and closed.

We say that $(\X, f)$ is \emph{irreducible} if for every pair of open
sets $U, V\subset \X$ there exists a positive integer $n$ such that
$f^n(U)\cap V\ne\emptyset$. We say that it is \emph{mixing} if for every pair of open sets $U, V\subset\X$ there exists $N$ such that $f^n(U)\cap V\ne\emptyset$ for all $n\ge N$.

Smale spaces were introduced by D.~Ruelle~\cite{ruelle:therm} as purely topological generalizations of basic sets of Axiom-A diffeomorphisms. Note, however, that general Smale spaces may have non-empty wandering sets. For example, every shift of finite type is a Smale space, but shifts of finite type may have wandering points.

On the other hand, the Smale's
Spectral Decomposition Theorem~\cite[Theorems~6.2, 6.6]{Smale:diffdynsyst} holds for Smale spaces.

\begin{theorem}
\label{th:spectdecomp}
Let $(\X, f)$ be a Smale space. Then the dynamical system $(NW(\X), f)$, where $NW(\X)$ is the set of non-wandering points, is a Smale space. The set $NW(\X)$ can be decomposed into a finite disjoint union of closed $f$-invariant sets $\X_1, \X_2, \ldots, \X_n$ such that $(\X_i, f)$ is an irreducible Smale space without wandering points.

We write $\X_i\prec\X_j$ if there exists a wandering point $x\in\X$ such that the set of the accumulation points of $f^n(x)$, $n\ge 0$, is contained in $\X_j$, and the set of the accumulation points of $f^n(x)$, $n\le 0$, is contained in $\X_i$. Then $\prec$ is a partial order on the set $\{\X_1, \ldots, \X_n\}$.
\end{theorem}

The sets $\X_1, \ldots, \X_n$ are called \emph{irreducible components}
of the Smale space $(\X, f)$.

A proof of the above theorem is similar to the proof of the classical
Smale's spectral decomposition theorem, and can be found in~\cite{putnam:smalespaces}. We also have the following relation between the notions of an
irreducible and mixing Smale spaces, see~\cite{putnam:smalespaces}.
 
\begin{theorem}
\label{th:mixing}
Suppose that $(\X, f)$ is an irreducible Smale space. Then $\X$ can be
decomposed into a finite union $\X=A_1\sqcup A_2\sqcup\cdots\sqcup
A_k$ of disjoint clopen sets cyclically permuted by $f$ and such that
$(A_i, f^k)$ are mixing Smale spaces.
\end{theorem}

\section{Lower exponents}
\label{s:lowerexponents}

\subsection{Lower exponents of log-scales}
We say that $d$ is a metric \emph{associated}
with a log-scale $\ell$, if there exist constants $\alpha>0$ and $C>1$ such
that
\[C^{-1}e^{-\alpha\ell(x, y)}\le d(x, y)\le Ce^{-\alpha\ell(x, y)}.\]
The number $\alpha$ is called the \emph{exponent} of the
metric. Topology defined by an associated metric obviously coincides
with the topology defined by the log-scale.

Note that if $d$ is a metric associated with $\ell$ of exponent
$\alpha$, then for any $0<r<1$ the function $(d(x, y))^r$ is
a metric associated with $\ell$ of exponent $r\alpha$. It follows that
the set of exponents $\alpha$ for which there exists a metric
associated with a given log-scale is an interval of the form $(0,
\alpha_0)$ or $(0, \alpha_0]$, where $\alpha_0\in [0, \infty]$. We
will see later that $\alpha_0>0$ (see also~\cite{nek:hyperbolic}).
The number $\alpha_0$ is the \emph{metric critical exponent} of the
log-scale.

Let $X$ be a set with a log-scale $\ell$.
Let $\Gamma_n$, for $n\in\R$, be the graph with
the set of vertices $X$ in which two points $x, y$ are connected by an edge
if and only if $\ell(x, y)\ge n$. Denote then by
$d_n$ the combinatorial distance in $\Gamma_n$ (we assume that $d_n(x,
y)=\infty$ if $x$ and $y$ belong to different connected components of $\Gamma_n$).

\begin{proposition}
\label{pr:lowerexp}
Let $\Delta$ be such as in Definition~\ref{def:logscale}.
There exist $C>0$ such that
\[d_n(x, y)\ge Ce^{\alpha(n-\ell(x, y))}\] for all $x,
y\in X$ and all $n\in\mathbb{N}$, where $\alpha=\frac{\log 2}{\Delta}$.
\end{proposition}

\begin{proof}
If $(x_0, x_1, x_2)$ is a path in $\Gamma_n$, then $\ell(x_0, x_2)\ge
n-\Delta$, hence $(x_0, x_2)$ is a path in $\Gamma_{n-\Delta}$. It
follows that $d_{n-\Delta}(x, y)\le \frac 12(d_n(x, y)+1)$. In other terms:
\[d_{n+\Delta}(x, y)\ge 2d_n(x, y)-1.\]

If $\ell(x, y)=m$, then $d_{m+1}(x, y)\ge 2$, and hence
\[d_{m+1+k\Delta}(x, y)\ge 2^{k+1}-2^{k-1}-2^{k-2}-\cdots
-1=2^k+1\]
Note that $d_m(x, y)\ge d_n(x, y)$ whenever $m\ge n$. It follows that
for $k=\left\lfloor\frac{n-\ell(x, y)-1}{\Delta}\right\rfloor\ge \frac{n-\ell(x,
  y)-1}\Delta-1$ we have
\[d_n(x, y)\ge d_{\ell(x, y)+1+k\Delta}(x, y)>2^k.\]
Consequently,
\[d_n(x, y)\ge 2^{(n-\ell(x,
  y)-1-\Delta)/\Delta}=Ce^{\alpha(n-\ell(x, y))}\]
for all $x, y\in X$ and $n\in\R$, where $C=2^{(-1-\Delta)/\Delta}$ and
$\alpha=\frac{\ln 2}\Delta$.
\end{proof}

\begin{defi}
\label{def:lowerexponent}
We say that $\alpha$ is a \emph{lower exponent} of a log-scale $\ell$
if there exists $C>0$ such that
\[d_n(x, y)\ge Ce^{\alpha(n-\ell(x, y))}\]
for all $x, y\in X$ and $n\in\Z$. The supremum of
all lower exponents is called the \emph{lower critical exponent}.
\end{defi}

The proof of the following proposition is straightforward.

\begin{proposition}
\label{pr:bLipschlexp}
Let $\ell_1$ and $\ell_2$ be bi-Lipschitz equivalent log-scales on
$X$. A number $\alpha>0$ is a lower exponent of $\ell_1$ if and only
if it is a lower exponent of $\ell_2$.
\end{proposition}

\begin{theorem}
The metric critical exponent of a log-scale $\ell$ is equal to its
lower critical exponent. In particular, the metric critical exponent is
positive.
\end{theorem}

\begin{proof}
Let $d$ be a metric on $X$ of exponent $\alpha$ associated with
$\ell$, and let $C_1>1$ be such that
\[C_1^{-1}e^{-\alpha\ell(x, y)}\le d(x, y)\le C_1e^{-\alpha\ell(x,
  y)}\]
for all $x, y\in X$.

Then for every $n$ the inequality $\ell(x, y)\ge n$ implies $d(x,
y)\le C_1e^{-\alpha n}$, hence
\[d(x, y)\le C_1d_n(x, y)e^{-\alpha n}\] for all $x, y$. It follows that
\[d_n(x, y)\ge C_1^{-1}d(x, y)e^{\alpha n}\ge C_1^{-2}e^{\alpha
  n-\alpha\ell(x, y)},\]
for all $x, y\in X$ and $n\in\N$, i.e., $\alpha$ is a lower exponent.

Let $\alpha$ be a lower exponent. Let $\beta$ be an arbitrary number such
that $\alpha>\beta>0$. It is enough to show that there exists a metric
on $X$ of exponent $\beta$ associated with $\ell$.

Define, for $x, y\in X$, $d_\beta(x, y)$ as the infimum of
$\sum_{i=1}^me^{-\beta\ell(x_{i-1}, x_i)}$
over all sequences $x_0=x, x_1, x_2, \ldots, x_m=y$. The function
$d_\beta(x, y)$ obviously satisfies the triangle inequality, is symmetric,
and
\[d_\beta(x, y)\le e^{-\beta\ell(x, y)}\]
for all $x, y\in X$.

It remains to prove that
there exists a constant $C_2$ such that for any sequence $x_0=x, x_1,
x_2, \ldots, x_m=y$ we have
\[\sum_{i=1}^me^{-\beta\ell(x_{i-1}, x_i)}\ge C_2e^{-\beta\ell(x, y)}.\]

Let $C$ be such that $0<C<1$ and $d_n(x, y)\ge Ce^{\alpha(n-\ell(x,
  y))}$ for all $x, y\in X$ and all $n$.
Let us prove our statement by induction on $m$ for
$C_2=\exp\left(\frac{\beta(\log
    C-2\alpha\Delta)}{\alpha-\beta}\right)$.

The statement is true for $m=1$, since $C_2<1$.
Suppose that we have proved it for
all $k<m$, let us prove it for $m$.

\begin{lemma}
\label{lem:subsequence}
Let $x_0, x_1, \ldots, x_m$ be a sequence such that $\ell(x_i,
x_{i+1})\ge n$ for all $i=0, 1, \ldots, m-1$. Let $n_0\le n$. Then
there exists a sub-sequence $y_0=x_0, y_1, \ldots, y_{t-1}, y_t=x_m$
of the sequence $x_i$ such that
\[n_0-2\Delta\le\ell(y_i, y_{i+1})< n_0\]
for all $i=0, 1, \ldots, t-1$.
\end{lemma}

\begin{proof}
Let us construct the subsequence $y_i$ by the following algorithm. Define $y_0=x_0$.
Suppose we
have defined $y_i=x_r$ for $r<m$. Let $s$ be the largest index such that $s>r$
and $\ell(x_r, x_s)\ge n_0$. Note that since $\ell(x_r, x_{r+1})\ge n\ge
n_0$, such $s$ exists.

If $s<m$, then $\ell(x_r, x_{s+1})<
n_0$, and
\[\ell(x_r, x_{s+1})\ge \min\{\ell(x_r, x_s), \ell(x_s,
x_{s+1})\}-\Delta\ge\min\{n_0, \ell(x_s,
x_{s+1})\}-\Delta=n_0-\Delta.\]
Define then $y_{i+1}=x_{s+1}$. We have
\[n_0-\Delta\le\ell(y_i, y_{i+1})<n_0.\]
If $s+1=m$, we stop and get our sequence $y_0, \ldots, y_t$, for $t=i+1$.

If $s=m$, then $\ell(x_r, x_m)=\ell(y_i, x_m)\ge n_0$, and
\[\ell(y_{i-1}, x_m)\ge\min\{\ell(y_{i-1}, y_i), \ell(y_i,
x_m)\}-\Delta\ge\min\{n_0-\Delta, n_0\}-\Delta=n_0-2\Delta\]
and
\[\ell(y_{i-1}, x_m)<n_0,\]
since $y_i$ was defined and was not equal to $x_m$. Then we redefine
$y_i=x_m$ and stop the algorithm.

In all the other cases we repeat the procedure. It is easy to see that
at the end we get a sequence $y_i$ satisfying the conditions of the lemma.
\end{proof}

Let $x_0=x, x_1, \ldots, x_m=y$ be an arbitrary sequence of points of
$X$. Let $n_0$ be the minimal value of $\ell(x_i, x_{i+1})$.
Let $y_0=x, y_1, \ldots, y_t=y$ be a sub-sequence of the sequence
$x_i$ satisfying conditions of Lemma~\ref{lem:subsequence}.

Suppose at first that
\[n_0<\ell(x, y)+\frac{2\alpha\Delta-\log C}{\alpha-\beta}.\]

Remember
that $n_0=\ell(x_i, x_{i+1})$ for some $i$, hence
\[\sum_{i=1}^me^{-\beta\ell(x_{i-1}, x_i)}\ge e^{-\beta
  n_0}>\exp\left(-\beta\ell(x, y)-\frac{\beta(2\alpha\Delta-\log
    C)}{\alpha-\beta}\right)=C_2e^{-\beta\ell(x, y)},\]
and the statement is proved.

Suppose now that $n_0\ge \ell(x,
y)+\frac{2\alpha\Delta-\log C}{\alpha-\beta}$, which is equivalent to
\begin{equation}
\label{eq:n0}
(\alpha-\beta)n_0-(\alpha-\beta)\ell(x, y)-2\alpha\Delta+\log C\ge 0.
\end{equation}

If $t=1$, then $n_0-2\Delta\le\ell(x, y)<n_0$, hence
\[n_0\le\ell(x, y)+2\Delta=\ell(x,
y)+\frac{2\alpha\Delta-2\beta\Delta}{\alpha-\beta}<
\ell(x, y)+\frac{2\alpha\Delta-\log C}{\alpha-\beta},\]
since $\log C<0<2\beta\Delta$. But this contradicts our assumption.

Therefore $t>1$, and the inductive assumption implies
\[\sum_{i=1}^me^{-\beta\ell(x_{i-1}, x_i)}\ge
\sum_{i=0}^{t-1}C_2e^{-\beta\ell(y_i, y_{i+1})}>tC_2e^{-\beta n_0}.\]

We have $t\ge d_{n_0-2\Delta}(x, y)\ge Ce^{\alpha(n_0-2\Delta-\ell(x,
  y))}$, hence
\begin{multline*}\sum_{i=1}^me^{-\beta\ell(x_{i-1}, x_i)}\ge CC_2e^{-\beta
n_0+\alpha n_0-2\alpha\Delta-\alpha\ell(x, y)}=\\
C_2\exp\left(\log C-\beta
n_0+\alpha n_0-2\alpha\Delta-\alpha\ell(x, y)\right)=\\
C_2\exp\left(-\beta\ell(x, y)+(\alpha-\beta)n_0-(\alpha-\beta)\ell(x,
  y)-2\alpha\Delta+\log C\right)\ge C_2e^{-\beta\ell(x, y)},
\end{multline*}
by~\eqref{eq:n0}.
\end{proof}

\subsection{Lower exponents of Smale spaces}

\begin{defi}
\label{def:lowerexppartial}
Let $(\X, f)$ be a Smale space.
A number $\alpha>0$ is a \emph{stable (resp.\ unstable) lower exponent} of the Smale
space if there exists a constant $C>0$ such that for any stable
(resp.\ unstable) leaf $W$ and any
$x, y\in W$ we have
\[d_n(x, y)\ge Ce^{\alpha(n-\ell_W(x, y))}\]
for the internal log-scale on $W$. The supremum of the stable (resp.\ unstable) lower
exponents is called the \emph{stable} (resp.\
\emph{unstable}) \emph{lower critical exponent}.
\end{defi}

Note that by Proposition~\ref{pr:lowerexp} lower stable and
unstable exponents exist and are positive for any Smale
space. Proposition~\ref{pr:bLipschlexp} implies that the lower
critical exponents of a Smale space depend only on the topological
conjugacy class of the Smale space.

\begin{proposition}
\label{pr:lowerexpchar}
Let $l\in\R$. A number $\alpha>0$ is a lower stable
(resp.\ unstable) exponent of $(\X, f)$
if and only if there exists $C_l>0$ such that for every stable (resp.\
unstable) leaf $W$
and for every two points $x, y\in W$ such that
$\ell_W(x, y)\le l$ we have
\begin{equation}\label{eq:lowerexpchar}
d_n(x, y)\ge C_le^{\alpha n}
\end{equation}
for all $n$.
\end{proposition}

\begin{proof}
Let us assume that $W$ is a stable leaf (the proof for an
unstable leaf is the same).
If $\alpha$ is a lower exponent and $C$ is as in
Definition~\ref{def:lowerexppartial},
then for any $x, y\in W$ such that
$\ell_W(x, y)\le l$ we have
\[d_n(x, y)\ge Ce^{\alpha(n-\ell_W(x, y))}\ge Ce^{-\alpha
  l}\cdot e^{\alpha n},\]
and we can take $C_l=Ce^{-\alpha l}$.

Suppose now that $C_l>0$ is such that $d_n(x, y)\ge C_le^{\alpha n}$
for all $x, y$ belonging to one stable leaf $W$ and such
that $\ell_W(x, y)\ge l$.

Let $x$ and $y$ be arbitrary stably equivalent points of $\X$. Let
$W_0$ be their stable leaf. Denote $n_0=\ell_{W_0}(x, y)$.
Then $\ell_W(f^{l-n_0}(x), f^{l-n_0}(y))=l$, where $W$ is
the stable leaf of $f^{l-n_0}(x)\sim_+f^{l-n_0}(y)$.
Consequently,
\[d_n(f^{l-n_0}(x), f^{l-n_0}(y))\ge C_le^{\alpha n}\]
for all $n$.

The map $z\mapsto f^{l-n_0}(z)$ transforms every path in
$\Gamma_n(W_1)$ to a path in $\Gamma_{n+l-n_0}(W_2)$, where $W_1$ and
$W_2$ are the stable leaves of $z$ and $f^{l-n_0}(z)$, see~\eqref{eq:ellplusf}.

It follows that
\[
d_n(x, y)\ge d_{n+l-n_0}(f^{l-n_0}(x), f^{l-n_0}(y))\ge
 C_le^{\alpha(n+l-n_0)}=C_le^l\cdot e^{\alpha(n-\ell_{W_0}(x, y))},
\]
which shows that $\alpha$ is a lower exponent.
\end{proof}

\subsection{Metric properties of leaves}

Let $(\X, f)$ be a Smale space, and let $\ell, \ell_+$, and $\ell_-$
be the standard log-scale on $\X$, and the internal log-scales on the
stable and unstable leaves of $\X$. Let $U$ be a neighborhood of the
diagonal satisfying the conditions of Definition~\ref{def:expansive}.

The following theorem describes the classical theory of Bowen-Margulis
measure on Smale spaces, see~\cite{bowen:periodicmeasures}. See its exposition
in~\cite{nek:psmeasure}, which is notationally close to our paper.

Denote by $d$, $d_+$, and
$d_-$ metrics associated with the log-scales $\ell$, $\ell_+$, and
$\ell_-$, respectively. Denote by $B_*(r, x)$ the ball of radius $r$,
with respect to the metric $d_*$, with center in $x$, where $*\in\{+, -\}$.

\begin{theorem}
\label{th:SRBmeasures} Suppose that $(\X, f)$ is mixing.
There exists a number $\eta>0$ (called the \emph{entropy} of $(\X,
f)$), and  a family of Radon measures $\mu_+$ and $\mu_-$ on the stable
and unstable leaves of $\X$ satisfying the following properties.
\begin{enumerate}
\item There exists a number $C>1$ such that
 \[C^{-1}r^{\eta/\alpha_*}\le\mu_*(B_*(r, x))\le
 Cr^{\eta/\alpha_*}\]
  for all $r\ge 0$, $x\in\X$, and $*\in\{+, -\}$, where $\alpha_*$ is the exponent of the metric $d_*$.
\item The measures are preserved under holonomies.
\item The measures are quasi-invariant with respect to $f$, and
  $\frac{df_*(\mu_+)}{d\mu_+}=e^\eta$, $\frac{df_*(\mu_-)}{d\mu_-}=e^{-\eta}$.
\end{enumerate}
\end{theorem}

It follows from condition (1) of the theorem, that $\mu_+$ and $\mu_-$
are equivalent to the Hausdorff measures of the metrics $d_+$ and
$d_-$ of dimension $\frac{\eta}{\alpha_+}$ and
$\frac{\eta}{\alpha_-}$, respectively.

\section{Locally connected Smale spaces and upper exponents}
\label{s:upperexponents}

\subsection{Connectivity}
The aim of this section is to prove the following description of
locally connected Smale spaces. 

\begin{theorem}
\label{th:connectedness}
Let $(\X, f)$ be a Smale space. The following conditions are
equivalent.
\begin{enumerate}
\item The space $\X$ is locally connected.
\item All stable and unstable leaves are locally connected.
\item All stable and unstable leaves are connected.
\item All stable and unstable leaves are locally path connected.
\item All stable and unstable leaves are path connected.
\item The graphs $\Gamma_0(W)$ are connected for every (stable or unstable)
  leaf $W$.
\item The graphs $\Gamma_n(W)$ are connected for every leaf $W$ and every $n$.
\item There exist $\alpha>0$ and $C>0$ such that for every
  leaf $W$ we have
\[d_n(x, y)\le Ce^{\alpha(n-\ell_W(x, y))}\]
for all $x, y\in W$ and all $n\ge \ell_W(x, y)$.
\end{enumerate}
\end{theorem}

Recall that for a stable or unstable
leaf $W$, we denote by $\Gamma_n(W)$ the graph with the set of
vertices $W$ in which two vertices $x, y$ are connected by an edge if and
only if $\ell_W(x, y)\ge n$, where $\ell_W$ is the corresponding
($\ell_+$ or $\ell_-$) internal log-scale on $W$.

Let us start by proving equivalence of conditions (1) and (2).

\begin{proposition}
\label{pr:oneandtwo}
Let $(\X, f)$ be a Smale space. The space $\X$ is locally connected if
and only if each leaf is locally connected.
\end{proposition}

\begin{proof}
Each point $x\in\X$ has a neighborhood
homeomorphic to the direct product of the neighborhoods of $x$ in the
corresponding stable and unstable leaves. It follows that if $x$
has bases of connected neighborhoods in the leaves, then $x$ has a
basis of connected neighborhoods in $\X$.

In the other direction, if $x$ has a basis of connected neighborhoods
in $\X$, then for any rectangular neighborhood $R$ of $x$ there
exists a connected neighborhood $U\subset R$ of $x$. Its projection
onto the direct factors of $R$ will be connected, hence the point $x$
has bases of connected neighborhoods in its leaves.
\end{proof}

\begin{proposition}
\label{pr:locthencon}
If every stable leaf of $(\X, f)$ is locally connected, then every
stable leaf of $(\X, f)$ is connected.
\end{proposition}

\begin{proof}
For every point $x\in\X$ there exists a connected neighborhood $U$ of
$x$ in its stable leaf and a rectangular neighborhood $R$ of $x$ in
$\X$ such that $P_+(R, x)=U$. Then each plaque of $R$ will be
homeomorphic to $U$, hence will be connected. It follows that every
point of $\X$ has a rectangular neighborhood $R$ such that all its
stable plaques are connected. Since $\X$ is compact,
there exists a finite covering $\mathcal{R}=\{R_i\}$ of $\X$ by open rectangles
with connected stable plaques.

Let $W$ be a stable leaf, and let $x, y\in W$. By Lebesgue's covering
lemma, there exists $n$ such that $f^n(x)$ and $f^n(y)$ belong to
one plaque $V$ of a rectangle $R_i\in\mathcal{R}$. Then $f^{-n}(V)$ is a
connected subset of $W$ containing $x$ and $y$. We have shown that any
two points of $W$ belong to one connected component of $W$, i.e., that
$W$ is connected.
\end{proof}

Let $\mathcal{R}=\{R_i\}_{i\in I}$ be a finite covering of $\X$ by
open rectangles. Let $\mathcal{R}_+$ be the
set of all stable plaques of elements of $\mathcal{R}$.

Every stable leaf $W$ is a union $\bigcup_{T\in\mathcal{R}_+, T\subset
W} T$ of stable plaques
contained in $W$. Each plaque is an open subset of $W$ and the
internal topology on $W$ coincides with the direct limit topology of the union
of the plaques.

Denote by $\Gamma'_n(W)$ the graph with the set of vertices $W$ in which
two vertices are connected by an edge if they belong to one set of the
form $f^n(T)$, $T\in\mathcal{R}_+$.

The map $f:W\arr f(W)$ induces isomorphisms
$\Gamma_n(W)\arr\Gamma_{n+1}(f(W))$
and $\Gamma_n'(W)\arr\Gamma_{n+1}'(f(W))$, see~\eqref{eq:ellplusf}
and~\eqref{eq:ellminusf} for the first isomorphism.

\begin{lemma}
\label{lem:inclusiongn}
There exists a number $k_0$ such that if $x, y\in W$ are adjacent in
$\Gamma_n(W)$, then they are adjacent in $\Gamma_{n-k_0}'(W)$, and if $x, y\in
W$ are adjacent in $\Gamma_n'(W)$, then they are adjacent in $\Gamma_{n-k_0}(W)$.
\end{lemma}

\begin{proof}
There exists $k_1$ such that for every plaque $V\in\mathcal{R}_+$ and
every pair $x, y\in V$ we have $(f^k(x), f^k(y))\in U$ for all $k\ge
k_1$ (where $U$ is a neighborhood of the diagonal defining $\ell,
\ell_+$, and $\ell_-$). If $x, y\in W$ are connected by an edge in
$\Gamma_n'(W)$, then $x, y\in f^n(T)$ for $T\in\mathcal{R}_+$, hence
$(f^{k-n}(x), f^{k-n}(y))\in U$ for all $k\ge k_1$, hence $\ell_+(x,
y)\ge n-k_1$, i.e., $x$ and $y$ are connected by an edge in
$\Gamma_{n-k_1}(W)$.

By Lebesgue's covering lemma, there exists
$k_2$ such that if $x, y\in W$ are such that $\ell_W(x,
y)\ge k_2$, then $x$ and $y$ belong to one plaque $V\in\mathcal{R}_+$.
Then every edge of $\Gamma_{k_2}(W)$ is an edge in
$\Gamma_0'(W)$. Consequently, every edge of $\Gamma_n(W)$ is an edge in $\Gamma_{n-k_2}(W)$.
\end{proof}

\begin{proposition}
\label{pr:primenonprime}
The following conditions are equivalent.
\begin{enumerate}
\item The graph $\Gamma_0(W)$ is connected for every stable leaf $W$.
\item The graph $\Gamma_n(W)$ is connected for every stable leaf $W$ and
  for every $n\in\Z$.
\item The graph $\Gamma_0'(W)$ is connected for every stable leaf $W$.
\item The graph $\Gamma_n'(W)$ is connected for every stable leaf $W$ and
  for every $n\in\Z$.
\end{enumerate}
\end{proposition}

\begin{proof}
Since the map $f^k:W\arr f^k(W)$ induces isomorphisms
$\Gamma_n(W)\arr\Gamma_{n+k}(f^k(W))$ and $\Gamma_n'(W)\arr\Gamma_{n+k}'(f^k(W))$, (1) is
equivalent to (2), and (3) is equivalent to (4).

Let $k_0$ be as in Lemma~\ref{lem:inclusiongn}. If all graphs
$\Gamma_n(W)$ are connected, then all graphs
$\Gamma_{n-k_0}'(W)\supseteq\Gamma_n(W)$ are connected. If all graphs
$\Gamma_n'(W)$ are connected, then all graphs
$\Gamma_{n-k_0}(W)\supseteq\Gamma_n'(W)$ are connected. This shows that all
conditions (1)--(4) are equivalent to each other.
\end{proof}

\begin{proposition}
\label{pr:congraphcon}
If a stable leaf $W$ is connected, then the graph $\Gamma_n'(W)$ is
connected for every $n$.
\end{proposition}

\begin{proof}
Suppose that $W$ is a connected stable leaf.
Let $A$ be a connected component of $\Gamma_0'(W)$. Let $W_A$ be the union of the
plaques $V\in\mathcal{R}_+$ containing vertices of $A$. It follows
from the definition of the graph $\Gamma_0'(W)$ that every plaque
$V\in\mathcal{R}_+$ is either contained in $W_A$, or is disjoint with
it. Consequently, $W_A$ is clopen, which implies that $W=W_A$, hence
$A=W$, and $\Gamma_0'(W)$ is connected.
\end{proof}

\begin{proposition}
\label{pr:A}
Suppose that the graphs $\Gamma_n'(W)$ are connected for all stable
leaves $W$ and all $n$. Then there exists
$A\ge 1$ such that any two adjacent vertices in $\Gamma_0'(W)$ are on
distance at most $A$ in $\Gamma_1'(W)$.
\end{proposition}

\begin{proof}
Let $E$ be the closure in $\X\times\X$ of the set pairs of points $(x, y)$ such that
there exists a plaque $V\in\mathcal{R}_+$ such that $x, y\in V$. It is
easy to see that $E$ is compact. It contains the set of edges of every
graph $\Gamma_0'(W)$, and is contained in the stable equivalence relation.

Any pair of points $x, y$ such that $(x, y)\in E$ is connected
by a path in $\Gamma_1'(W)$, as all graphs $\Gamma_n'(W)$ are connected. It
means that there exists a sequence of rectangles $R_1, \ldots,
R_n\in\mathcal{R}$, and a sequence of points $x_i\in R_i$ such that $x$
and $x_1$ belong to the same stable plaque of $R_1$, $y$ and $x_n$
belong to the same stable plaque of $R_n$, and the stable plaque of
$x_i$ in $R_i$ intersects with the stable plaque of $x_{i+1}$ in
$R_{i+1}$. This sequence $R_1, \ldots, R_n$ will define a path in $\Gamma_1'(W)$
connecting any two points $(x', y')$ belonging to a neighborhood of
$(x, y)$ in $E$. It follows then from compactness of $E$ that we can
find a finite upper bound on the length of a path connecting any two
points of $E$, which finishes the proof.
\end{proof}

Recall that $d_n(x, y)$ is the distance in the graph $\Gamma_n(W)$.

\begin{proposition}
\label{pr:upperexp}
Suppose that all graphs $\Gamma_n'(W)$ are connected. Then there exist positive
constants $\alpha$ and $C$ such that for any two points $x, y\in W$ we have
\[d_n(x, y)\le Ce^{\alpha(n-\ell_W(x, y))}\]
for all $n\ge\ell_W(x, y)$.
\end{proposition}

\begin{proof}
Let $A$ and $k_0$ be as in Propositions~\ref{pr:A}
and~\ref{lem:inclusiongn}, and let $x, y\in W$ be arbitrary. Denote
$n_0=\ell_W(x, Y)$. The
points $x$ and $y$ are connected by an edge in $\Gamma_{n_0}(W)$,
hence they are connected by an edge in $\Gamma_{n_0-k_0}'(W)$. It follows from
Proposition~\ref{pr:A} that for every $k\ge 0$,
distance between $x$ and $y$ in $\Gamma_{n_0-k_0+k}'(W)$ is not greater
than $A^k$. As the set of edges of $\Gamma_{n_0-k_0+k}'(W)$ is contained
in the set of edges of $\Gamma_{n_0-2k_0+k}(W)$, we have
\[d_{n_0-2k_0+k}(x, y)\le A^k.\]
for all $k\ge 0$. Consequently,
\[d_n(x, y)\le A^{n-n_0+2k_0}=A^{2k_0}\cdot A^{n-\ell_W(x, y)}\]
for all $n\ge\ell_W(x, y)-2k_0$.
\end{proof}

\begin{proposition}
\label{pr:pathcon}
Suppose that all graphs $\Gamma_n'(W)$ are connected. Let $d$ be a metric
on $W$ associated with $\ell_W$. There exists a
constant $C$ such that for any two
points $x, y\in W$ there exists a curve $\gamma:[0, 1]\arr W$
connecting $x$ to $y$ and such that the diameter the range of $\gamma$
is not larger than $C d(x, y)$.
\end{proposition}

\begin{proof}
Let $A$ and $k_0$ be as in Propositions~\ref{pr:A}
and~\ref{lem:inclusiongn}. Let $C_1>1$ and $\alpha>0$ be such that
$C_1^{-1}e^{-\alpha\ell_W(x, y)}\le d(x, y)\le C_1e^{-\alpha\ell_W(x,
  y)}$ for all $x, y\in W$.

Take arbitrary $x, y\in W$. Let $n_0=\ell_W(x, y)$. Then $x$ and $y$
are adjacent in $\Gamma_{n_0-k_0}'(W)$, hence they are connected by a path $\gamma_1:
\{x=x_{1,1}, x_{1, 2}, \ldots, x_{1, m_1}=y\}$ of length at most $A$ in
$\Gamma_{n_0-k_0+1}'(W)$. Each pair of points $x_{1, i}, x_{1, i+1}$ is
connected by a path of length at most $A$ in $\Gamma_{n_0-k_0+2}'(W)$. We
get then a path $\gamma_2=\{x=x_{2, 1}, x_{2, 2}, \ldots, x_{1,
  m_2}\}$ of length at most $A^2$ in $\Gamma_{n_0-k_0+2}'(W)$, containing
$\gamma_1$. We get then inductively defined sequence of paths
$\gamma_n=\{x=x_{n, 1}, x_{n, 2}, \ldots, x_{n, m_n}=y\}$ in
$\Gamma_{n_0-k_0+n}'(W)$ such that each next path $\gamma_n$
is obtained from $\gamma_{n-1}$ by inserting at most $A-1$ points
between each pair of neighbors of $\gamma_n$.

Every pair of points $x_{n, j}$, $x_{n, j+1}$ is adjacent in
$\Gamma_{n_0-k_0+n}'(W)$, hence $\ell_W(x_{n, j}, x_{n, j+1})\ge
n_0-2k_0+n$, hence $d(x_{n, j}, x_{n, j+1})\le
C_1e^{-\alpha(n+n_0-2k_0)}$. In particular, for every point $t_1$ of
$\gamma_{n+1}$ there exists a point $t_2$ of $\gamma_n$ such that
$d(t_1, t_2)<AC_1e^{-\alpha(n+1+n_0-2k_0)}$.

It follows that the diameter of the set $\gamma_n$ is not greater than
\begin{multline*}
C_1e^{-\alpha(n_0-2k_0)}+2\sum_{i=1}^nAC_1e^{-\alpha(i+n_0-2k_0)}<\\
2C_1Ae^{-\alpha(n_0-2k_0)}\cdot\frac{1}{1-e^{-\alpha}}=\\
\frac{2C_1Ae^{2\alpha k_0}}{1-e^{-\alpha}}\cdot e^{-\alpha\ell_W(x,
  y)}\le \frac{2C_1^2Ae^{2\alpha k_0}}{1-e^{-\alpha}}\cdot d(x, y).
\end{multline*}

Since $d(x_{n, j}, x_{n, j+1})\le C_1e^{-\alpha(n_0-2k_0+n)}$, the
closure of $\bigcup_{n=1}^\infty\gamma_n$ is the image of a continuous
curve connecting $x$ to $y$. Diameter of the image of the curve is not greater
than $\frac{2C_1^2Ae^{2\alpha k_0}}{1-e^{-\alpha}}\cdot d(x, y)$.
\end{proof}

Let us summarize now the proof of
Theorem~\ref{th:connectedness}. The equivalence of (1) and (2) is shown in
Proposition~\ref{pr:oneandtwo}. The implication (2)$\Rightarrow$(3) is
given in Proposition~\ref{pr:locthencon}. The equivalence of (6) and (7)
is contained in
Proposition~\ref{pr:primenonprime}. Proposition~\ref{pr:congraphcon}
shows then that (3) implies (6). Proposition~\ref{pr:upperexp} proves
that (7) implies (8). Condition (8) obviously implies
(7). Proposition~\ref{pr:pathcon} shows that (7) implies path connectivity
and local path connectivity of the leaves, i.e., that (7) implies (4) and (5). The implications (3)$\Rightarrow$(1), (5)$\Rightarrow$(3), and (4)$\Rightarrow$(2) are obvious.
This finishes the proof of Theorem~\ref{th:connectedness}.

\subsection{Local product structure on locally connected Smale spaces}

\begin{proposition}
\label{pr:uniquerectangle}
Let $R$ be a sub-rectangle of a Smale space $(\X, f)$. If $R$ is
connected and locally connected, then the direct product structure on $R$
compatible with the local product structure on $\X$ is unique.
\end{proposition}

\begin{proof}
Suppose that, on the contrary, there exist two different direct
product structures $[\cdot, \cdot]_1$ and $[\cdot, \cdot]_2$, both
compatible with the local product structure on $\X$.

By Definition~\ref{def:locprod}, there exists a covering
$\mathcal{R}$ of $\X$ by open rectangles
such that for any $U\in\mathcal{R}$ and $x, y\in R\cap U$ we have $[x,
y]_U=[x, y]_1=[x, y]_2$. Then for every
$U\in\mathcal{R}$ the intersection $U\cap R$ is a (possibly empty)
sub-rectangle of $R$ with respect to both direct product structures;
and restrictions of the direct products structures $[\cdot, \cdot]_i$,
$i=1, 2$, onto $U\cap R$ coincide.

Since $R$ is connected, all plaques of $R$ (with respect to both direct
products structures) are connected. Let $P_+$ be a stable plaque of
$(R, [\cdot, \cdot]_1)$. Let $x, y\in P_+$. Since $P_+$ is connected, there exists a
sequence of points $x_0=x, x_1, \ldots, x_n=y$ and a sequence of
rectangles $U_0, U_1, \ldots, U_n\in\mathcal{R}$ such that $x_i\in
U_i$, and $U_i\cap U_{i+1}\cap P_+\ne\emptyset$. The set $U_i\cap P_+$ is a
plaque of the rectangle $U_i\cap R$, hence it is a subset of the stable
plaque of $(R, [\cdot, \cdot]_2)$. We get a sequence $U_i\cap P_+$ of subsets of
plaques of $(R, [\cdot, \cdot]_2)$ such that $(U_i\cap P_+)\cap
(U_{i+1}\cap P_+)\ne\emptyset$. But it means that $U_i\cap P_+$ are
subsets of one plaque of $(R, [\cdot, \cdot]_2)$. We have shown that
if two points belong to one stable plaque of $(R, [\cdot, \cdot]_1)$,
then they belong to one stable plaque of $(R, [\cdot, \cdot]_2)$. The
converse is proved in the same way. Consequently, the stable plaques of $R$
with respect to $[\cdot, \cdot]_1$ are the same as the stable plaques
of $R$ with respect to $[\cdot, \cdot]_2$. The same statement is
obviously true for the unstable plaques, which implies that the direct
product structures $[\cdot, \cdot]_1$ and $[\cdot, \cdot]_2$ on $R$ coincide.
\end{proof}

\subsection{Upper exponents}
\begin{defi}
\label{def:upperexponent}
A positive number $\alpha>0$ is a  \emph{stable (resp.\ unstable)
  upper exponent} of the Smale space
if there exists $C>0$ such that for any stable (resp.\ unstable) leaf
$W$ and every pair of points
$x, y\in W$ we have
\[d_n(x, y)\le Ce^{\alpha(n-\ell_W(x, y))}\]
for all $n\ge\ell_W(x, y)$.
\end{defi}

Note that changing $\ell_W$ to a bi-Lipschitz equivalent log-scale,
one does not change the set of upper exponents, i.e., this notion is well defined
and depends only on the topological conjugacy class of the Smale space
(see Lemma~\ref{lem:standlogscale}).

By Theorem~\ref{th:connectedness} a finite upper exponent
exists if $\X$ is locally connected.

The proof of the next proposition is analogous to the proof of
Proposition~\ref{pr:lowerexpchar}.

\begin{proposition}
\label{pr:upperexpchar}
Fix $l\in\R$. A number $\alpha>0$ is a stable (resp.\ unstable)
upper exponent if and only if there exists a
constant $C_l>0$ such that for any stable (resp.\ unstable) leaf $W$
and any $x, y\in W$ such that $\ell_W(x, y)\ge l$ we have
\[d_n(x, y)\le C_le^{\alpha n}\]
for all $n\ge 0$.
\end{proposition}

\section{Splittings of Smale spaces}
\label{s:splitting}

\subsection{Groups of deck transformations}

\begin{defi}
Let $(\X, f)$ be a Smale space. A \emph{splitting} of $(\X, f)$ is a
covering map $\pi:\splitting\arr\X$, where $\splitting$ is a space with a (global) direct
product structure, such that
\begin{enumerate}
\item $\pi$ preserves the local product
structures on $\splitting$ and $\X$, see Definition~\ref{def:preserveslocprod};
\item restriction of $\pi$ onto every plaque $P_1(\splitting, x)$ of $\splitting$ is a homeomorphism
  with the stable leaf $W_+(\pi(x))$, and restriction of $\pi$
  onto every plaque $P_2(\splitting, x)$ of $\splitting$ is a homeomorphism with the
  unstable leaf $W_-(\pi(x))$, with respect to their intrinsic topology.
\end{enumerate}
\end{defi}

\begin{proposition}
\label{pr:splittingmixing}
If there exists a splitting of a Smale space $(\X, f)$, then the Smale
space is irreducible and the set of non-wandering points of $(\X, f)$
is equal to $\X$. If, in addition, $\X$ is connected, then $(\X, f)$
is mixing.
\end{proposition}

\begin{proof}
The proof is the same as, for example, the proof of Theorem~5 in~\cite{brin:asterisque}. Suppose that $\pi:\splitting\arr\X$ is a splitting. Let $\X_1,
\X_2\subset\X$ be irreducible components of $(\X, f)$, see
Theorem~\ref{th:spectdecomp}. Take $x_1, x_2\in\splitting$ such that
$\pi(x_1)\in\X_1$ and $\pi(x_2)\in\X_2$. Consider the point $[x_1,
x_2]\in\splitting$. Then $\pi([x_1, x_2])\in W_+(\pi(x_1))\cap
W_-(\pi(x_2))$. It follows that the set of the accumulation points of
$f^n(\pi([x_1, x_2]))$ for $n\ge 0$ belongs to $\X_1$, and the set of
the accumulation points of $f^n(\pi([x_1, x_2]))$ for $n\le 0$ belongs
to $\X_2$, hence $\X_2\prec\X_1$. But we will also have
$\X_1\prec\X_2$ by considering $[x_2, x_1]$, which implies
$\X_1=\X_2$, by Theorem~\ref{th:spectdecomp}. Consequently, $(\X, f)$
has no wandering points and is irredusible. The rest of the
proposition follows from Theorem~\ref{th:mixing}.
\end{proof}

\begin{proposition}
\label{pr:uniqesplitting}
Suppose that $\X$ is connected and locally connected.
Let $\pi_1:\splitting_1\arr\X$ and $\pi_2:\splitting_2\arr\X$ be splittings of $(\X,
f)$. If $x_1\in\splitting_1$ and $x_2\in\splitting_2$ are such that
$\pi_1(x_1)=\pi_2(x_2)$, then there exists a unique homeomorphism
$F:\splitting_1\arr\splitting_2$ preserving the local product structures such that
$\pi_1=\pi_2\circ F$ and $F(x_1)=x_2$.
\end{proposition}

\begin{proof}
Since the leaves of locally connected Smale spaces are connected and
locally connected, the spaces $\splitting_i$ are connected and locally connected.

Denote $x=\pi_1(x_1)=\pi_2(x_2)$.
Restriction of $\pi_i$ onto the plaques $P_1(\splitting_i, x_i)$ and $P_2(\splitting_i,
x_i)$ are homeomorphisms with the leaves $W_+(x)$ and $W_-(x)$,
respectively. Therefore, the only possible way to define $F$ is by the
equality
\[F([y_1, y_2]_{\splitting_1})=[z_1, z_2]_{\splitting_2},\]
where $y_1\in P_1(\splitting_1, x_1)$, $y_2\in P_2(\splitting_1, x_1)$ are arbitrary,
while $z_1\in P_1(\splitting_2, x_2)$, $z_2\in P_2(\splitting_2, x_2)$ are uniquely
determined by the condition $\pi_1(y_1)=\pi_2(z_1)$ and
$\pi_1(y_2)=\pi_2(z_2)$.

The defined map $F$ is a homeomorphism, since it is a direct product
of two homeomorphisms. Consequently,
$\pi_2\circ F:\splitting_1\arr\X$ is a covering map. Since $F$ and $\pi_2$
preserve the local product structures of $\splitting_i$ and $\X$, their
composition $\pi_2\circ F$ preserves the local product
structures, i.e., the image of the direct product structure on $\splitting_1$
by $F$ defines the same local product structure on $\splitting_2$ as the direct product
structure $[\cdot, \cdot]_{\splitting_2}$. By the same arguments as in the
proof of Proposition~\ref{pr:uniquerectangle}, the direct
product structure on $\splitting_2$ is uniquely determined by the corresponding local product
structure. Consequently, $F$ preserves the direct product structures,
i.e., $F([y_1, y_2]_{\splitting_1})=[F(y_1), F(y_2)]_{\splitting_2}$ for all $y_1,
y_2\in \splitting_1$. It follows that if $\pi_1(y)=\pi_2(F(y))$ for $y\in \splitting_1$,
then $\pi_1=\pi_2\circ F$ on a rectangular neighborhood of
$y$. Consequently, the set of points $y\in \splitting_1$ such that
$\pi_1(y)=\pi_2(F(y))$ is open and closed, it contains $x$, hence it
is equal to $\splitting_1$.
\end{proof}

We assume now that $(\X, f)$ is a locally connected and connected
Smale space.

Let $\pi:\splitting\arr\X$ be a splitting. Let $G$ be the set of all
homeomorphisms $g:\splitting\arr\splitting$ preserving the direct product structure on
$\splitting$ and such that $\pi=\pi\circ F$. Then $G$ is obviously a group. By
Proposition~\ref{pr:uniqesplitting} (for the case $\pi_1=\pi_2$),
the action of the group on $\splitting$ is
free and transitive on $\pi^{-1}(x)$ for every $x\in\X$. We call $G$
the \emph{group of deck transformations} of the splitting. It acts
properly on $\splitting$, since every point of $\splitting$ has a neighborhood $U$ such
that $\pi(U)$ is evenly covered, i.e., such that $g(U)\cap U=\emptyset$ for all
not-trivial $g\in G$.

Note that if $\pi:\splitting\arr\X$ is a splitting,
then $\pi\circ f$ is also a splitting. Choose $x_1,
x_2\in\splitting$ such that $f\circ\pi(x_1)=\pi(x_2)$, and apply
Proposition~\ref{pr:uniqesplitting} to $\pi_1=f\circ\pi$ and
$\pi_2=\pi$. We get that there exists a unique homeomorphism $F:\splitting\arr
\splitting$ preserving the direct product structure such that $F(x_1)=x_2$ and
$\pi\circ F=f\circ\pi$.

\begin{defi}
Let $\pi:\splitting\arr\X$ be a splitting.
We say that a homeomorphism $F:\splitting\arr\splitting$ preserving the direct product
structure is a \emph{lift of $f$} if $\pi\circ F=f\circ\pi$.
\end{defi}

It follows
from the above arguments that lifts of $f$ exist, and if $F_1$ and
$F_2$ are two lifts, then $F_1^{-1}F_2$ and $F_1F_2^{-1}$ belong to $G$.

It also follows that for every $g\in G$, and every lift $F$ of $f$, we
have $F^{-1}gF\in G$. The map $g\mapsto F^{-1}gF$ is an automorphism
of $G$. We say that it is \emph{induced by $f$}. Any two automorphisms
induced by $f$ on $G$ differ from each other by an inner automorphism.

\begin{proposition}
\label{pr:finitegeneration}
The group of deck transformations of a splitting of a connected and locally
connected Smale space is finitely generated.
\end{proposition}

\begin{proof}
Let $\pi:\splitting\arr\X$ be a splitting. Consider a finite covering $\mathcal{U}$ of
$\X$ by open connected evenly covered by $\pi$ subsets. Let
$\mathcal{U}_\splitting$ be the union of the sets of connected components of
$\pi^{-1}(U)$ for all $U\in\mathcal{U}$. Then $\mathcal{U}_\splitting$ is a
covering of $\splitting$. Since $\splitting$ is connected, every two elements
$U, V\in\mathcal{U}_\splitting$ are connected with each other by a chain of elements
$U_0=U, U_1, \ldots, U_n=V$ such that $U_i\cap
U_{i+1}\ne\emptyset$. It follows that there exists a connected finite
union $V$ of elements of $\mathcal{U}_\splitting$ such that $\pi(V)=\X$. Note
that $\overline V$ is compact.

Let $S$ be the set of elements $g\in G$ such that $g(V)\cap
V\ne\emptyset$. It is finite, since the action of $G$ on $\splitting$ is proper.

Let $g\in G$ and $x_0\in V$. Since $\splitting$ is connected and
$\splitting=\bigcup_{g\in G}g(V)$, there exists a sequence $g_i\in G$, $i=0, 1,
\ldots, k$ such that $g_0=1$, $g_k=g$, and $g_i(V)\cap
g_{i+1}(V)\ne\emptyset$ for all $i=0, 1, \ldots, k-1$. Note that
$g_i(V)\cap g_{i+1}(V)\ne\emptyset$ implies $V\cap
g_i^{-1}g_{i+1}(V)\ne\emptyset$, hence $g_i^{-1}g_{i+1}\in S$. We see
that $g=g_0^{-1}g_1\cdot g_1^{-1}g_2\cdots g_{k-1}^{-1}g_k$ is a
product of $k$ elements of $S$.
\end{proof}

\subsection{Splittable Smale spaces and hyperbolic graphs}
\label{ss:hypgraphs}
A connection between Smale spaces and Gromov hyperbolic graphs
described in this subsection is a particular case of the theory of
Cayley graphs of hyperbolic groupoids, described
in~\cite{nek:hyperbolic}. Since the theory for Smale spaces is simpler than
the general case, and in order to make our paper more self-contained,
we describe them directly.

Let $(\X, f)$ be a Smale space with locally connected and connected
space $\X$, and let $\pi:\splitting\arr\X$ be a
splitting. We denote by $[\cdot, \cdot]$ the local product structures
on $\splitting$ and $\X$. Let $G$ be the group of deck transformations of the
splitting. Let $F:\splitting\arr\splitting$ be a lift of $f$.

Let $d$ be a metric on $\X$ associated with the standard log-scale
$\ell$. We will denote by $d_+$ and $d_-$ metrics on the stable and
unstable leaves of $\X$ associated with the respective standard
log-scales $\ell_+$ and $\ell_-$. We assume that exponents of the
metrics $d$, $d_+$, and $d_-$ are equal (by taking them sufficiently small). Then we have the following
corollary of~\cite[Lemma~7.8]{nek:hyperbolic}, see
also~\cite{fried:naturalmetric}.

\begin{proposition}
\label{pr:agreeslocprod}
There
exist constants $\epsilon, L>1$ such that for every point $x\in\X$ and
every rectangle $R$ contained in the $\epsilon$-neighborhood of $x$ the
restriction of $d$ to $R$ is $L$-bi-Lipschitz equivalent to the
metric given by \[d_x(y_1, y_2)=d_+([y_1, x], [y_2, x])+d_-([x,
  y_1], [x, y_2]).\]
\end{proposition}

Let $\epsilon>0$ be such that it satisfies the conditions of Proposition~\ref{pr:agreeslocprod} and for every $x\in\X$ the
$\epsilon$-neighborhood of $x$ is evenly covered by $\pi$.

Define then $d_\splitting(x, y)$ as the infimum of the sum
$\sum_{i=0}^{m-1}d(\pi(x_i), \pi(x_{i+1}))$ over all sequences $x_0,
x_1, \ldots, x_m$ such that $x_0=x$, $x_m=y$, and $d(\pi(x_i),
\pi(x_{i+1}))<\epsilon$ for all $i=0, 1, \ldots, m-1$.

Then $d_\splitting$ is a $G$-invariant metric on $\splitting$ such that $d_\splitting(x, y)=d(\pi(x), \pi(y))$
for all $x, y$ such that $d_\splitting(x, y)<\epsilon$.

The map $\pi$ bijectively identifies the plaques $W_+(x)$ and $W_-(x)$
of $\splitting$ with their images $W_+(\pi(x))$ and $W_-(\pi(x))$, respectively.
We obtain the metrics $d_+$ and $d_-$ on
$W_+(x)$ and $W_-(x)$, respectively.
Then Proposition~\ref{pr:agreeslocprod} holds when we replace $\X$ by $\splitting$ and $d$ by $d_\splitting$.

Let $R\subset\splitting$ be an open relatively compact
rectangle such that $\bigcup_{g\in G}g(R)=\splitting$.
Let $W_+$ be a stable plaque of $\splitting$. Denote by $\Omega_n(W_+,
R)$, for $n\in\Z$, the set of elements $g\in G$ such that of
$F^n(g(R))$ intersects $W_+$.

Denote by $\Xi(W_+, R)$ the graph with the set of vertices
$\bigcup_{n\in\Z}\Omega_n(W_+, R)\times\{n\}$ in which two vertices
are connected by an edge if and only if they are either of the form
$(g_1, n)$, $(g_2, n)$, where $g_1(R)\cap g_2(R)\ne\emptyset$, or of
the form $(g_1, n)$, $(g_2, n+1)$, where $g_1(R)\cap
F(g_2(R))\ne\emptyset$. In other words, we connect two vertices of
$\Xi(W_+, R)$ if and only if they belong to the same or neighboring
levels $\Omega_n(W_+, R)$ and the corresponding rectangles intersect.

\begin{theorem}
\label{th:hyperbolicgraphs}
The graph $\Xi=\Xi(W_+, R)$ is Gromov hyperbolic. There exists a point $\omega\in\partial\Xi$ such that all paths of the form $(g_n, -n)\in\Xi$, $n\ge 0$,
converge to $\omega$.
The correspondence mapping the limit in $\partial\Xi$ of a path $(g_n, n)$,
$n\ge 0$ to the limit of the intersections of $F^n(g_n(R))$ with $W_+$
induces a homeomorphism between $\partial\Xi\setminus\{\omega\}$
and $W_+$.
\end{theorem}

\begin{proof}
Consider two vertices $(g, 0)$ and $(h, 0)$. Let $m$
be the distance between them in $\Omega_0(W_+, R)$.

Let $(g_k, -k)$ and $(h_k, -k)$ for $k\ge 0$ be arbitrary paths in $\Xi$ such
that $g_0=g$ and $h_0=h$.

There exists $\epsilon>0$ such that for every $x\in\splitting$ there exist
$g\in G$ such that the ball of radius $\epsilon$
with center $x$ is contained in $g(R)$. It follows that for every
$m\in\N$ there exists
$N_m>0$ such that for any path $(a_0, 0), (a_1, 0), \ldots,
(a_m, 0)$ of length $m$ in $\Omega_0(W_+, R)$ there exists
$g\in G$ such that $F^{N_m}(a_0(R))\cup\cdots\cup
F^{N_m}(a_m(R))\cap W_+\subset g(R)\cap W_+$. It follows that $(g_{N_m},
-N_m)$ and $(h_{N_m}, -N_m)$ are on distance at most $2$ in $\Omega_{-N_m}(W_+, R)$.

It follows now from~\cite[Theorem~2.10]{nek:hyperbolic}
that the graph $\Xi$ is Gromov hyperbolic,
that the level function $(g(R), n)\mapsto n$ is a Busemann function of
a point $\omega\in\partial\Xi$, and that every path $(g_n, n)$
for $n\ge 0$ converges to a point of $\partial\Xi\setminus\omega$,
whereas every path of the form $(g_n, -n)$, $n\ge 0$, converges to
$\omega$.

\begin{lemma}
Every point of
$\partial\Xi\setminus\omega$ is the limit of a sequence of the form
$(g_n, n)$.
\end{lemma}

\begin{proof}
Note that since $(g, n)\mapsto n$ is a Busemann function associated
with $\omega\in\partial\Xi$, every point
$\xi\in\partial\Xi\setminus\omega$ is the limit of a sequence of
the form $(g_n, n)$ (which is not necessarily a path). For every $n\ge
0$ there exists a path $(h_{n, k}, k)$, $k\le n$, such that $h_{n,
  n}=g_n$. All these paths converge to $\omega$, and by the above
arguments, for any $n_1, n_2$, distance from $(h_{n_1, k}, k)$ to
$(h_{n_2, k}, k)$ is not more than 2 for all $k$ smaller than some
$k(n_1, n_2)$. Moreover, $k(n_1,
n_2)\to\infty$ as $n_1, n_2\to\infty$. It follows then by compactness
arguments (since $\Xi$ has bounded valency) that there exists a
path $(h_n, n)$, $n\ge 0$, converging to $\xi$.
\end{proof}

Let $(g_n, n)$, for $n\ge 0$, be a path in $\Xi$. The sets
$V_n=F^n(g_n(R))\cap W_+$ are compact, their diameters decrease
exponentially, and we have $V_n\cap V_{n+1}\ne\emptyset$ for every
$n$. It follows that the sequence $V_n$ converges (in the Hausdorff
metric) to a point $x\in W_+$. Let us show that the map
$\Lambda:\lim_{n\to\infty}(g_n, n)\mapsto\lim_{n\to\infty}F^n(g_n(R))$ is a
homeomorphism between $\partial\Xi\setminus\omega$ and $W_+$.

The arguments basically repeat the proof
of~\cite[Theorem~6.9]{nek:hyperbolic}.
Let us show that the map is well defined. If $(g_n, n)$ and $(h_n, n)$ converge to the same
limit in $\partial\Xi$, then the distance between $g_n$ and $h_n$
in $\Omega_n(W_+, R)$ is uniformly bounded. But this implies that the
Hausdorff distance between $F^n(g_n(R))$ and $F^n(h_n(R))$ is
exponentially decreasing, hence
$\lim_{n\to\infty}F^n(g_n(R))=\lim_{n\to\infty}F^n(h_n(R))$. The same
argument shows that the map $\Lambda$ is continuous, since if $\xi_1$ and
$\xi_2$ are close to each other, then the sequences $(g_n, n)$ and
$(h_n, n)$ are close to each other for an initial interval $n=0,
\ldots, L$, where $L$ is big. But then the limits $\lim_{n\to\infty}F^n(g_n(R))$ and
$\lim_{n\to\infty}F^n(h_n(R))$ are close to each other.

The map $\Lambda$ is onto, since for every point $x\in W_+$ there
exists a path $(g_n, n)$ defined by the condition $F^n(g_n(R))\ni x$.

Using Lebesgue's covering lemma, we show that if $x$ and $y$ are close
to each other, then there exists a sequence $(g_n, n)$ such that
$F^n(g_n(R))\supset\{x, y\}$ for all $n=0, \ldots, L$, where $L$ is
big. This shows that $\Lambda^{-1}$ exists and is continuous.
\end{proof}

Suppose now that the map $F:\splitting\arr\splitting$ has a fixed point $x_0$. Let
$\phi:G\arr G$ be the automorphism defined by the condition
$F(g(x_0))=\phi(g)(x_0)$. Let $W_+$ and $W_-$ be the stable and
the unstable plaques of $\splitting$ containing $x_0$. Let $R$, $\Omega_n(W_+, R)$,
and $\Xi(W_+, R)$ be as above. We assume that $R$ is connected and
$x_0\in R$.

Note that $F^n(g(R))\cap W_+\ne\emptyset$ is equivalent to $g(R)\cap
F^{-n}(W_+)=g(R)\cap W_+\ne\emptyset$. It follows that the set $\Omega_n(W_+,
R)$ does not depend on $n$.

The graph $\Xi(W_+, R)$ is isomorphic then to the graph with the
set of vertices $\Omega_0(W_+, R)\times\Z$ in which two vertices are
connected by an edge if and only if they are either of the form $(g_1,
n)$ and $(g_2, n)$, where $g_1, g_2\in\Omega_0(W_+, R)$ and
$g_1(R)\cap g_2(R)\ne\emptyset$, or of the form $(g_1, n)$ and $(g_2,
n+1)$, where $g_1, g_2\in\Omega_0(W_+, R)$ and $F^n(g_1(R))\cap
F^{n+1}(g_2(R))\ne\emptyset$, which is equivalent to
$g_1(R)\cap\phi(g_2)(F(R))\ne\emptyset$.
Note that $(g, n)\mapsto (g, n+1)$ is an automorphism of $\Xi(W_+, R)$.

Let $A\subset G$ be a finite set containing the identity, and let
$S$ be a finite generating set of $G$. We assume that $S$
contains all elements $g\in G$ such that $R\cap g(R)\ne\emptyset$ or
$R\cap g(F(R))\ne\emptyset$ and that $A\subset S$.

Let $\Omega_0'\subset G$ be any set such that $\Omega_0(W_+,
R)\subset\Omega_0'\subset\Omega_0(W_+, R)A$.

Denote then by $\Xi'$ the graph with the set of vertices
$\Omega_0'\times\Z$ with edges of two kinds: vertical
and horizontal. The horizontal edges connect two vertices $(g_1, n),
(g_2, n)$ if and only if $g_1^{-1}g_2\in S$. The
vertical edges connect a vertex $(g_1, n)$ to a vertex
$(g_2, n+1)$ if and only if $g_1^{-1}\phi(g_2)\in S$.

Note that if $g_1(R)\cap g_2(R)\ne\emptyset$, then $R\cap
g_1^{-1}g_2(R)\ne\emptyset$. If $g_1(R)\cap F(g_2(R))\ne\emptyset$,
then $g_1(R)\cap\phi(g_2)(F(R))\ne\emptyset$, hence $R\cap
g_1^{-1}\phi(g_2)(F(R))\ne\emptyset$. It follows that $\Xi(W_+, R)$
is a sub-graph of $\Xi'$.

\begin{proposition}
\label{pr:gammaprime}
The inclusion $\Xi(W_+, R)\hookrightarrow\Xi'$ is a quasi-isometry.
\end{proposition}

\begin{proof}
Let us prove at first the following lemmas.

\begin{lemma}
\label{lem:n1D}
There exists $n_1>0$ such that $\phi^{-n_1}(\Omega_0(W_+,
R))\subset\Omega_0(W_+, R)$.

For every such $n_1$ there exists $D_{n_1}$
such that distance from $(g, n)$ to
$(\phi^{-n_1}(g), n+n_1)$ in $\Xi(W_+, R)$ is less than $D_{n_1}$.
\end{lemma}

\begin{proof} By Lebesgue's covering lemma, there exists $\epsilon>0$
  such that for every $x\in\splitting$ there exists $g\in G$ such that the
  $\epsilon$-neighborhood of $x$ is contained in $g(R)$.
It follows that if distance from $x$ to $W_+$ is less
than $\epsilon$, then there exists $g\in\Omega_0(W_+, R)$ such that
$x\in g(R)$.

There exists an upper bound (equal to the diameter of $R$)
on the distance from $g(x_0)$ to $W_+$ for all $g\in\Omega_0(W_+,
R)$.  Consequently, there exists $n_1>0$ such that $\phi^{-n_1}(\Omega_0(W_+,
R))\subset\Omega_0(W_+, R)$.

Let us prove the second part of the lemma. Let $(g, n)\in\Xi(W_+,
R)$, and let $x\in g(R)\cap W_+$. Choose for $k=1, \ldots, n_1$,
$h_k\in\Omega_0(W_+, R)$ such that $F^{-k}(x)\in h_k(R)$. Then
$F^{-1}(x)\in h_1(R)$ and $x\in g(R)$, hence $F(h_1(R))\cap
g(R)\ne\emptyset$, which implies that $(g, n)$ is
connected to $(h_1, n+1)$. Similarly, $F^{-k}(x)\in
h_k(R)$ and $F^{-(k+1)}(x)\in h_{k+1}(R)$, hence $F(h_{k+1}(R))\cap
h_k(R)\ne\emptyset$, so that $(h_k, n+k)$ is connected
to $(h_{k+1}, n+k+1)$. We have $g(x_0)\in g(R)$, hence $F^{-n_1}(x)\in
h_{n_1}(R)$ and $F^{-n_1}(x), \phi^{-n_1}(g)(x_0)\in
F^{-n_1}(g(R))=\phi^{-n_1}(g)(F^{-n_1}(R))$. The set $F^{-n_1}(g(R))$ is
connected, hence there exists a path $f_1, f_2, \ldots, f_m$ in
$\Omega_0(W_+, R)\times\{n+n_1\}$ connecting $(h_{n_1}, n+n_1)$ to
$(\phi^{-n_1}(g), n+n_1)$. Since $F^{-n_1}(R)$ is relatively compact, there
exists a uniform bound $M$ such that we may assume that $m\le M$. It
follows that the distance in $\Xi(W_+, R)$ from $(g, n)$ to
$(\phi^{-n_1}(g), n+n_1)$ is not more than $n_1+m-1$.
\end{proof}

\begin{lemma}
\label{lem:B1}
For every finite set $B\subset G$ there exists $D>0$
such that if $g_1^{-1}g_2\in B$ for $g_1, g_2\in\Omega_0(W_+, R)$,
then the distance between $(g_1, n)$ and $(g_2, n)$ in $\Xi(W_+, R)$ is not
greater than $D$.
\end{lemma}

\begin{proof}
Let $R_B$ be a compact connected rectangle of $\splitting$ containing
$B(x_0)$. Note that it follows from Proposition~\ref{pr:agreeslocprod} that
 there exists a uniform upper bound on the $d_-$-diameter of the unstable
 plaques of $R_B$. Then it is also a uniform upper bound on the $d_-$-distance
from $W_+$ to a point of $g(R_B)$ for $g\in\Omega_0(W_+, R)$.

It
follows that for every $\epsilon>0$ there exists $n_2>0$ such that for every
$g\in\Omega_0(W_+, R)$ the set $F^{-n_2}(g(R_B))$ belongs to the
$\epsilon$-neighborhood of $W_+$, hence (if $\epsilon$ is small
enough) it is covered by the sets $h(R)$
for $h\in\Omega_0(W_+, R)$. Since $F^{-n_2}(g(R_B))$ is connected, for
every two points $x, y\in g(R_B)$ there exists a sequence $h_1, h_2,
\ldots, h_m\in\Omega_0(W_+, R)$ such that $F^{-n_2}(x)\in h_1(R)$,
$F^{-n_2}(y)\in h_m(R)$, $h_i(R)\cap h_{i+1}(R)\ne\emptyset$ for all
$i=1, \ldots, m-1$, and $h_i$ are pairwise different.
Since $F^{-n_2}(g(R_B))=\phi^{-n_2}(g)(F^{-n_2}(R_B))$ belong to the
$G$-orbit of $F^{-n_2}(R_B)$,
there exists a uniform upper bound $M(n_2)$ on the length $m-1$ of the
corresponding path $h_1, \ldots, h_m$ in $\Omega_0(W_+,
R)\times\{0\}$. We can choose $n_2$ bigger than the number
$n_1$ from Lemma~\ref{lem:n1D}.

Let $g_1, g_2\in\Omega_0(W_+, R)$ be such that $g_1^{-1}g_2\in
B$. Since $g_1^{-1}g_2\in B$, $g_1(x_0),
g_2(x_0)\in g_1(R_B)$. It follows that there exists a path $h_1,
\ldots, h_m\in\Omega_0(W_+, R)$ such that $m<M$, and
$F^{-n_2}(g_1(x_0))\in h_1(R)$, $F^{-n_2}(g_2(x_0))\in h_m(R)$. The
last two conditions are equivalent to $\phi^{-n_2}(g_1)(x_0)\in
h_1(R)$ and $\phi^{-n_2}(g_2)(x_0)\in h_m(R)$, which imply that
$(\phi^{-n_2}(g_1), 0)$ and $(\phi^{-n_2}(g_2), 0)$ are connected to
$(h_1, 0)$ and
$(h_m, 0)$, respectively, by horizontal edges. By Lemma~\ref{lem:n1D}, we
have a uniform bound on the distances from $(g_1, n)$ to
$(\phi^{-n_2}(g_1), n+n_1)$ and from $(g_2, n)$ to $(\phi^{-n_2}(g_2),
n+n_1)$, which finishes the
proof.
\end{proof}

The proof of the following lemma is analogous.

\begin{lemma}
\label{lem:B2}
For every finite set $B\subset G$ there exists $D>0$ such that if
$g_1^{-1}\phi(g_2)\in B$ for $g_1, g_2\in\Omega_0(W_+, R)$,
then the distance between $(g_1, n)$ and
$(g_2, n+1)$ in $\Xi(W_+, R)$ is not greater than $D$.
\end{lemma}

Let us go back to proving Proposition~\ref{pr:gammaprime}.
The image of $\Xi(W_+,
R)$ under the inclusion map is a 1-net in $\Xi'$. Distance between
vertices in $\Xi(W_+, R)$ is not less than the distance between them in $\Xi'$.

Let us show that there exists a constant $D>1$ such that distance
between $(g_1, n_1), (g_2, n_2)$ in $\Xi(W_+, R)$ is not more than $D$ times the
distance from $(g_1, n_1)$ to $(g_2, n_2)$ in $\Xi'$. Let $(g_1, n_1)=v_0, v_1, \ldots,
v_n=(g_2, n_2)$ be a geodesic path in $\Xi'$. Since $\Xi(W_+, R)$ is a
net in $\Xi'$, there exists a constant $C>1$ such that every such
geodesic path can be replaced by a path $(g_1, n_1), v_1', \ldots, v_{n-1}',
(g_2, n_2)$, where $v_i'\in\Xi(W_+, R)$, and distance from $v_i'$ to
$v_{n-1}'$ in $\Xi'$ is bounded from above by $C$. Moreover, we may
assume that each $v_i'$ belongs to the same level $\Omega_n'$ as
$v_i$. Then $v_i'$ and $v_{i+1}'$ either belong to one level, or to
two neighboring levels. Then Lemmas~\ref{lem:B1} and~\ref{lem:B2}
finish the proof.
\end{proof}

\begin{defi}
Let $G$ be the group of deck transformations of the splitting
$\pi:\splitting\arr\X$, and let $F:\splitting\arr\splitting$ be a lift of $f$ with a fixed point
$x_0$. A set $\Sigma\subset G$ is a \emph{coarse stable
  (resp.\ unstable) plaque} if the
stable plaque $W_+(x_0)$ (resp.\ unstable plaque $W_-(x_0)$)
and the set $\Sigma$ are of finite Hausdorff distance from each other.
\end{defi}

Recall that two subsets $A_1, A_2$ of a metric space $(X, d)$ are of a finite
Hausdorff distance from each other if there exists $D>0$ such that for
every $x\in A_1$ there exists $y\in A_2$, and for every $y\in A_2$
there exists $x\in A_1$ such that $d(x, y)<D$.

\begin{theorem}
\label{th:conjugacycriterion}
Let $(\X, f)$ be a connected and locally connected Smale
space. Suppose that there exists a splitting $\pi:\splitting\arr\X$ and a lift
$F:\splitting\arr\splitting$ of $f$ with a fixed point $x_0$. Let $\phi$ be the
associated automorphism of $G$, and let $\Sigma_+$ and $\Sigma_-$ be coarse
stable and unstable plaques of $x_0$. Then $(\X, f)$ is uniquely determined, up
to topological conjugacy by the quadruple $(G, \phi, \Sigma_+, \Sigma_-)$.
\end{theorem}

\begin{proof}
The dynamical system $(\X, f)$ is uniquely determined by the $G$-space
$\splitting$ and the map $F:\splitting\arr\splitting$. The group $G$ acts on the plaques $W_+$
and $W_-$ by the actions
\[g:x\mapsto [g(x), x],\qquad g:x\mapsto [x, g(x)],\]
respectively. The action of $G$ on $\splitting\cong W_+\times W_-$ is reconstructed from these
actions by the formula
\[g(x)=[[g(x), x], [x, g(x)]].\]
Similarly, the map $F:\splitting\arr\splitting$ is determined by the action of $F$ on
$W_+$ and $W_-$, since
\[F([x, y])=[F(x), F(y)].\]

Consequently, it is enough to show that the quadruple $(G, \phi,
\Sigma_+, \Sigma_-)$ uniquely determines the dynamical systems $(W_+,
G)$, $(W_+, F)$, $(W_-, G)$, and $(W_-, F)$, up to topological
conjugacy. Let us prove that the triple $(G, \phi, \Sigma_+)$ uniquely
determines the dynamical systems $(W_+, G)$ and $(W_+, F)$. The same
proof will show that $(G, \phi, \Sigma_-)$ uniquely determines $(W_-,
G)$ and $(W_-, F)$.

Let $R\subset\splitting$ be a relatively compact open rectangle such that
$x_0\in R$ and $\bigcup_{g\in G}g(R)=\splitting$. Let $\Omega_0(W_+, R)$ be, as before, the set
of elements $g\in G$ such that $g(R)\cap W_+\ne\emptyset$.

For a set $\Sigma\subset G$ and a finite generating set $S$ of $G$,
denote by $\Xi(\Sigma, S)$ the graph with the set of vertices $\Sigma\times\Z$ in which two
vertices are adjacent either if they are of the form $(g, n)$ and
$(gs, n)$ for $g, gs\in\Sigma$ and $s\in S$, or of the form $(g, n)$
and $(\phi^{-1}(gs), n+1)$ for $g, \phi^{-1}(gs)\in\Sigma$ and $s\in
S$. Note that the map $(g, n)\mapsto (g, n+1)$ is an automorphism of
$\Xi(\Sigma, S)$.

If $A$ is big enough, then $\Sigma_+A$ contains $\Omega_0(W_+,
R)$. Then, by Proposition~\ref{pr:gammaprime}, the identical embedding $\Xi(W_+,
R)\hookrightarrow\Xi(\Sigma_+A, S)$ is a quasi-isometry, provided
$S$ is big enough. It follows then from
Theorem~\ref{th:hyperbolicgraphs}
that $\Xi(\Sigma_+A, S)$ is Gromov hyperbolic, and that the
boundary of $\Xi(\Sigma_+A, S)$ minus the common limit $\omega$ of
quasi-geodesic paths of the form $(g_n, -n)$, $n\ge 1$, is
homeomorphic to $W_+$. Moreover, it follows directly from
Theorem~\ref{th:hyperbolicgraphs} that the natural homeomorphism
$\Phi:\partial\Xi(\Sigma_+A, S)\setminus\omega\arr W_+$ maps the limit of a
sequence $(g_n, n)\in\Xi(\Sigma_+A, S)$ to the limit of the
sequence $[F^n(g_n(x_0)), x_0]=[\phi^n(g_n)(x_0), x_0]\in W_+$.

Consequently, the homeomorphism $\Phi$ conjugates $F:W_+\arr W_+$ with the
map on the boundary of $\Xi(\Sigma_+A, S)$ induced by the automorphism $(g,
n)\mapsto (g, n+1)$. This shows that the dynamical system $(F, W_+)$
is uniquely determined by $(G, \phi, \Sigma_+)$.

It remains to show that for every $g\in G$ the homeomorphism $g:x\mapsto
[g(x), x_0]$ of $W_+$ is uniquely determined by $(G, \phi, \Sigma_+)$
and $g$.

Let $\xi$ be the limit of a sequence $(g_n, n)\in\Xi(\Sigma_+A, S)$, where
$s_n=g_n^{-1}\phi(g_{n+1})\in S$ for all $n\ge 0$. Note that every point
of $\partial\Xi(\Sigma_+A, S)$ can be represented this way, provided $S$
is big enough (see Theorem~\ref{th:hyperbolicgraphs}).

There exists $\epsilon>0$ such that for every $x\in\splitting$ there exists
$g\in G$ such that the $\epsilon$-neighborhood of $x$ is contained in
$g(R)$. Let $x\in W_+$. Choose for every $n\ge 0$ an element $g_n\in
G$ such that the $\epsilon$-neighborhood of $F^{-n}(x)$ is contained
in $g_n(R)$. Then $x$ is contained in $F^n(g_n(R))$. In particular,
$F^n(g_n(R))\cap F^{n+1}(g_{n+1}(R))\ne\emptyset$, i.e, the sequence
$(g_n, n)$ is a path in $\Xi(W_+, R)$, and its limit in
$\partial\Xi(W_+, R)$ is mapped by the natural homeomorphism to
$x$.

The rectangles $gF^n(g_n(R))$ contain $g(x)$ for all $n$. Since
$F$ is expanding in the unstable direction, the sets $gF^n(g_n(R))$
intersects $W_+$, i.e., $\phi^{-n}(g)g_n\in\Omega_0(W_+, R)$,
for all $n$ big enough.

The limit of the
intersections of $gF^n(g_n(R))$ with $W_+$ is equal to $[x_0, g(x)]$. It follows that
$(\phi^{-n}(g)g_n, n)$, where $n$ is big enough, is a path in
$\Xi(W_+, R)$ converging to the point of $\partial\Xi(W_+, R)$
corresponding to $[x_0, g(x)]\in W_+$, i.e., to the image of $x$ under
the action of $g$ on $W_+$.

Note that the left multiplication by $g$ preserves the distances between
the vertices of the graph $\Xi(\Sigma_+A, S)$ (when the images of the vertices
belong to the graph). It also follows from the classical properties
of Gromov hyperbolic graphs that there exists a
constant $\Delta_1$ such that if two paths $(g_n, n)$ and $(h_n, n)$
of $\Xi(\Sigma_+A, S)$ converge to the same point of the boundary,
then the distance between $(g_n, n)$ and $(h_n, n)$ is less than $\Delta_1$
for all $n$ big enough.

It follows that the action of $g$ on $W_+$ can
be modeled on the boundary of $\Xi(\Sigma_+A, S)$ by the following
rule. Take a path $(g_n, n)\in\Xi(\Sigma_+A, S)$ converging to
a point $\xi\in\partial\Xi(\Sigma_+A, S)$. If $(gg_n, n)$ for $n$
big enough belong to $\Xi(\Sigma_+A, S)$, then its limit is
$g(\xi)$. Since $\Xi(W_+, R)\subset\Xi(\Sigma_+A, S)$, this rule
will determine the action of $g$ on $\partial\Xi(\Sigma_+A, S)\setminus\{\omega\}$.
\end{proof}

\begin{theorem}
\label{th:equivariance}
Let $(\X_1, f_1)$ and $(\X_2, f_2)$ be connected and locally connected
Smale spaces. Suppose that there exist fixed points of $f_i$ and splittings
$\pi_i:\splitting_i\arr\X_i$. Let $G_i$ be the groups of deck
transformations of the splittings. Let $F_i$ be lifts of $f_i$, with
fixed points $x_i\in\splitting_i$. If there exists a continuous map
$\Phi:\splitting_1\arr\splitting_2$ and an isomorphism $\Psi:G_1\arr
G_2$ such that $\Phi(x_1)=x_2$, and
\[\Phi(F_1(x))=F_2(\Phi(x)),\qquad\Phi(F_1(g(x)))=F_2(\psi(g)(\Phi(x)))\]
for all $x\in\splitting_1$ and $g\in G_1$, then $(\X_1, f_1)$ and
$(\X_2, f_2)$ are topologically conjugate.
\end{theorem}

Note that we do not require $\Phi$ to be a homeomorphism.

\begin{proof}
The map $\Phi$ is proper as an equivariant map between two proper
actions (see, for instance~\cite[Lemma~5.2]{nek:models}).

Let $U$ be a compact neighborhood of $x_2\in\splitting_2$. Then
$\bigcup_{n\ge 1}\bigcap_{k\ge n}F_2^k(U)$ is equal to the unstable
plaque $W_-(x_2)$ in $\splitting_2$. Similarly, \[\bigcup_{n\ge
  1}\bigcap_{k\ge
  n}F_1^k(\Phi^{-1}(U))=\Phi^{-1}\left(\bigcup_{n\ge
    1}\bigcap_{k\ge n}F_2^k(U)\right)\] is equal to the unstable
plaque $W_-(x_1)$ in $\splitting_1$. It follows that
$\Phi^{-1}(W_-(x_2))=W_-(x_1)$. Similarly, $\Phi^{-1}(W_+(x_2))=W_+(x_1)$.

Let $K_2$ be a compact subset of $\splitting_2$ such that $K_2$ and
$K_1=\Phi^{-1}(K)$ are $G_i$-transversals, i.e., intersect every $G_i$-orbit. They exists, since the
actions of $G_i$ are co-compact, proper, and the map $\Phi$ is
continuous and proper.

Then $g_2(K_2)\cap W_-(x_2)\ne\emptyset$ for $g_2\in G_2$ is equivalent to
$\Psi^{-1}(g_2)(K_1)\cap W_-(x_1)\ne\emptyset$. The same is true for the stable
plaques $W_+(x_1)$ and $W_+(x_2)$. It follows that the sets $\Sigma_{i, *}=\{g_i(K_i)\cap W_*(x_i)\ne\emptyset\}$ for $*\in\{+, -\}$ and $i\in\{1, 2\}$ are coarse stable and unstable plaques for which we can use
Theorem~\ref{th:conjugacycriterion} to show that $(\X_1, f_1)$ and $(\X_2, f_2)$ are topologically conjugate.
\end{proof}

In fact, it follows from 
Theorem~\ref{th:hyperbolicgraphs} and the proof of
Theorem~\ref{th:equivariance} that any continuous map $\Phi$
satisfying the conditions of Theorem~\ref{th:equivariance} is a homeomorphism.

\section{Smale spaces with virtually nilpotent splitting}
\label{s:nilpotentsplitting}

Let $L$ be a simply connected nilpotent Lie group. 
Let $G$ be a finitely generated subgroup of $\mathop{\mathrm{Aut}}
L\ltimes L$ such that the action
of $G$ on $L$ is free, proper, and co-compact. Here we identify the
elements of $A$ with the transformations $g\mapsto \alpha(g)\cdot h$ of
$L$, where $\alpha\in\mathop{\mathrm{Aut}} L$ and $h\in L$.

Let $F\in\mathop{\mathrm{Aut}} L$ be a \emph{hyperbolic automorphism} of
$L$ (i.e., such that its differential $DF$ at the identity of $L$
has no eigenvalues of absolute value one). Then $F$ induces an
automorphism $\phi$ of $\mathop{\mathrm{Aut}}\ltimes L$ by conjugation. Suppose
that $G$ is invariant under this automorphism. Then $F$ induces an
Anosov homeomorphism $f:G\backslash L\arr G\backslash L$. Such homeomorphisms are called
\emph{hyperbolic infra-nilmanifold automorphisms}.

The aim of this section is to prove the following description of
locally connected Smale spaces that have a splitting with a virtually nilpotent group of
deck transformation.

\begin{theorem}
\label{th:Smaleinfranil}
Let $(\X, f)$ be a Smale space such that $\X$ is connected and locally connected, and
there exists a splitting $\pi:\splitting\arr\X$ with a virtually nilpotent
group of deck transformations. Then $(\X, f)$ is topologically
conjugate to a hyperbolic infra-nilmanifold automorphism.
\end{theorem}

\begin{proof}
Let $(\X, f)$ satisfy conditions of the theorem. Let us assume at
first that the group $G$ of deck transformations is nilpotent and
torsion free.

Let $F$ be a lift of $f$, and let $F(x_0)=x_1$. Then the map $g\mapsto \phi(g)$
defined by $F(g(x_0))=\phi(g)(x_1)$ is an automorphism of $G$.

Denote by $Z(G)$ the \emph{center} of $G$, i.e., the set of elements
of $G$ that commute with every element of $G$.

The group $Z(G)$ is obviously abelian and torsion free. It is finitely
generated, since all subgroups of a finitely generated nilpotent group
are finitely generated (see~\cite[5.2.17]{robinson}). Consequently,
$Z(G)$ is isomorphic to $\Z^d$ for some $d$.

\begin{lemma}
\label{lem:FCshift}
Let $g\in Z(G)$. There exist positive constants $D_-$ and $D_+$  such that for every
stable (resp.\ unstable) plaque $V$ of $\splitting$ and any $x\in V$, $y\in
g(V)$ we have $d_-(x, [x, y])\le D_-$ (resp.\ $d_+(x, [y, x])\le D_+$).
\end{lemma}

\begin{proof}
Let us prove the lemma for stable plaques.
Note that $[x, y]$ is equal to the intersection of $g(V)$ with
$W_-(x)$, and so does not depend on the choice of $y\in g(V)$.
Therefore, it is enough to show that $d_-(x, [x, g(x)])$ is bounded
for all $x\in\splitting$, see Figure~\ref{fig:center}.
Let $R\subset\splitting$ be a compact rectangle such that $\pi(R)=\X$. Let
$D$ be an upper bound on the value of $d_-(x, [x, y])$ for $x\in R$
and $y\in g(R)$.
It is finite, since there exists a compact rectangle
$P$ such that $P\supset R\cup g(R)$ (see also
Proposition~\ref{pr:agreeslocprod}).

\begin{figure}
\includegraphics{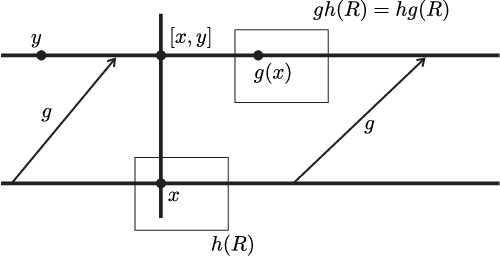}
\caption{}
\label{fig:center}
\end{figure}

For every $x\in\splitting$ there exists $h\in G$ such that $h(x)\in R$. Then
$d_-(x, [x, g(x)])=d_-(h(x), [h(x), hg(x)])=d_-(h(x), [h(x), gh(x)])<D$.
\end{proof}

Denote for $g\in Z(G)$
\[D_-(g)=\sup_{x\in\splitting}d_-(x, [x, g(x)]),\qquad D_+(g)=\sup_{x\in\splitting}d_+(x, [g(x), x]),
\] which are finite by Lemma~\ref{lem:FCshift}.
Note that we obviously have
\begin{equation}
\label{eq:Dstnorm}
D_+(g_1g_2)\le D_+(g_1)+D_+(g_2),\qquad D_-(g_1g_2)\le
D_-(g_1)+D_-(g_2)
\end{equation}
for all $g_1, g_2\in Z(G)$.

Let $\lambda\in (0, 1)$ and $C>1$ be such that for any two stably (resp.\
unstably) equivalent points $x, y\in \X$ we have
\[d_+(f^n(x),
f^n(y))\le C\lambda^nd_+(x, y)\] (resp.\ $d_-(f^{-n}(x), f^{-n}(y))\le
C\lambda^nd_-(x, y)$) for all $n\ge 0$. Then the same estimates will hold for $F$ and
$x, y\in\splitting$ belonging to one stable (resp.\ unstable) plaque.

Note that the center $Z(G)$ is characteristic (i.e., invariant under
automorphisms of $G$), hence $\phi(Z(G))=Z(G)$.

\begin{proposition}
\label{pr:Dstcontraction}
For every $g\in Z(G)$ and $n\ge 0$ we have
\[D_+(\phi^n(g))\le C\lambda^n D_+(g),\qquad D_-(\phi^{-n}(g))\le
C\lambda^n D_-(g).\]
\end{proposition}

\begin{proof}
Let us prove the first inequality. The second is proved the same
way.

Let $V$ be an unstable plaque of $\splitting$, and let $x\in V$ and
$y\in\phi^n(g)(V)$ be such that $x$ and $y$ belong to the same stable
plaque. It is enough to prove that $d_+(x, y)\le C\lambda^n D_+(g)$.

The points $F^{-n}(x)$
and $F^{-n}(y)$ belong to one stable plaque, and $F^{-n}(x)\in
F^{-n}(V)$, $F^{-n}(y)\in F^{-n}(\phi^n(g)(V))=g (F^{-n}(V))$, hence
$d_+(F^{-n}(x), F^{-n}(y))\le D_+(g)$. But this implies $d_+(x, y)\le
C\lambda^n d_+(x, y)$.
\end{proof}

\begin{proposition}
\label{pr:Awidth}
For every finite set $S\subset Z(G)$ there exists a constant $D_S>0$
satisfying the following condition. For every finite set $A\subset
Z(G)$ there exists $n_0$ such that for all $n\ge n_0$,  $g_1,
g_2\in\phi^n(A)\phi^{n-1}(S)\phi^{n-2}(S)\cdots\phi(S)S$, and every
unstable plaque $V$ we have
\[d_+(x, [y, x])<D_S\]
for all $x\in g_1(V)$ and $y\in g_2(V)$.
\end{proposition}

\begin{proof}
Let $\Delta_S$ and
$\Delta_A$ be upper bounds on $D_+(g)$ for $g\in S$ and $g\in A$, respectively.

Then, by~\eqref{eq:Dstnorm} and Proposition~\ref{pr:Dstcontraction}, we
have, for all $h\in A$, $g_i\in S$, and all $n$ big enough,
\begin{multline*}D_+(\phi^n(h)\phi^{n-1}(g_1)\cdots \phi(g_{n-1})g_n)\le\\
C\lambda^n\Delta_A+C(\lambda^{n-1}+\cdots +\lambda+1)\Delta_S<
1+\frac{C\Delta_S}{1-\lambda}.
\end{multline*}
It follows that we can take $D_S=2+\frac{2C\Delta_S}{1-\lambda}$.
\end{proof}

\begin{proposition}
\label{pr:hyperbolicautomorphism}
The restriction of the automorphism $\phi$ to $Z(G)\cong\Z^d$ is hyperbolic, i.e., has no eigenvalues of absolute value 1.
\end{proposition}

\begin{proof}
Suppose that on the contrary, there exists an eigenvalue
$\cos\alpha+i\sin\alpha$ of $\phi$ of absolute value 1. Suppose at
first that $\alpha\notin \pi\cdot\Z$. Then there
exists a two-dimensional subspace $L\le\R^d$ and a Euclidean structure
on it such that $\phi$ acts on
$L$ as a rotation by the angle $\alpha$. Denote
$K=\{(x_i)_{i=1}^d\in\R^d\;:\;|x_i|<1\}$, and let $S$ be
the set of elements $g\in Z(G)=\Z^d$ such that $\phi(K)\cap
(K+g)\ne\emptyset$ or $\phi^{-1}(K)\cap (K+g)\ne\emptyset$.
The set $S$ is obviously finite.

Let $R>0$ be
arbitrary, and consider the circle $\gamma$ of radius $R$ in $L$ with
center in the origin. Then $\phi(\gamma)=\gamma$. Let $A_R$ be the set of elements
$g\in Z(G)$ such that $K+g\cap\gamma\ne\emptyset$. It is finite and
non-empty. Note that union of the sets $A_R$ for all $R>0$ is infinite.

Let $h$ be an arbitrary element of $A_R$, and let $x\in K+h\cap\gamma$. Then
$\phi^{-1}(x)\in\gamma$, and there exists $g\in\Z^d$ such that
$\phi^{-1}(x)\in K+g$. Then $g\in A_R$, and $x\in\phi(K)+\phi(g)\cap
K+h$. It follows that $K+h-\phi(g)\cap \phi(K)\ne\emptyset$, so that
$h-\phi(g)\in S$. We see that $h=\phi(g)+(h-\phi(g))\in\phi(A_R)+S$. We
have proved that $A_R\subset\phi(A_R)+S$. It is proved the same way
that $A_R\subset\phi^{-1}(A_R)+S$. By induction we conclude that
\[A_R\subset\phi^n(A_R)+\phi^{n-1}(S)+\cdots+\phi(S)+S\]
and
\[A_R\subset\phi^{-n}(A_R)+\phi^{-(n-1)}(S)+\cdots+\phi^{-1}(S)+S\]
for all $n\ge 1$.

Fix an arbitrary point $x_0\in\splitting$. Since
$\phi^n(A_R)+\phi^{n-1}(S)+\cdots \phi(S)+S\supset A_R$ for all $n$,
it follows from Proposition~\ref{pr:Awidth} that
there exists $D_S>0$, not depending on $R$, such that $d_+(x_0, [g(x_0),
x_0])<D_S$ and $d_-(x_0, [x_0,
g(x_0)])<D_S$ for all $g\in A_R$. It follows that $g(x_0)$ belongs to
the rectangle $[B_+, B_-]$, where $B_\pm$ are the balls of radius $D_S$
with center in $x_0$ in the corresponding plaque containing
$x_0$. Note that the set $\{g\in G\;:\;g(x_0)\in [B_+, B_-]\}$ is
finite, does not depend on $R$, and contains $A_R$. But this is a
contradiction.

The case when the eigenvalue is equal to $\pm 1$ is similar (with
one-dimensional space $L$).
\end{proof}

Let $E_+$ (resp.\ $E_-$) be the sum of the
root subspaces of $\R^d$ of the eigenvalues $\lambda$ of $\phi$ such
that $|\lambda|<1$ (resp.\ $|\lambda|>1$). We have $\R^d=E_+\oplus
E_-$. Denote by $P_+$ and $P_-=1-P_+$ the projections onto $E_+$ and $E_-$,
respectively.

Denote
$K=\{(x_i)_{i=1}^d\in\R^d\;:\;|x_i|<1\}$.
Let $S\subset Z(G)$ be a finite set containing all elements $g\in Z(G)$
such that $K+g\cap (\phi(K)\cup\phi^{-1}(K))\ne\emptyset$.

\begin{proposition}
\label{pr:limits}
For every point $x\in E_+$ there exists a sequence $g_i\in
S$, $i=1, 2,\ldots$, and an element $g_0\in Z(G)$ such that
\[x=\lim_{n\to\infty}P_+(\phi^n(g_n)+\phi^{n-1}(g_{n-1})+\cdots+\phi(g_1)+g_0).\]

There exists a finite set $N\subset Z(G)$ such that an equality
\begin{multline*}
\lim_{n\to\infty}P_+(\phi^n(g_n)+\phi^{n-1}(g_{n-1})+\cdots+\phi(g_1)+g_0)=\\
\lim_{n\to\infty}P_+(\phi^n(g_n')+\phi^{n-1}(g_{n-1}')+\cdots+\phi(g_1')+g_0')
\end{multline*}
holds for $g_i, g_i'\in S$, $i\ge 1$, and
$g_0, g_0'\in Z(G)$ if and only if there exists
a sequence $h_n\in N$ such that
\[\phi^n(g_n')+\phi^{n-1}(g_{n-1}')+\cdots+\phi(g_1')+g_0'=
\phi^n(h_n+g_n)+\phi^{n-1}(g_{n-1})+\cdots+\phi(g_1)+g_0\]
for all $n$ big enough.
\end{proposition}

\begin{proof}
The sets $K+g$ cover $\R^d$ for $g\in Z(G)=\Z^d$, and the group $Z(G)$
is $\phi$-invariant, hence
for every $x\in E_+$ and $n\ge 0$ there exists $h_n\in Z(G)$ such that
$x\in \phi^n(K)+h_n$. We have then
$\phi^{n-1}(K)+h_{n-1}\cap\phi^n(K)+h_n\ne\emptyset$, hence
$\phi^{-1}(K)\cap K+\phi^{-n}(h_n)-\phi^{-n}(h_{n-1})\ne\emptyset$
which implies that $\phi^{-n}(h_n)-\phi^{-n}(h_{n-1})=g_n\in S$, i.e.,
$h_n=\phi^n(g_n)+h_{n-1}$. It follows that there exists a sequence
$g_i\in S$ such that
$h_n=\phi^n(g_n)+\phi^{n-1}(g_{n-1})+\cdots+\phi(g_1)+h_0$.

Note that since $x\in\phi^n(K)+h_n$, we have $\|P_+(h_n)-x\|<C\lambda^n$ for
some constant $C$. It follows that $x=\lim_{n\to\infty}P_+(h_n)$.

Note that the set of all limits
$\lim_{n\to\infty}P_+(\phi^n(g_n)+\phi^{n-1}(g_{n-1})+\cdots+\phi(g_1))$
for all choices of $g_i\in S$ is a bounded subset $T_+$ of $E_+$.

Suppose
that
\begin{multline*}\lim_{n\to\infty}P_+(\phi^n(g_n)+\phi^{n-1}(g_{n-1})+\cdots+\phi(g_1)+h)=\\
\lim_{n\to\infty}P_+(\phi^n(g_n')+\phi^{n-1}(g_{n-1}')+\cdots+\phi(g_1')+h')\end{multline*}
for $g_i, g_i'\in S$ and $h, h'\in Z(G)$.

Then for every $n\ge 0$ we have
\begin{multline*}
P_+(\phi^n(g_n)+\phi^{n-1}(g_{n-1})+\cdots+\phi(g_1)+h)-P_+(\phi^n(g_n')+\phi^{n-1}(g_{n-1}')+\cdots+\phi(g_1')+h')=\\
\left(P_+(\phi^{n+1}(g_{n+1}'))+P_+(\phi^{n+2}(g_{n+2}'))+\cdots\right)-\\
\left(P_+(\phi^{n+1}(g_{n+1}))+P_+(\phi^{n+2}(g_{n+2}))+\cdots\right)\in\phi^n(T_+-T_+).
\end{multline*}
It follows that
\begin{multline*}
P_+((g_n+\phi^{-1}(g_{n-1})+\cdots+\phi^{-(n-1)}(g_1)+\phi^{-n}(h))-\\
(g_n'+\phi^{-1}(g_{n-1}')+\cdots+\phi^{-(n-1)}(g_1')+\phi^{-n}(h')))\in
T_+-T_+.
\end{multline*}

Since $\phi^{-1}$ is contracting on $E_-$, there exists a compact set
$T_-\subset E_-$ such that for any $h$ and any sequence $g_i\in S$ we
have
\[P_-(g_n+\phi^{-1}(g_{n-1})+\cdots+\phi^{-(n-1)}(g_1)+\phi^{-n}(h))\in
T_-\]
for all $n$ big enough.

It follows that for all $n$ big enough the difference
\begin{multline*}(g_n+\phi^{-1}(g_{n-1})+\cdots+\phi^{-(n-1)}(g_1)+\phi^{-n}(h))-\\
(g_n'+\phi^{-1}(g_{n-1}')+\cdots+\phi^{-(n-1)}(g_1')+\phi^{-n}(h'))\end{multline*}
belongs to a bounded set $T=(T_+-T_+)\oplus(T_--T_-)$, hence we can
take $N=T\cap\Z^d$.
\end{proof}

Fix a stable plaque
$W_+=W_+(x_0)$ of $\splitting$. The group $G$ acts on $W_+$ by $x\mapsto [g(x), x]$,
since $G$ preserves the direct product structure of $\splitting$.

Let $v\in\R^d$, and denote $v_+=P_+(v)$ and $v_-=P_-(v)$. Using
Proposition~\ref{pr:limits},
find a sequence $g_i\in S$, $i\ge 1$, and $g_0\in Z(G)$ such that
\[v_+=\lim_{n\to\infty}P_+(\phi^n(g_n)+\cdots+\phi(g_1)+g_0),\]
and define for $x\in W_+$
\begin{equation}
\label{eq:eplusaction}
v_+(x)=\lim_{n\to\infty}[(\phi^n(g_n)+\cdots+\phi(g_1)+g_0)(x), x].
\end{equation}

We also define for $x\in W_-$, where $W_-$ is an unstable plaque:
\begin{equation}
\label{eq:minusaction}
v_-(x)=\lim_{n\to\infty}[x,
(\phi^{-n}(g_n)+\cdots+\phi^{-1}(g_1)+g_0)(x)],
\end{equation}
where $g_i\in S$, for $i\ge 1$, and $g_0\in Z(G)$ are such that
\[v_-=\lim_{n\to\infty}P_-(\phi^{-n}(g_n)+\cdots+\phi^{-1}(g_1)+g_0).\]

(Replacing in Proposition~\ref{pr:limits} $\phi$ by $\phi^{-1}$ and
$E_+$, $P_+$ by $E_-$, $P_-$, we see that such a sequence $g_n$ exists.)

\begin{proposition}
The limit~\eqref{eq:eplusaction} exists and depends only on $v_+$ and
$x$. The limit~\eqref{eq:minusaction} exists and depends only on $v_-$
and $x$.
\end{proposition}

\begin{proof}
It follows directly from~\eqref{eq:Dstnorm},
Propositions~\ref{pr:Dstcontraction} and~\ref{pr:limits}.
\end{proof}

\begin{theorem}
\label{th:rdaction}
The limits~\eqref{eq:eplusaction} and~\eqref{eq:minusaction} define
continuous actions of $E_+$ and $E_-$ 
on $W_+=W_+(x_0)$ and $W_-=W_-(x_0)$, respectively.
Their direct sum is a continuous action
of $\R^d$ on $\splitting$. This action satisfies the following conditions:
\begin{enumerate}
\item it is free and proper;
\item its restriction onto $Z(G)=\Z^d<\R^d$ coincides with the original action
  of $Z(G)$ on $\splitting$;
\item it preserves the direct product structure, i.e., $v([x,
  y])=[v(x), v(y)]$ for all $v\in\R^d$ and $x, y\in\splitting$;
\item $F(v(x))=\phi(v)(F(x))$ for all $x\in\splitting$ and $v\in\R^d$;
\item the action commutes with $G$, i.e.,
$v(g(x))=g(v(x))$ for all $x\in\splitting$, $g\in G$, and $v\in\R^d$;
\item if $g(x)=v(x)$ for $g\in G$ and $v\in\R^d$, then $g=v\in Z(G)$.
\end{enumerate}
\end{theorem}

\begin{proof}
The fact that conditions~\eqref{eq:eplusaction} and~\eqref{eq:minusaction}
define actions follows directly from the fact
that the limits do not depend on $S$ and the choice of the sequences
$g_i$.

Let us prove that the action is continuous. It is enough to prove that
the action of $E_+$ on $W_+$ is continuous. We have to
show that for every $v_1\in E_+$, $x\in W_+$, and
$\epsilon>0$ there exists $\delta>0$ such that if
$v_2\in E_+$ and $y\in W_+$ are such that $\|v_1-v_2\|<\delta$
and $d_+(x, y)<\delta$, then $d_+(v_1(x), v_2(y))<\epsilon$.

Take an arbitrary $\epsilon>0$.
For every $n$ there exists $\delta_1(n)$ and a sequence $g_0\in H$, $g_i\in
S$, $i\ge 1$, such that
$v_1=\lim_{m\to\infty}P_+(g_0+\phi(g_1)+\cdots+\phi^m(g_m))$ and all
points $v_2$ in the $\delta_1(n)$-neighborhood of $v_1$ can be represented
as limits $v_2=\lim_{m\to\infty}P_+(h_0+\phi(h_1)+\cdots+\phi^m(h_m))$
for $h_0\in Z(G)$, and $h_i\in S$, $i\ge 1$, such that $h_i=g_i$ for $i=0,
1, \ldots, n$ (see the proof of Proposition~\ref{pr:limits} and use
Lebesgue's covering lemma).
There exists $\delta_2(n)$ such that if
$y\in W_+$ is such that $d_+(x, y)<\delta_2(n)$, then
\[
d_+([(g_0+\phi(g_1)+\cdots+\phi^n(g_n))(x), x],
[(g_0+\phi(g_1)+\cdots+\phi^n(g_n))(y), y])<\epsilon/2,
\]
since the function $y\mapsto [(g_0+\phi(g_1)+\cdots+\phi^n(g_n))(y),
y]$ is continuous. There exist constants $C>0$ and $\lambda\in (0, 1)$
such that
\[d_+(u(z), [(g_0+\phi(g_1)+\cdots+\phi^n(g_n))(z), z])<C\lambda^n\]
for all $z\in W_+$ and $u\in E_+$ such that
$u=\lim_{m\to\infty}P_+(g_0+\phi(g_1)+\cdots+\phi^m(g_m))$ for $g_0\in
Z(G)$ and $g_i\in S$ for $i\ge 1$. 

Take $n\ge
-\frac{\log(\epsilon/4C)}{\log\lambda}$.
Then for all $v_2\in E_+$ and
$y\in W_+$ such that $\|v_1-v_2\|<\delta_1(n)$ and $d_+(x, y)<\delta_2(n)$ we
have
\begin{multline*}
d_+(v_1(x), v_2(y))\le\\
d_+(v_1(x),
[(g_0+\phi(g_1)+\cdots+\phi^n(g_n)(x),
x])+\\ d_+([(g_0+\phi(g_1)+\cdots+\phi^n(g_n)(x), x],
[(g_0+\phi(g_1)+\cdots+\phi^n(g_n)(y), y])+\\
d_+(v_2(y), [(g_0+\phi(g_1)+\cdots+\phi^n(g_n)(y), y])\le
\epsilon/4+\epsilon/2+\epsilon/4=\epsilon.
\end{multline*}
Which shows that the action of $E_+$ on $W_+$ is continuous.

The same arguments (using Proposition~\ref{pr:Dstcontraction} and
inequalities~\eqref{eq:Dstnorm}) 
as in the proof of the criterion of equality of two
limits in Proposition~\ref{pr:limits}
show that an equality
\[\lim_{n\to\infty}[(\phi^n(g_n)+\cdots\phi(g_1)+g_0)(x), x]=
\lim_{n\to\infty}[(\phi^n(g_n')+\cdots\phi(g_1')+g_0')(x), x]\] for
$g_i, g_i'\in S$, $i\ge 1$, and $g_0, g_0'\in Z(G)$
is equivalent to the equality
\[\lim_{n\to\infty}P_+(\phi^n(g_n)+\cdots\phi(g_1)+g_0)=\lim_{n\to\infty}P_+(\phi^n(g_n')+\cdots\phi(g_1')+g_0').\]
This (and a similar statement for $P_-$ and the action on $W_-$) shows
that the action of $\R^d$ is free.

Let us show that the action is proper. Let $B\subset\splitting$ be a compact
set.
We have to show that the set
$\{v\in\R^d\;:\;v(B)\cap B\ne\emptyset\}$ is compact. It is closed,
since the action is continuous.

Denote $K=\{(x_i)_{i=1}^d\in\R^d\;:\;|x_i|\le 1\}$. Then for every
$v\in\R^d$ there exists $h\in Z(G)=\Z^d$ such that $v-h\in K$.
The set $K(B)=\{v(x)\;:\;x\in B, v\in K\}$ is compact, since the
action is continuous and the sets $B$ and $K$ are compact. The action
of $G$ on $\splitting$ is proper, hence the set $A$ of elements $h\in Z(G)$ such
that $h(K(B))\cap B\ne\emptyset$ is finite.

Suppose
that $x\in B$ and $v\in\R^d$ are such that $v(x)\in B$. There exists
$h\in Z(G)$ such that $v-h\in K$. Then $v(x)=(h+v-h)(x)\in h(K(B))\cap
B$, hence $h\in A$, so that $v\in K+A$. But the set $K+A$ is compact,
which proves that the action of $\R^d$ on $\splitting$ is proper.

The proof of statements (2)--(5) is straightforward, using the
fact that the action does not depend on the choice of $S$.

Let us prove the last statement. Suppose that $g(x)=v(x)$ for $g\in G$
and $v\in\R^d$. Then $g$ leaves invariant the orbit $\R^d(x)$ of $x$. Let
$G_1$ be the group of all elements leaving the $\R^d(x)$
invariant. The action of $Z(G)$ on $\R^d(x)$ is co-compact, the action
of $G$ on $\splitting$ is proper, hence the index of $Z(G)$ in $G_1$ is
finite, i.e., the image of $G_1$ in $G/Z(G)$ is finite. But $G/Z(G)$
is torsion free (see~\cite[5.2.19]{robinson}). Consequently, $G_1=Z(G)$. Since the action of $\R^d$
on $\R^d(x)$ is free, this implies that $g=v\in Z(G)$.
\end{proof}

\begin{proposition}
\label{pr:Htransitive}
If $G$ is abelian, then the action of $\R^d$ on $\splitting$ is transitive (i.e., has exactly one orbit).
\end{proposition}

\begin{proof}
It is enough to show that for every point
$x\in W_+$ there exists a sequence $g_i\in S$, $i\ge 1$, and an
element $g_0\in G$ such that
\begin{equation}
\label{eq:surjective}
x=\lim_{n\to\infty}[(\phi^n(g_n)+\cdots+\phi(g_1)+g_0)(x_0), x_0].
\end{equation}

The action of $G$ on $\splitting$ is co-compact, hence
there exists a relatively compact open rectangle
$R\subset\splitting$ containing $x_0$ and such that $\bigcup_{h\in G}h(R)=\splitting$.

Then for every $x\in W_+$ and every $n\ge 0$ there exists $h_n\in G$ such that
$x\in h_n(F^n(R))$. Assume that $S$ is big enough so that it contains all
elements $h\in G$ such that $h(R)\cap(F(C)\cup
F^{-1}(C))\ne\emptyset$. Then the same arguments as in the proof of
Proposition~\ref{pr:limits} show that there exists a sequence $g_i\in
S$, $i\ge 1$, and an element $g_0\in G$ for which~\eqref{pr:limits} holds.
\end{proof}

\begin{theorem}
\label{th:virtuallyabelian}
Let $(\X, f)$ be a locally connected and connected Smale space which
has a splitting with a free abelian group of deck transformations
$G\cong\Z^d$. Let $\phi$ be the automorphism of $G$ induced by a lift
of $f$. Then $(\X, f)$ is topologically conjugate to the hyperbolic automorphism
of the torus $\R^d/\Z^d$ induced by $\phi$. In particular, $(\X, f)$
has a fixed point.
\end{theorem}

\begin{proof}
By Proposition~\ref{pr:Htransitive}, the action of $\R^d$ on
$\splitting$ defined in Theorem~\ref{th:rdaction} is transitive.

Fix a basepoint $x_0\in\splitting$, define
$\rho_0:\R^d\arr\splitting$ by $v\mapsto v(x_0)$. The map $\rho_0$ is a homeomorphism, since
it is continuous, bijective, and proper.
Denote $v_0=\rho_0^{-1}(F(x_0))$, i.e., $v_0\in\R^d$ is such that $v_0(x_0)=F(x_0)$.

Then the map $\phi_0=\rho_0^{-1}F\rho_0:\R^d\arr\R^d$ satisfies
\[\phi_0(v)=\rho_0^{-1}(F(v(x_0))=\rho_0^{-1}(\phi(v)(F(x_0)))=
\rho_0^{-1}((\phi(v)+v_0)(x_0))=\phi(v)+v_0.\]

The linear operator $1-\phi$ is invertible, since $\phi$ is
hyperbolic. Therefore, there exists $w_0\in\R^d$ such that
$w_0-\phi(w_0)=v_0$, i.e., $\phi_0(w_0)=\phi(w_0)+v_0=w_0$. Define then
\[\rho_1(v)=\rho_0(v+w_0).\]
We have then
\begin{multline*}
F(\rho_1(v))=F(\rho_0(v+w_0))=\rho_0(\phi_0(v+w_0))=\\
\rho_0(\phi(v+w_0)+v_0)=\rho_0(\phi(v)+\phi(w_0)+v_0)=\rho_0(\phi(v)+w_0)=\rho_1(\phi(v)).
\end{multline*}
The statement of the theorem follows now directly from Theorem~\ref{th:rdaction}.
\end{proof}

Let us go back to the case when $G$ is torsion free nilpotent.

\begin{proposition}
\label{pr:lipschitzaction}
The action of $\R^d$ on $\splitting$ is uniformly locally Lipschitz, i.e.,
there exist $\epsilon>0$ and $C>1$ such that for every $v\in\R^d$ and
all $x, y\in\splitting$ such that $d_\splitting(x, y)<\epsilon$ we have
\[d_\splitting(v(x), v(y))\le Cd_\splitting(x, y).\]
\end{proposition}

Note that Proposition~\ref{pr:lipschitzaction} implies that $C^{-1}d_\splitting(x, y)\le d_\splitting(v(x), v(y))$ for all $x,
y\in\splitting$ such that $d_\splitting(x, y)\le C^{-1}\epsilon$.

\begin{proof}
By Theorem~\ref{th:rdaction}, $G$ maps $\R^d$-orbits to
$\R^d$-orbits. 

Let $K=\{(x_i)_{i=1}^d\in\R^d\;:\;|x_i|\le 1\}$, and let $R\subset\splitting$ be a
relatively compact rectangle such that $\pi(R)=\X$. Let
$\epsilon>0$ and $C>1$ be such that 
\begin{equation}
\label{eq:dWdpm}
C^{-1}d_\splitting(x, y)\le d_+(x, [y,
x])+d_-(x, [x, y])\le Cd_\splitting(x, y)
\end{equation}
for all $x, y\in\splitting$ such that $d_\splitting(x,
y)<\epsilon$, see Subsection~\ref{ss:hypgraphs}, where the metric
$d_\splitting$ is defined.
We also assume that $\epsilon$ is sufficiently small so that
for all $x, y\in\splitting$ such that $d_\splitting(x,
y)<\epsilon$ there exists $g\in G$ such that $g(x), g(y)\in R$.

Let $\delta$ is such that $d_\splitting(v(x), v(y))<\epsilon$ for all $v\in K$
and all $x, y\in R$ such that $d_\splitting(x, y)<\delta$. It exists, since the
action of $\R^d$ is continuous, and the set $K$ and the closure of
$R$ are compact.
For all $x, y\in\splitting$ such that $d_\splitting(x, y)<\delta$ and all
$v\in\R^d$ there exists $g\in G$ and $h\in Z(G)$ such that $g(x), g(y)\in
R$, and $v+h\in K$.
Then $d_\splitting(g(x), g(y))<\delta$, hence
\begin{multline*}d_\splitting(v(x), v(y))=d_\splitting(g^{-1}\cdot v\cdot g(x),
g^{-1}\cdot v\cdot g(y))=\\
d_\splitting((h+v)(g(x)), (h+v)(g(y)))<\epsilon.
\end{multline*}
We have shown that for all $x, y\in\splitting$ such that $d_\splitting(x, y)<\delta$
and all $v\in\R^d$ we have $d_\splitting(v(x), v(y))<\epsilon$.

Let $x, y\in\splitting$ be such that $d_\splitting(x, y)<\epsilon$.
Let $n$ be the biggest positive integer such
that $d_+(F^{-n}(x), F^{-n}([y, x]))<C^{-1}\delta$. Then $n$ is
equal, up to an additive constant, to $-\log d_+(x, [y, x])/\alpha_+$,
where $\alpha_+$ is the exponent of $d_+$.

We have then $d_+(v(F^{-n}(x)), v(F^{-n}([y, x]))<\epsilon$ for all
$v\in\R^d$. Applying $F^n$, and using the fact that $\phi$ is an
automorphism of $\R^d$, we get that $d_+(u(x), u([y,
x]))\le C_1e^{-n\alpha_+}\le C_2d_+(x, [y, x])$
for all $u\in\R^d$, where $C_1, C_2$ are constant (not depending on
$x, y$).

In the same way we prove that $d_-(u(x), u([x,
y]))<C_3d_-(x, [x, y])$ for all $u\in\R^d$, if $d_\splitting(x, y)<\epsilon$.
It follows then from~\eqref{eq:dWdpm} that there
exist $\epsilon_1>0$ and $C_4>0$ such that if
$x, y\in\splitting$ are such that $d_\splitting(x, y)<\epsilon_1$, then $d_\splitting(u(x),
u(y))<C_4d_\splitting(x, y)$ for all $u\in\R^d$.
\end{proof}

Let $\overline\splitting$ be the quotient of $\splitting$ by the $\R^d$-action defined
in Theorem~\ref{th:rdaction}.
We denote for $x\in\splitting$ by $\overline x$ the $\R^d$-orbit of
$x$. Since $G$ maps $\R^d$-orbits to $\R^d$-orbits, the action of $G$
on $\splitting$ induces a well-defined action of $G$ on $\overline\splitting$. Denote
$\overline G=G/Z(G)$, it is a torsion-free finitely generated
nilpotent group of nilpotency class one less than the class of $G$. By
Theorem~\ref{th:rdaction}, $Z(G)$ is equal to the kernel of the action
of $G$ on $\overline\splitting$, and the action of $\overline G$ on $\overline
\splitting$ is free.

The action of $\R^d$ on $\splitting$ descends to a free action of $\R^d/\Z^d$ on
$\X$, whose orbits are the images of the $\R^d$-orbits under the map
$\pi:\splitting\arr\X$.
Let $\overline\pi:\overline\splitting\arr\overline\X$ be the corresponding map
induced by $\pi:\splitting\arr\X$.

Denote, for $\overline x, \overline y\in\overline\splitting$
\[\delta(\overline x, \overline y)=\inf\{x_1\in\overline x,
y_1\in\overline y\;:\;d_\splitting(x_1, y_1)\}.\]

\begin{lemma}
\label{lem:overlined}
There exist $\epsilon>0, C>1$ and a $\overline G$-invariant metric
$\overline d$ on $\overline\splitting$ such that
\[C^{-1}\delta(\overline x, \overline y)\le
\overline d(\overline x, \overline y)\le C\delta(\overline x,
\overline y)\]
for all $x, y$ such that $\delta(\overline x, \overline
y)<\epsilon$.
\end{lemma}

\begin{proof}
The function $\delta$ is $\overline G$-invariant, since the metric
$d_\splitting$ is $G$-invariant. Note that
$\delta(\overline x, \overline y)>0$ for all $\overline x\ne\overline
y$, since the quotient map $\pi:(\splitting, d_\splitting)\arr(\X, d)$ is a local
isometry, the images of $\overline x$ and $\overline y$ in $\X$ are
compact, hence distance between any two points of $\pi(\overline x)$
and $\pi(\overline y)$ are bounded from below.

Let $\epsilon$ be as in Proposition~\ref{pr:lipschitzaction}.
Define $\overline d(\overline x, \overline y)$ as infimum of
$\sum\delta(\overline x_i, \overline x_{i+1})$ over all sequences
$\overline x=\overline x_0, \ldots, \overline x_n=\overline y$ such
that $\delta(\overline x_i, \overline x_{i+1})<\epsilon$. Note
that by Proposition~\ref{pr:lipschitzaction} there exists $C>1$ such that
$\delta(\overline x, \overline y)\le C\sum\delta(\overline
x_i, \overline x_{i+1})$, hence
\[\delta(\overline x,
\overline y)\le C\overline d(\overline x, \overline y)\]
for all $\overline x, \overline y\in\overline\splitting$. We also have
\[\overline d(\overline x, \overline y)\le\delta(\overline x,
\overline y)\]
for all $\overline x, \overline y$ such that $\delta(\overline x,
\overline y)<\epsilon$.
\end{proof}

\begin{proposition}
\label{pr:topologyonoverline}
The topology defined by the metric $\overline d$ coincides with the
quotient topology on $\overline\splitting$. The map $\overline\pi:\overline
\splitting\arr\overline\X$ is uniformly locally bi-Lipschitz with respect to $\overline
d$ and the metric on $\overline\X$ coming from the Hausdorff distance
between compact subsets of $\X$.
\end{proposition}

\begin{proof}
By Proposition~\ref{pr:lipschitzaction} and the definition of $\overline
d$, there exist $C>1$ and $\epsilon>0$ such that if $\overline
d(\overline x, \overline y)<\epsilon$, then for every $y\in\overline
y$ there exists $x\in\overline x$ such that $d_\splitting(x, y)<C\epsilon$.
Suppose that $\overline U\subset\overline\splitting$ is open with respect to
$\overline d$. Let $U\subset\splitting$ be the preimage of $\overline U$. Then
for every $\overline x\in\overline U$ there exists $\epsilon>0$ such
that the $\epsilon$-neighborhood of $\overline x$ (with respect to
$\overline d$) is contained in $\overline U$. Let $\overline y$ be
such that $\overline d(\overline x, \overline y)<C^{-1}\epsilon$. Then for every
$y\in\overline y$ there exists $x\in\overline x$ such that $d_\splitting(x,
y)<\epsilon$. It follows that the $\epsilon$-neighborhood of the set
$\overline x$ contains the set $\overline y$. It follows that $U$ is
open in $\splitting$.

Suppose that $U\subset\splitting$ is an $\R^d$-invariant open subset of
$\splitting$. Then for every $x\in U$ there exists $\epsilon>0$ such that the
$\epsilon$-neighborhood of $x$ is contained in $U$. Suppose that
$\overline y\subset U$ is such that $\overline d(\overline x,
\overline y)<C^{-1}\epsilon/2$. Then there exists $y\in\overline y$ such
that $d_\splitting(x, y)\le\epsilon$. Then $y\in U$, hence $\overline y\subset
U$, since $U$ is $\R^d$-invariant. We have shown that every set that
is open in the quotient topology is also open with respect to
$\overline d$.

The statement about the Hausdorff distance also follows directly from the
definition of $\overline d$ and Proposition~\ref{pr:lipschitzaction}
(and the fact that $\pi:\splitting\arr\X$ is a local isometry).
\end{proof}

Note that $v(x)=[v_+(x), v_-(x)]$, where $v_+=P_+(v)$ and $v_-=P_-(v)$. It
follows that for any $v, u\in\R^d$ and $x, y\in\splitting$ we have \[[v(x),
u(y)]=[v_+(x), u_-(y)]=[(v_++u_-)(x), (v_++u_-)(y)]=(v_++u_-)([x,
  y]),\] i.e., the value
of $\overline{[x, y]}$ depends only on $\overline x$ and $\overline
y$. It follows that the function $[\overline x, \overline y]=\overline{[x, y]}$ 
is well defined and satisfies the equalities (1) and (2) of
Definition~\ref{def:directproduct}. We will prove that it is
continuous in the next proposition.

We also have a well defined homeomorphism
$\overline F(\overline x)=\overline{F(x)}$, by condition (4) of
Theorem~\ref{th:rdaction}.

\begin{lemma}
\label{lem:locproddelta}
The metric $\overline d$ agrees with the local product structure on
$\overline\splitting$. In particular, the map $[\cdot, \cdot]$ is continuous.
\end{lemma}

\begin{proof}
We know that the metric $d_\splitting$ on $\splitting$ agrees with the local product structure on
$\splitting$, since it is locally isometric to the standard metric on $\X$.
Let $\overline x$, $\overline y$ be points of $\overline\splitting$ such that
$\delta(\overline x, \overline y)$ is small. Let us prove that
\begin{equation}
\label{eq:delta1}
\delta(\overline x, \overline y)\asymp\delta(\overline x,
[\overline x, \overline y])+\delta(\overline y, [\overline x,
\overline y])
\end{equation}
for all $\overline x$ and $\overline y$ that are close enough to each
other. (Here $F_1\asymp F_2$ means that there exists a constant $C>1$
such that $C^{-1}F_1\le F_2\le CF_1$.)

There exist points $x\in\overline x$ and $y\in\overline
y$ such that $d_\splitting(x, y)\le 2\delta(\overline x, \overline y)$. Since
$d$ agrees with the local product structure on $\splitting$, we have
\[d_\splitting(x, y)\asymp d_\splitting(x, [x, y])+d_\splitting(y, [x, y]).\]

We have $\delta(\overline x, [\overline x, \overline y])\le d_\splitting(x, [x,
y])$ and $\delta(\overline y, [\overline x, \overline y])\le d_\splitting(y, [x,
y])$, hence there exists a constant $C_1>1$ such that
\[\delta(\overline x, \overline y)\ge C_1^{-1}(\delta(\overline x,
[\overline x, \overline y])+\delta(\overline y, [\overline x,
\overline y])).\]
On the other hand, since $\delta$ is equivalent to a metric
(see Lemma~\ref{lem:overlined}), there exists $C_2>1$ such that
\[\delta(\overline x, \overline y)\le C_2(\delta(\overline x,
[\overline x, \overline y])+\delta(\overline y, [\overline x,
\overline y])),\] by the triangle inequality.
This proves~\eqref{eq:delta1}.

\begin{lemma}
\label{lem:stablemin}
There exist $C>1$ and $\epsilon>0$ such that if
$\overline x, \overline y\in\overline{\splitting}$ are such that
$[\overline x, \overline y]=\overline x$ (i.e., $\overline x$ and
$\overline y$ belong to the same stable plaque of $\overline
\splitting$) and $\delta(\overline x, \overline y)<\epsilon$, then there exist
$x\in\overline x$ and $y\in\overline y$ such that $[x, y]=x$ and $d_\splitting(x,
y)\le C\delta(\overline x, \overline y)$.
\end{lemma}

\begin{proof}
Let $x, y$ belong to one stable plaque of $\splitting$.
There exist
$\epsilon>0$ and $C>1$ (not depending on $x, y$) such that if $d_\splitting([x, g(x)], [y,
h(y)])<\epsilon$ for $g, h\in G$, then
\begin{multline*}
d_\splitting([x, g(x)], [y, h(y)])\ge\\ C^{-1}(d_\splitting([x, g(x)],
[y, g(y)])+d_\splitting([y, g(y)], [y, h(y)]))\ge\\ C^{-1}d_\splitting([x, g(x)], [y,
g(y)]),
\end{multline*} see Figure~\ref{fig:prod}.

\begin{figure}
\includegraphics{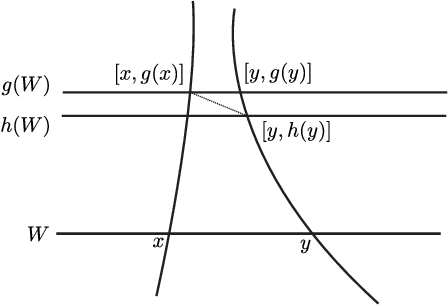}
\caption{}
\label{fig:prod}
\end{figure}

It follows that
\begin{equation}
\label{eq:2inf}
\inf\{d_\splitting(v(x), u(y))\;:\;v, u\in E_-\}\ge C^{-1}\inf\{d_\splitting(v(x),
v(y))\;:\; v\in E_-\},
\end{equation}
if the left hand side of the inequality is less than $\epsilon$.

We always can find
$x_1\in\overline x$ and $y_1\in\overline y$ such that
$d_\splitting(x_1, y_1)\le 2\delta(\overline x, \overline y)$. Then $[x_1, y_1]=v(x_1)$ for some
$v\in\R^d$.

Note that $[x_1, y_1]$ and $x_1$ belong to the same unstable
plaque of $\splitting$, hence $v\in E_-$. Then $[v(x_1), y_1]=[x_1, y_1]=v(x_1)$, i.e.,
the points $v(x_1)$ and $y_1$ belong to the same stable plaque.

It follows from~\eqref{eq:2inf} that
\begin{multline*}\inf\{d_\splitting(uv(x_1), u(y_1))\;:\;u\in E_-\}\le\\
C\inf\{d_\splitting(uv(x_1), w(y_1))\;:\;u,
w\in E_-\}\le Cd_\splitting(x_1, y_1)\le\\
2C\delta(\overline x, \overline y).
\end{multline*}

It follows that there exists a pair of points $x'\in\overline x,
y'\in\overline y$ such that $x'$ and $y'$ belong to the same stable
plaque, and $d_\splitting(x', y')\le 2C\delta(\overline x, \overline y)$.
\end{proof}

Let $\overline x, \overline y\in \overline\splitting$ be such that
$\delta(\overline x, \overline y)$ is small, and
$[\overline x, \overline y]=\overline x$. Let $\overline x_2\in
\overline\splitting$ be a point close to $\overline x$. Using
Lemma~\ref{lem:stablemin}, find $x\in\overline x$ and $y\in\overline
y$ such that $[x, y]=x$ and $d_\splitting(x, y)\le C_1\delta(\overline x, \overline
y)$. By Proposition~\ref{pr:lipschitzaction},
there exists $x_2\in\overline x_2$ such that $d_\splitting(x, x_2)\le
C_2\delta(\overline x, \overline x_2)$. We conclude from this, and from the fact
that $d_\splitting$ agrees with the
local product structure on $\splitting$, that there exists a constant $C_3>1$
such that if $d_\splitting(x, y)$ and $\delta(\overline x, \overline x_2)$ are
small enough, we have
\[d_\splitting([x, x_2], [y, x_2])\le C_3d_\splitting(x, y).\]
Consequently,
\[\delta([\overline x, \overline x_2], [\overline y, \overline
x_2])\le
d_\splitting([x, x_2], [y, x_2])\le C_3d_\splitting(x, y)\le  C_3C\delta(\overline x,
\overline y).\]
It follows that the maps $W_+(\overline x_1)\arr 
W_+(\overline x_2):\overline x\mapsto [\overline x, \overline x_2]$
are locally Lipschitz with respect to the metric $\overline d$. Since the
inverses of these maps are also maps of the same form, they are in
fact locally bi-Lipschitz. This shows that the local product structure
on $\overline\splitting$ agrees with $\overline d$.
\end{proof}

\begin{proposition}
\label{pr:quotientstep}
The dynamical system $(\overline{\X}, \overline f)$, where $\overline
f$ is the map induced by $\overline F$ is a connected and locally
connected Smale space. The quotient map $\overline
\splitting\arr\overline{\X}$ is a splitting with the group of deck
transformations $\overline G$.
\end{proposition}

\begin{proof}
It follows from Theorem~\ref{th:rdaction} that $\overline
f:\overline{\X}\arr\overline{\X}$ is a well defined homeomorphism.
The plaques of the direct product structure of $\overline\splitting$ are
continuous images of plaques of $\splitting$, hence they are connected. The map
$\overline F$ is a lift of $\overline f$.

The map $\overline\pi:\overline\splitting\arr\overline\X$ is uniformly locally
bi-Lipschitz with respect to $\overline d$ and the Hausdorff distance on
$\overline\X$, by Proposition~\ref{pr:topologyonoverline}. It follows
that the Hausdorff distance agrees with the quotient topology on
$\overline\X$ and that $\overline\pi$ is a covering.

It follows from Proposition~\ref{pr:topologyonoverline},
Lemma~\ref{lem:locproddelta}, and the fact that $\overline G$
preserves the direct product structure on $\overline\splitting$,
that the image under $\overline\pi$ of
the local product structure on $\overline\splitting$ is a well defined local
product structure on $\overline\X$.

Suppose that $\overline x, \overline y\in\overline{\splitting}$ are such that
$[\overline x, \overline y]=\overline x$ (i.e., $\overline x$ and
$\overline y$ belong to the same stable plaque of $\overline\splitting$). By
Lemma~\ref{lem:stablemin}, there exist $x\in\overline x$ and $y\in\overline
y$ belonging to the same stable plaque of $\splitting$. Note also that it
follows from Lemmas~\ref{lem:overlined} and~\ref{lem:locproddelta}
that there exists a constant $C>1$, not depending on $\overline x$ and
$\overline y$, such that we can find $x$, $y$ satisfying
\[C^{-1}\overline d(\overline x, \overline y)\le d_\splitting(x, y)\le C\overline
d(\overline x, \overline y),\]
provided $\overline d(\overline x, \overline y)$ is small enough.

Then $d_\splitting(F^n(x),
F^n(y))\le C\lambda^nd_\splitting(x, y)$ for some fixed $C>1$ and $\lambda\in (0,
1)$. It follows that there exists a constant $C_2>1$ such that
for any two points $\overline x, \overline
y\in\overline\splitting$ such that $[\overline x, \overline y]=\overline x$,
and  $\overline d(\overline x, \overline y)$ is
small enough we have
$\overline d({\overline F}^n(\overline x), {\overline
  F}^n(\overline y))\le C_2\lambda^n\overline d(\overline x, \overline
y)$ for all $n$. Analogous statement about the unstable plaques of
$\overline\splitting$ is proved in the same way.

The map $\overline{\pi}:\overline{\splitting}\arr\overline{\X}$ is locally
bi-Lipschitz with respect to $\overline d$ and the Hausdorff distance on
$\overline{\X}$. It follows
that the images of the stable and unstable plaques of $\overline\splitting$
are stable and unstable leaves of $(\overline{\X}, \overline f)$.

Suppose that $\overline x, \overline y\in\overline\splitting$ are such that
$\overline\pi(\overline x)=\overline\pi(\overline y)$, i.e., the
$\R^d$-orbits $\overline x$ and $\overline y$ are mapped to the same
set in $\X$. Then there
exist $x\in\overline x$ and $y\in\overline y$ such that
$\pi(x)=\pi(y)$, i.e., there exists $g\in G$ such that $g(x)=y$. Then
$\overline g(\overline x)=\overline y$. Since the action of $\overline
G$ on $\overline\splitting$ is free, we conclude that $\overline G$ is the
group of deck transformations of the splitting $\overline\pi:\overline
\splitting\arr\overline\X$.
\end{proof}

\begin{proposition}
\label{pr:fixedpoint}
If a connected and locally connected Smale space $(\X, f)$ has a splitting with a
nilpotent torsion free group of deck transformations, then $f$ has a fixed point.
\end{proposition}

\begin{proof}
We argue by induction on the nilpotency class. We know that the
statement is true for abelian groups of deck
transformations, see Theorem~\ref{th:virtuallyabelian}. The Smale space
$(\overline{\X}, \overline f)$ is a locally connected Smale space with
the group of deck transformations $\overline G$ of lower
class.
Therefore, by the inductive hypothesis, $\overline f$ has a
fixed point. Its preimage in $\X$ is an $f$-invariant
torus $T\subset\X$ equal to an orbit of the $\R^d/\Z^d$-action.
It follows from the definition of the action of
$\R^d$ on $\splitting$ that $T$ is locally closed with respect to the local product
operation $[\cdot, \cdot]$, hence $(T, f)$ is a Smale space, and
$f$ restricted to this torus is a hyperbolic automorphism, hence it
has a fixed point (see Theorem~\ref{th:virtuallyabelian}).
\end{proof}

If $G$ is a torsion-free finitely generated nilpotent group, then
there exists a unique simply connected nilpotent Lie group $L$ such
that $G$ is isomorphic to a co-compact lattice in $L$,
see~\cite{malcev2_en}. Moreover, every automorphism of $G$ is uniquely
extended to $L$. The Lie group $L$ is called the \emph{Malcev
  completion} of $G$.

Let $G$ be a finitely generated torsion free nilpotent group, and let
$\phi$ be its automorphism. We say that $\phi:G\arr G$ is
\emph{hyperbolic} if its unique extension $\phi:L\arr L$ to the Malcev
completion is hyperbolic, i.e., if the automorphism $D\phi$ of the
Lie algebra of $L$ has no eigenvalues on the unit circle.

\begin{proposition}
\label{pr:phihyperbolic}
Let $(\X, f)$ be a connected and locally connected Smale space with a
splitting $\pi:\splitting\arr\X$ with a torsion free nilpotent group of deck
transformations $G$. Let $\phi$ be an automorphism of $G$ induced by a
lift $F$ of $f$ which has a fixed point in $\splitting$. Then $\phi$ is hyperbolic.
\end{proposition}

\begin{proof}
Let us prove our proposition by induction on the nilpotency class of
$G$. It is true for abelian groups, by
Proposition~\ref{pr:hyperbolicautomorphism}.

Suppose that we have proved the proposition for all nilpotent groups
of class $n$. Suppose that nilpotency class of $G$ is $n+1$.
By Proposition~\ref{pr:quotientstep}
and the inductive hypothesis, the
automorphism of $\overline G$ induced by $\phi$ is hyperbolic.

Let $x_0\in\splitting$ be the fixed point of $F$. Then $\overline
x_0\in\overline\splitting$ is a fixed point of $\overline F$. The image of
$\overline x_0$ in $\X$ is an $f$-invariant torus $T$, such that $(T,
f)$ is topologically conjugate to a hyperbolic automorphism of the torus.

The map $\pi:\overline x_0\arr T$ is its splitting with the group of deck
transformations equal to $Z(G)=\Z^d$. It follows then from
Proposition~\ref{pr:hyperbolicautomorphism} that the
restriction of $\phi$ onto $Z(G)$ is hyperbolic.

We see that restriction of $\phi:L\arr L$ onto $Z(L)$ and the
automorphism induced by $\phi$ on $L/Z(L)$ both are hyperbolic, hence
$\phi$ itself is hyperbolic.
\end{proof}

\begin{theorem}
\label{th:nilpotrigidity}
Let $(\X, f)$ be a connected and locally connected Smale space with a
splitting $\pi:\splitting\arr\X$ with a torsion free nilpotent group of deck
transformations $G$. Let $\phi$ be an automorphism of $G$ induced by a
lift of $f$. Let $f_L:G\backslash L\arr G\backslash L$ be the
diffeomorphism induced by $\phi$. Then
$(\X, f)$ and $(G\backslash L, f_L)$ are topologically conjugate.
\end{theorem}

\begin{proof}
Let $F$ be a lift of $f$ to $\splitting$ with a fixed point
$x_0$. Extend the automorphism $\phi:G\arr G$ to an automorphism
$\phi:L\arr L$ of the Lie group.

\begin{proposition}
\label{pr:equivariantmap}
There exists a $G$-equivariant map
$h:\splitting\arr L$ such that $\phi\circ h=h\circ F$ and $h(x_0)=1$. 
\end{proposition}

\begin{proof}
Let us show at first that there exists a $G$-equivariant map $h_0:\splitting\arr L$.
The space $\X$ is a quotient of the Cantor set under a finite-to-one
map (see~\cite{bowen:markovpartition,fried}) with an upper bound on
the cardinality of its fibers. It follows then from
Hurewicz formula~\cite{kur_top} that $\X$ has finite topological dimension.

By a theorem of Alexandroff~\cite{alexandroff:simplicialapprox}, $\X$
is homeomorphic to an inverse limit of simplicial complexes, which are
nerves of finite open coverings of $\X$. We can make the elements of
the coverings sufficiently small, so that they can be lifted to a
$G$-invariant covering of $\splitting$. It follows that $\splitting$ is an inverse limit
of a sequence of simplicial complexes with $G$-actions and
$G$-equivariant maps between them. In particular, there exists a
$G$-equivariant map $A$ from $\splitting$ to a simplicial complex $\Delta$ with a $G$ action
on it. Since $L$ is homeomorphic to $\R^n$, there exists
a $G$-equivariant map $B:\Delta\arr L$. 
Composition $h_0=B\circ A$ is then a
$G$-equivariant map from $\splitting$ to $L$.

Let us show now that there exists a $G$-equivariant map
$h:\splitting\arr L$ such that $\phi\circ h=h\circ F$. 
We will use the arguments of~\cite[Theorem~2.2]{franks}, which we
repeat here for the sake of completeness, and since our setting is
slightly different.

Consider the space $Q$ of all continuous maps $\gamma:\X\arr L$ such
that $\gamma(\pi(x_0))=1$ with the topology of uniform convergence on $\X$. It is
a nilpotent group (of the same class as $L$) with respect to pointwise
multiplication. Define $\Phi_0(\gamma)=\phi^{-1}\circ\gamma\circ f$. Then
$\Phi_0$ is a continuous automorphism of the group $Q$.

Let $\mathfrak{L}$ be the Lie algebra of $L$, and let
$\exp:\mathfrak{L}\arr L$ be the exponential map. It is a
diffeomorphism, since $L$ is simply connected and nilpotent.
Let $\mathfrak{Q}$ be the Banach space of continuous maps
$\X\arr\mathfrak{L}$ mapping $\pi(x_0)$ to zero. Then
$\Log:\gamma\mapsto\exp^{-1}\circ\gamma$ is a homeomorphism of $Q$
with $\mathfrak{Q}$.

Let $T_0:Q\arr Q$ be defined by $T_0(\gamma)=\Phi_0(\gamma)\gamma^{-1}$. Let
us show that $T_0$ is a homeomorphism. We show at first that it is a
local homeomorphism at the identity (i.e., the constant map $x\mapsto
1$), using the homeomorphism
$\Log:Q\arr\mathfrak{Q}$ and computing the derivative of $T=\Log\circ
T_0\circ\Log^{-1}$. Denote $\Phi=\Log\circ \Phi_0\circ\Log^{-1}$.

We have $\exp\circ d\phi=\phi\circ\exp$, where $d\phi$ is the
derivative of $\phi:L\arr L$ at the identity. 
It follows that $d\phi^{-1}\circ\exp^{-1}=\exp^{-1}\circ\phi^{-1}$,
and for every $\gamma\in\mathfrak{Q}$ we have
\begin{multline*}\Phi(\gamma)=\Log\circ \Phi_0\circ\Log^{-1}(\gamma)=\\
\exp^{-1}\circ\phi^{-1}\circ\exp\circ\gamma\circ f
=d\phi^{-1}\circ\exp^{-1}\circ\exp\circ\gamma\circ
f=\\
d\phi^{-1}\circ\gamma\circ f.
\end{multline*} It follows that
$\Phi:\mathfrak{Q}\arr\mathfrak{Q}$ is linear.

For every $\gamma\in\mathfrak{Q}$, we have
\begin{multline*}
T(\gamma)=\Log\circ T_0\circ\Log^{-1}(\gamma)=\\
\exp^{-1}(T_0(\exp\circ\gamma))=\exp^{-1}(\Phi_0(\exp\circ\gamma)\cdot (\exp\circ\gamma)^{-1})=\\
\exp^{-1}(\exp\circ \Phi(\gamma)\cdot (\exp\circ\gamma)^{-1})=\\
\Log(\Log^{-1}(\Phi(\gamma))\cdot\Log^{-1}(\gamma)^{-1})=
\Log(\Log^{-1}(\Phi(\gamma))\cdot\Log^{-1}(-\gamma)),
\end{multline*}
since
$(\Log^{-1}(\gamma(x)))^{-1}=(\exp(\gamma(x)))^{-1}=\exp(-\gamma(x))$
for all $x\in\X$.

Let us compute the derivative of $T$ at zero. If
$\gamma\in\mathfrak{Q}$, then
\[\lim_{t\to 0}\frac 1tT(t\gamma)=\lim_{t\to 0}\frac
1t\Log(\Log^{-1}(\Phi(t\gamma))\Log^{-1}(-t\gamma))=\Phi(\gamma)-\gamma,\]
by the Campbell-Hausdorf formula.

Since $d\phi$ is hyperbolic, there exist a direct sum decomposition
$\mathfrak{L}=\mathfrak{L}_+\oplus\mathfrak{L}_-$ and constants $C>0$
and $0<\lambda<1$ such that $\|d\phi^n(v)\|\le C\lambda^n\|v\|$ for
all $n\ge 0$ and $v\in\mathfrak{L}_+$, and $\|d\phi^{-n}(v)\|\le
C\lambda^n\|v\|$ for all $n\ge 0$ and $v\in\mathfrak{L}_-$. Define
\[\mathfrak{Q}_*=\{\gamma\in\mathfrak{Q}\;:\;\gamma(\X)\subset\mathfrak{L}_*\},\]
for $*\in\{+, -\}$. Since $\mathfrak{L}_+$ and $\mathfrak{L}_-$ are
$d\phi$-invariant, the spaces $\mathfrak{Q}_+$ and $\mathfrak{Q}_-$
are $\Phi$-invariant. We obviously have
$\mathfrak{Q}=\mathfrak{Q}_+\oplus\mathfrak{Q}_-$,
$\|\Phi^n(\gamma)\|\le C\lambda^n\|\gamma\|$ for $n\ge 0$,
$\gamma\in\mathfrak{Q}_+$, and $\|\Phi^{-n}(\gamma)\|\le
C\lambda^n\|\gamma\|$ for $n\ge 0$, $\gamma\in\mathfrak{Q}_-$. It
follows that $\Phi-I$ is invertible.

This shows that $T$ is a local homeomorphism at zero. Consequently,
$T_0$ is a local homeomorphism at the identity of $Q$. Let us show
that $T_0$ is surjective. 
 Let $Z_1(Q)=Z(Q)\subset Z_2(Q)\subset\cdots\subset Z_n(Q)=Q$ be the upper
central series of $Q$. Let us prove by induction on $i$ that
$T_0(Q)\supset Z_i(Q)$.
It is easy to see that
$T_0(\gamma_1\gamma_2)=T_0(\gamma_1)T_0(\gamma_2)$ for all $\gamma_1,
\gamma_2\in Z(Q)$. Since $T_0$ is a local homeomorphism at the identity and
$Z(Q)$ is generated by any neighborhood of the identity (as any
connected topological group,
see~\cite[Theorem~15, p.~76]{pontryagin:topgroups}, this
implies that $Z(Q)\subset T_0(Q)$.

Suppose that we have proved that $Z_i(Q)\subset
T_0(Q)$. Let $T_0(\gamma_1), T_0(\gamma_2)\in Z_{i+1}(Q)$.
Then
\begin{multline*}
T_0(\gamma_1)T_0(\gamma_2)=\Phi_0(\gamma_1)   \gamma_1^{-1}\Phi_0(\gamma_2)\gamma_2^{-1}=\\
\Phi_0(\gamma_1)\Phi_0(\gamma_2)\cdot
\left(\Phi_0(\gamma_2)^{-1}\gamma_1^{-1}\Phi_0(\gamma_2)\gamma_2^{-1} \gamma_1\gamma_2\right)\cdot
\gamma_2^{-1}\gamma_1^{-1}.
\end{multline*}
We have 
\begin{multline*}
\gamma'=\Phi_0(\gamma_2)^{-1}\gamma_1^{-1}\Phi_0(\gamma_2)\gamma_2^{-1}
\gamma_1\gamma_2=
\gamma_2^{-1}(\Phi_0(\gamma_2)\gamma_2^{-1})^{-1}\gamma_1^{-1}(\Phi_0(\gamma_2)\gamma_2^{-1})\gamma_1\gamma_2=\\
\gamma_2^{-1}T_0(\gamma_2)^{-1}\gamma_1^{-1}T_0(\gamma_2)\gamma_1\gamma_2=
\gamma_2^{-1}[T_0(\gamma_2), \gamma_1]\gamma_2\in Z_i(Q)
\end{multline*}
(here $[T_0(\gamma_2), \gamma_1]$ is the commutator, and has nothing
to do with the direct product decomposition).
By the inductive hypothesis, there exists $\gamma_3\in Q$ such that
$\gamma'=T_0(\gamma_3)$. Then
\[T_0(\gamma_1)T_0(\gamma_2)=\Phi_0(\gamma_1)\Phi_0(\gamma_2)\Phi_0(\gamma_3)\gamma_3^{-1}\gamma_2^{-1}\gamma_1^{-1}=T_0(\gamma_1\gamma_2\gamma_3)\in
T_0(Q).\]
Since $Z_{i+1}(Q)$ is connected, hence 
generated by any neighborhood of the identity,
it follows that $Z_{i+1}(Q)\subset T(Q)$.

Since $\Phi-I$ is invertible, the only fixed point of $\Phi$ is
$0$. Consequently, the only fixed point of $\Phi_0$ is the unit of
$Q$. If $T(\gamma_1)=T(\gamma_2)$, then
$\Phi_0(\gamma_1)\gamma_1^{-1}=\Phi_0(\gamma_2)\gamma_2^{-1}$, hence
$\Phi_0(\gamma_2^{-1}\gamma_1)=\gamma_2^{-1}\gamma_1$. But then
$\gamma_1=\gamma_2$, as the identity is the only fixed point of $\Phi_0$.

We have proved that $T_0:Q\arr Q$ is a homeomorphism. Let $h_0:\splitting\arr
L$ be any $G$-equivariant map. Consider the map
\[\overline\gamma(x)\mapsto (\phi^{-1}\circ h_0\circ F(x))^{-1}\cdot
h_0(x).\]
It is easy to see that for every $g\in G$ we have
\[\overline\gamma(g(x))=\overline\gamma(x),\]
i.e., $\overline\gamma$ is constant on $G$-orbits, hence it descends
to a continuous map $\gamma:\X\arr L$, which is an element of
$Q$. There exists $\gamma'\in Q$ such that $T(\gamma')=\gamma$. 
Then $F_0(\gamma')(\gamma')^{-1}=\gamma$, which means that
$\gamma'(f(x))=\phi(\gamma(x))\cdot\phi(\gamma'(x))$.

Then
the map $h(x)=h_0(x)\cdot\gamma'(\pi(x))$ is $G$-equivariant, and
\begin{multline*}
h(F(x))=h_0(F(x))\cdot\gamma'(f(\pi(x)))=\\
h_0(F(x))\cdot\phi(\overline\gamma(x))\cdot\phi(\gamma'(\pi(x)))=\\
h_0(F(x))h_0(F(x))^{-1}\phi(h_0(x))\phi(\gamma'(\pi(x)))=\\
\phi(h_0(x)\gamma'(\pi(x)))=\phi(h(x)),
\end{multline*}
which finishes the proof of the proposition.
\end{proof}

Theorem~\ref{th:equivariance} shows now that $(\X, f)$ and $(G\backslash
L, f_L)$ are topologically conjugate.
\end{proof}

Let us finish the proof of Theorem~\ref{th:Smaleinfranil}. Let $(\X,
f)$ be a connected and locally connected Smale space, and let
$\pi:\splitting\arr\X$ be its splitting with a virtually nilpotent group of
deck transformations $G$. Let $F$ be a lift of $f$ to $\splitting$. Let $\phi$
be the automorphism induced by $F$ on $G$.

Every finitely generated virtually nilpotent group $G$ contains a torsion
free nilpotent subgroup $G_0$ of finite index
(see~\cite[17.2.2]{kargm}). For every $g\in G$ and $n\in\Z$ the
subgroup $g^{-1}\phi^n(G_0)g$ has the same index in $G$ as
$G_0$. There exists only a finite number of subgroups of given index
in a finitely generated group. Taking then intersection of all subgroups
of the form $g^{-1}\phi^n(G_0)g$ for $n\in\Z$ and $g\in G$,
we get a normal $\phi$-invariant torsion free nilpotent
subgroup $G_1$ of finite index in $G$. It will be finitely generated
as a finite index subgroup of a finitely generated group.

Then $G_1\backslash\splitting$ together with the map $f_1$ induced by $F$ is a Smale
space. It is a finite covering of $\X$, and its group of deck
transformations is $G_1$. Then, by Proposition~\ref{pr:fixedpoint},
$f_1$ has a fixed point, hence we may assume that $F$ has a fixed
point $x_0$. We assume then that $\phi:G\arr G$ is given by $\phi(g)(x_0)=F(g(x_0))$.

Let $L$ be the Malcev completion of $G_1$. Extend $\phi:G_1\arr G_1$
to an automorphism $\phi:L\arr L$.
Then, by Theorem~\ref{th:nilpotrigidity}, there exists
a homeomorphism $\Phi:L\arr\splitting$ conjugating the actions of $G_1$ on $\splitting$
and $L$, and such that $\Phi\circ F=\phi\circ L$. Note that $\Phi(1)=x_0$,
$\Phi(L_+)=W_+$, and $\Phi(L_-)=W_-$, where $L_+=W_+(1), L_-=W_-(1)$ are the stable
and unstable plaques of the identity element of $L$, and $W_+=W_+(x_0), W_-=W_-(x_0)$
are the stable and unstable plaques of $x_0$.

Note that $L_+=\{g\in L\;:\;\lim_{n\to+\infty}\phi^n(g)=1\}$ and
$L_-=\{g\in L\;:\;\lim_{n\to-\infty}\phi^n(g)=1\}$ are closed
subgroups of $L$ (they are closed since they are plaques of a
splitting). The stable plaques of $L$ are the left cosets of $L_+$; the unstable
plaques of $L$ are the left cosets of $L_-$.

Consider the action of $G$ on $L$ obtained by conjugating by $\Phi$ the action
of $G$ on $\splitting$. The action of its subgroup $G_1\le G$ will
coincide with the natural action of $G_1\le L$ on $L$ by left
multiplication.

The Smale space $(\X, f)$ is then topologically conjugate to the
homeomorphism induced by $\phi$ on $G\backslash L$.

\begin{proposition}
\label{pr:byaffinetransformations}
The group $G$ acts on $L$ by affine transformations.
\end{proposition}

\begin{proof}
The action of $G_1\le G$ on $L$ coincides with the natural left action of $G_1$ on
$L$ as a subgroup of $L$. The action of $G$ on $G_1$ by conjugation
can be uniquely extended to an action of $G$ on $L$ by
automorphism. Denote by $\alpha_g(h)$ for $g\in G$ and $h\in L$ the
image of $h$ under the automorphism of $L$ equal to the extension of the
automorphism $h\mapsto ghg^{-1}$ of $G_1$.

Let $g\in G$. Then $a_g=g(1)=\Phi^{-1}(g(x_0))$ is an element of $L$. Let
us prove that the action of $g$ on $L$ is given by the formula
\[g(x)=\alpha_g(x)\cdot a_g.\]

Consider the map $A_g(x)=a_g^{-1}\alpha_g(x)a_g:L\arr L$. Note that if $h\in G_1$, then $A_{hg}=A_g$, since $a_{hg}=hg(1)=ha_g$ and $\alpha_{hg}(x)=h\alpha_g(x)h^{-1}$. It follows that there is a finite number of possibilities for $A_g$, since $G_1$ has finite index in $G$. Note also that $\phi(a_g)=\phi(g(1))=\phi(g)(1)$ and $\phi(\alpha_g(x))=\alpha_{\phi(g)}(\phi(x))$, so that
$\phi(A_g(x))=A_{\phi(g)}(\phi(x))$.

Suppose that $x\in L_+$. Then $\phi^n(x)\to 1$ as $n\to\infty$. Since
the set of possible maps of the form $A_{\phi^n(g)}$ is finite and
they are continuous, we have 
\[\phi^n(A_g(x))=A_{\phi^n(g)}(\phi^n(x))\to 1,\]
i.e., $A_g(x)\in L_+$. Consequently, the maps $A_g$ preserve $L_+$.

If $x, y$ belong to one stable plaque, then $x^{-1}y\in L_+$, hence
\[A_g(x^{-1}y)=(\alpha_g(x)a_g)^{-1}(\alpha_g(y)a_g)\in L_+.\]
Consequently, the affine map $x\mapsto \alpha_g(x)a_g$ preserves the
stable plaques of $L$. It is proved in the same way that it
preserves the unstable plaques, hence it preserves the local product structure.

Let $h\in G_1$. Then $g(hL_+)=ghg^{-1}(g(L_+))=\alpha_g(h)a_gL_+$,
since $g(L_+)$ is the stable plaque $a_gL_+$ of the point
$a_g=g(1)$. By the same argument, $g(hL_-)=\alpha_g(h)a_gL_-$ for all
$h\in G_1$.

Let $R\subset L$ be a relatively compact open rectangle such that $G_1\backslash
R=G_1\backslash L$ and $1\in R$. Then for every $x\in L_-$ and every $n\in\N$ there
exists $g_n\in G_1$ such that $x\in\phi^n(g_nR)$. Note that then
distance from $x$ to $\phi^n(g_nL_+)$ is exponentially decreasing with
$n$. It follows  that the union of the stable plaques of the form
$hL_+$ for $h\in G_1$ is dense in $L$. Similarly, the union of the
unstable plaques of the form $hL_-$ for $h\in G_1$ is also dense in
$L$.

The actions of the maps $x\mapsto g(x)$  and
$x\mapsto\alpha_g(x)a_g$ on the stable and unstable plaques of the
form $hL_\pm$ for $h\in G_1$ coincide. Both maps are continuous on $L$
and preserve the direct product structure, hence they are equal. 
\end{proof}

This finishes the proof of Theorem~\ref{th:Smaleinfranil}.
\end{proof}

\section{Smale spaces with pinched spectrum}
\label{s:pinched}

\subsection{Splitting}

\begin{defi}
Let $(\X, f)$ be a Smale space such that $\X$ is connected and locally connected.
Let $a_0, a_1$ be the stable lower and upper critical exponents, and let
$b_0, b_1$ be the unstable lower and upper critical exponents.

We say that the Smale space has \emph{pinched spectrum} if
\[\frac{a_0}{a_1}+\frac{b_0}{b_1}>1,\]
\end{defi}

\begin{theorem}
\label{th:developable}
A Smale space with pinched spectrum is splittable.
\end{theorem}

\begin{proof}
Choose numbers $\alpha_0, \alpha_1, \beta_0, \beta_1$ such that
$0<\alpha_0<a_0\le a_1<\alpha_1$, $0<\beta_0<b_0\le
b_1<\beta_1$, and
\[\frac{\alpha_0}{\alpha_1}+\frac{\beta_0}{\beta_1}>1.\]

Let $d_+$ and $d_-$ be metrics associated with the
internal log-scales $\ell_+$ and $\ell_-$ on the stable and unstable leaves of the exponents
$\alpha_0$ and $\beta_0$, respectively. All distances inside the
leaves will be measured using these metrics.

Let $\mathcal{R}$ be a finite covering of $\X$ by small open
rectangles. We assume (making the rectangles small enough) that the
holonomies inside the rectangles $R\in\mathcal{R}$ are
bi-Lipschitz. Then it follows from equalities~\eqref{eq:ellplusf}
and~\eqref{eq:ellminusf} in Section~\ref{s:Smalespaces} that the holonomies inside $f^k(R)$ for all
$R\in\mathcal{R}$ and $k\in\Z$ are bi-Lipschitz with a common
Lipschitz constant $L>1$.

There exists $\epsilon>0$ such that for any $x\in\X$
there exists a rectangle $R\in\mathcal{R}$ such that the
$\epsilon$-neighborhood of $x$ in $W_+(x)$ (with respect to $d_+$) and
the $\epsilon$-neighborhood of $x$ in $W_-(x)$ (with respect to $d_-$)
are contained in $R$. Then for some constant $c>0$ and
for every $k\in\Z$, $x\in\X$ there exists $R\in\mathcal{R}$
such that the $ce^{-\alpha_0 k}$-neighborhood of $x$ in
$W_+(x)$ is contained in $f^k(R)$, and the $ce^{\alpha_0
  k}$-neighborhood of $x$ in $W_-(x)$ is contained in
$f^k(R)$.

Choose $x_0\in\X$. Let us construct a splitting $\pi_{x_0}:W_+(x_0)\times W_-(x_0)\arr\X$.
Let $x\in W_+(x_0)$, and let $n$ be a positive
integer. Since $\alpha_1$ is a stable upper exponent, there exists a sequence
$x_0, x_1, \ldots, x_{m_1}=x$ of points of $W_+(x_0)$ such that $m_1\le
C_1e^{\alpha_1(n-\ell_-(x_0, x))}$, and $\ell_+(x_i,
x_{i+1})\ge n$, for some constant $C_1$.

Passing to $d_+$, we get that
\[m_1\le C_2d_+(x_0, x)^{\alpha_1/\alpha_0}e^{\alpha_1 n},\qquad d_+(x_i,
x_{i+1})\le C_3e^{-\alpha_0 n}.\]

For every $k\in\Z$ there exist rectangles $R_{i, 1}\in\mathcal{R}$ such
that the $ce^{-\alpha_0 k}$-neighborhood of $x_i$ in
$W_+(x_i)=W_+(x_0)$ and the $ce^{\beta_0
  k}$-neighborhood of $x_i$ in $W_-(x_i)$ belong to
$f^k(R_{i, 1})$. If $ce^{-\alpha_0 k}>C_3e^{-\alpha_0 n}$, then $R_{i,
  1}$ contains $x_{i-1}$ and $x_{i+1}$. The last inequality is equivalent to
$k\le n-r_1$ for some constant $r_1\in\Z$. Choose $k_1=n-r_1$, and find a
sequence of rectangles $R_{i, 1}$ satisfying the above conditions for $k=k_1$.

\begin{figure}
\includegraphics{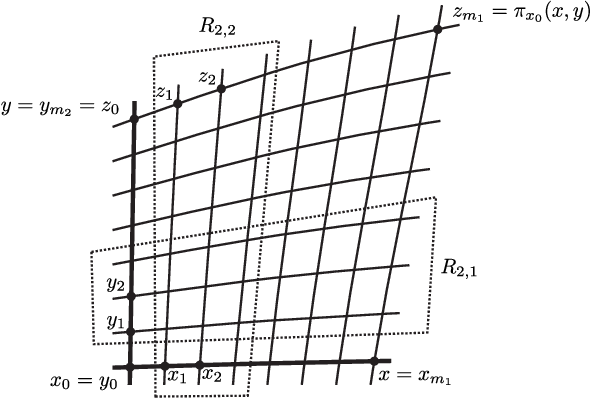}
\caption{Splitting}
\label{fig:pinched}
\end{figure}

Let $y\in W_-(x_0)$ be such that $d_-(x_0, y)\le
ce^{\alpha_0k_1}=C_4 e^{\alpha_0 n}$ (where
$C_4=ce^{-\alpha_0r_1}$). Denote $z_0=y$, $z_1=[x_1,
z_0]_{f^{k_1}(R_{0, 1})}$,
$z_2=[x_2, z_1]_{f^{k_1}(R_{1, 1})}$, e.t.c.. If all points $z_0,
\ldots, z_{m_1}$ are defined, then we say that $y$ \emph{can be continued to
  $x$}, and denote $\pi_{x_0}(x, y)=z_{m_1}$. Note that $x_0$ can be continued to
$x$ and $\pi_{x_0}(x_0, x)=x$.

If $y$ can be continued to $x$, then, in the above
notation, \[d_+(z_i, z_{i+1})\le
Ld_+(x_i, x_{i+1})\le C_5e^{-\alpha_0 n},\] hence
\[d_+(y, \pi_{x_0}(x, y))\le C_5m_1
e^{-\alpha_0 n}\le C_6d_+(x_0, x)^{\alpha_1/\alpha_0}e^{(\alpha_1-\alpha_0)n}. \]

Let $k_2$ be such that
\begin{equation}
\label{eq:k2}
ce^{\alpha_0 k_2}\ge C_6d_+(x_0,
x)^{\alpha_1/\alpha_0}e^{(\alpha_1-\alpha_0)n},
\end{equation}
so that for every point
$z\in\X$ there exists a rectangle $R\in\mathcal{R}$ such that the $C_6d_+(x_0,
x)^{\alpha_1/\alpha_0}e^{(\alpha_1-\alpha_0)n}$-neighborhood of $z$ in
$W_+(z)$ is contained in $f^{-k_2}(R)$, and the
$ce^{-\beta_0k_2}$-neighborhood of $z$ in $W_-(z)$ is also contained in
$f^{-k_2}(R)$.

Inequality~\eqref{eq:k2} follows from an inequality
\[k_2\ge\frac{\alpha_1}{\alpha_0^2}\log d_+(x_0,
x)+\frac{\alpha_1-\alpha_0}{\alpha_0}n+s\]
for some constant $s$. Consequently, we can take
\begin{equation}
\label{eq:k22}
k_2=\frac{\alpha_1}{\alpha_0^2}\log d_+(x_0,
x)+\frac{\alpha_1-\alpha_0}{\alpha_0}n+s_1,
\end{equation}
where $s_1>0$ is bounded from above.

Let $n_2$ be such that $\ell_-(z_1, z_2)\ge n_2$ implies
$d_-(z_1, z_2)\le ce^{-\beta_0 k_2}$. There exists a constant $r_2$
(not depending on $k_2$) such that we can take $n_2=k_2+r_2$. Let
$y\in W_-(x_0)$. There exists a sequence $y_0=x_0, y_1, y_2, \ldots,
y_{m_2}=y$ such that \[m_2\le C_7d_-(y,
x_0)^{\beta_1/\beta_0}e^{\beta_1 n_2},\qquad d_-(y_i, y_{i+1})\le
ce^{-\beta_0 n_2}\le ce^{-\beta_0 k_2}.\]
Choose a sequence of rectangles $R_{i, 2}\in\mathcal{R}$ such that the
$ce^{-\beta_0 n_2}$-neighborhood of $y_i$ in $W_-(y_i)=W_-(x_0)$, and
the $ce^{\alpha_0 n_2}$-neighborhood of $y_i$ in $W_+(y_i)$ are
contained in $f^{-n_2}(R)$.

Suppose that $y_i$ can be continued to $x$. Then $d_+(y_i, \pi_{x_0}(x,
y_i))\le ce^{\alpha_0 k_2}\le ce^{\alpha_0 n_2}$, hence $\pi_{x_0}(x, y_i)\in
f^{-n_2}(R_{i, 2})$. Define then a sequence $z_{0, i}=y_{i+1}$,
$z_{1, i}=[x_1, z_{0, i}]_{f^{k_1}(R_0)}$, $z_{2, i}=[x_2, z_{1,
  i}]_{f^{k_1}(R_1)}$, e.t.c.. 
Each of the points $z_{j, i}$ will be defined, provided $d_-(z_{j-1, i}, x_{j-1})\le
ce^{\beta_0 k_1}$. We have an estimate
\begin{multline*}
d_-(z_{j-1, i}, x_{j-1})\le C_8(d_-(x_0, y_1)+d_-(y_1, y_2)+\cdots+
d_-(y_{j-2}, y_{j-1}))\le\\
C_8m_2ce^{-\beta_0 k_2}\le C_9d_-(y,
x_0)^{\beta_1/\beta_0}e^{(\beta_1-\beta_0)k_2}.
\end{multline*}
(We used that $n_2=k_2+r_1$ for some constant $r_1$.)

It follows that $y$ can be continued to $x$ if
\[C_9d_-(y,
x_0)^{\beta_1/\beta_0}e^{(\beta_1-\beta_0)k_2}\le ce^{\beta_0 k_1},\]
i.e., if
\[(\beta_1-\beta_0)k_2+\frac{\beta_1}{\beta_0}\log d_-(y, x_0)+s_2\le
\beta_0 n\]
for some constant $s_2$.

Replacing $k_2$ by the value given in~\eqref{eq:k22}, we get that $y$
can be continued to $x$ if
\begin{equation}
\label{eq:bign}
(\beta_1-\beta_0)\left(\frac{\alpha_1-\alpha_0}{\alpha_0}n+
\frac{\alpha_1}{\alpha_0^2}\log d_+(x_0,
  x)\right)+\frac{\beta_1}{\beta_0}\log
d_-(y, x_0)+s_3\le \beta_0 n\end{equation}
for some constant $s_3$. If
\begin{equation}
\label{eq:alphabeta}
\frac{(\beta_1-\beta_0)(\alpha_1-\alpha_0)}{\alpha_0}<\beta_0,
\end{equation}
then taking $n$ big enough, we can guarantee that
inequality~\eqref{eq:bign} is satisfied. 
Inequality~\eqref{eq:alphabeta} is equivalent to
\[(\beta_1-\beta_0)(\alpha_1-\alpha_0)<\alpha_0\beta_0,\]
i.e., to
\[\beta_1\alpha_1<\beta_0\alpha_1+\beta_1\alpha_0\Longleftrightarrow\frac{\alpha_0}{\alpha_1}+\frac{\beta_0}{\beta_1}>1.\]

It follows that if the Smale space has pinched spectrum, then every point $y\in
W_-(x_0)$ can be continued to every $x\in W_+(x_0)$, and we can define
$\pi_{x_0}:W_+(x_0)\times W_-(x_0)$ using the rules described above.

Let us show that the map $\pi_{x_0}:W_+(x_0)\times W_-(x_0)\arr\X$ is well
defined, i.e., does not depend on the choice of the rectangles $R_{i,
  1}$ (we did not use the rectangles $R_{i, 2}$ in the definition of
$\pi_{x_0}$).

It follows from the construction that the map $y\mapsto\pi_{x_0}(x, y)$ is
equal to composition of holonomy maps of a sequence of rectangles $R_{i, 1}\in
f^{k_1}(\mathcal{R})$ for some positive $k_1$.

We also showed that for every $y$, the germ of the map $y\mapsto\pi_{x_0}(x,
y)$ is equal to a germ of a holonomy in a rectangle $R_{j, 1}\in
f^{-n_2}(\mathcal{R})$ for some positive $n_2$.

Note also that given such a sequence $R_{i, 1}\in
f^{k_1}(\mathcal{R})$
we can find a sequence $R_i'\in f^{m}(\mathcal{R})$ such that $m$ is arbitrarily big and the map
$y\mapsto\pi_{x_0}(x, y)$ defined by the original sequence $R_{1, i}$ is a
restriction of the maps defined by the new sequence $R_i'$.

Suppose that $h_i:W_-(x_0)\arr W_-(x)$ for $i=1, 2$ are compositions
of holonomies defined using two sequences $R_{1, i}$, and $R_{1,
  i}'$. We may assume that both sequences belong to
$f^{k_1}(\mathcal{R})$ for some fixed $k_1$. Let $y$ belongs to the
domain of both maps $h_i$. We may assume (taking $k_1$ big enough) that $x$ and $y$ belong to
connected components of the domains of $h_i$. Then there exists a
connected chain of rectangles $R_{2, i}\in f^{-n_2}(\mathcal{R})$, for
some $n_2>0$, such that the restrictions of $h_i$ to the corresponding
plaques of $R_{2, i}$ are equal to holonomies in $R_{2, i}$. It
follows then from $h_1(x_0)=h_2(x_0)=x$ that
$h_1(y)=h_2(y)$. Consequently, $\pi_{x_0}$ is well defined.
The map $\pi_{x_0}$ is obviously a local homeomorphism.

Note that if $R\in\mathcal{R}$ is such that $x_0\in R$, then
$\pi_{x_0}:P_+(R, x_0)\times P_-(R, x_0)\arr R$ coincides with
$[\cdot, \cdot]_R$.

Let
$(a, b)\in W_+(x_0)\times W_-(x_0)$ be an arbitrary point, and let
$x_1=\pi_{x_0}(a, b)$. Then it follows from the definition of the maps
$\pi_{x_i}$ and their uniqueness that
\begin{eqnarray*}
\label{eq:pix0x1}
\pi_{x_0}(x, y)=\pi_{x_1}(\pi_{x_0}(x, b), \pi_{x_0}(a, y)).
\end{eqnarray*}
It follows that $\pi_{x_0}$ is onto, since its range contains every
rectangle $R\in\mathcal{R}$ intersecting it.
It also follows that the map $\pi_{x_0}$ is a covering, since every point
of $W_+(x_0)\times W_-(x_0)$ has a neighborhood mapped homeomorphically
by $\pi_{x_0}$ to an element of $\mathcal{R}$.

Another corollary of~\eqref{eq:pix0x1} is that $\pi_{x_0}$
homeomorphically maps the plaques of $W_+(x_0)\times W_-(x_0)$ to the leaves
of $\X$, since $\pi_{x_1}$ maps the direct factors of $W_+(x_1)\times
W_-(x_1)$ identically onto the leaves $W_+(x_1)$ and $W_-(x_1)$.
This finishes the proof of the theorem.
\end{proof}

\subsection{Polynomial growth}

\begin{theorem}
\label{th:growth}
Let $(\X, f)$ be a Smale space with pinched spectrum. Then the group of deck transformations
of the splitting of $(\X, f)$ has polynomial growth.
\end{theorem}

\begin{proof}
Our proof essentially repeats the proof of the main theorem
of~\cite{brin:dan2}. Let $\mathcal{R}$ be a finite covering of $\X$ by
open connected rectangles. Let $\pi:\splitting\arr\X$ be the splitting
constructed in Theorem~\ref{th:developable}, where
$\splitting=W_+(x_0)\times W_-(x_0)$. Denote by $\wt{\mathcal{R}}$ the
union of the sets of connected components of $\pi^{-1}(R)$ for $R\in\mathcal{R}$.

Consider the graph $\Gamma$ with the set of vertices identified with
$\wt{\mathcal{R}}$ in which two vertices are connected if the
corresponding sets have non-empty intersection. It is easy to show
(see the proof of Proposition~\ref{pr:finitegeneration}) that the graph $\Xi$
is quasi-isometric to the Cayley graph of the group $G$ of deck
transformations of $\pi$ and has the same growth rate as $G$.

Let $B(r)$ be the set of
elements of $\wt{\mathcal{R}}$ that are on distance at most $r$ in
$\Gamma$ from a vertex $R\in\wt{\mathcal{R}}$ such that $(x_0, x_0)\in R$.

Let $0<\alpha_0<a_0\le a_1<\alpha_1$ and $0<\beta_0<b_0\le b_1<\beta_1$ such that
$\alpha_0/\alpha_1+\beta_0/\beta_1>1$. Denote by $d_+$ and $d_-$ the
metrics of exponents $\alpha_0$ and $\beta_0$ on the
corresponding leaves of $\X$ and plaques of $W_+(x_0)\times W_-(x_0)$.
(Every plaque of $\splitting$ is identified with a leaf
of $\X$ by $\pi$.)

Take $R\in B(r)$. Choose a
sequence $R_0\ni (x_0, x_0), R_1, \ldots, R_m=R$ of elements of 
$\wt{\mathcal{R}}$ forming a chain in $\Gamma$ of
length $m\le r$.

We will denote by $[\cdot, \cdot]$ the direct product
structure on $W_+(x_0)\times W_-(x_0)$. Let $D_-$ and $D_+$ be the suprema of
the $d_-$- and $d_+$-diameters of the sets $[y, \cup R_i]$ and $[\cup
R_i, y]$ for all $y\in \cup R_i$ (see Figure~\ref{fig:growth}).

\begin{figure}
\includegraphics{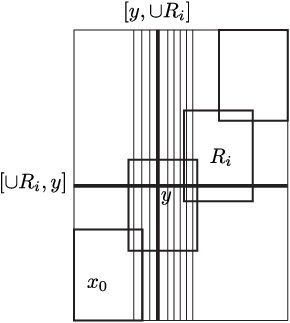}
\caption{Growth estimation}
\label{fig:growth}
\end{figure}

There exist constants $C_1, C_2$ not depending on $r$, and 
a number $n=n(r)$ such that $|n-\log D_-/\beta_0|<C_1$, and
for every point $y\in\splitting$ there exists a rectangle $V\in
f^n(\mathcal{R})$ such that the $C_2e^{-\alpha_0 n}$-neighborhood of $y$ in the stable
plaque of $y$ (with respect to $d_+$) and the set
$[y, \cup R_i]$ are both contained in $V$. 

It follows that we can find a
sequence of rectangles $V_i\in f^n(\mathcal{R})$ of length at most
$C_3e^{\alpha_1n}$ such that $V_i\cap V_{i+1}\ne\emptyset$, the
first rectangle in the sequence contains a given point of $R_{i-1}\cap R_i$, while
the last one contains a given point of $R_i\cap R_{i+1}$.

Consequently, we can find a sequence of rectangles $V_0, \ldots,
V_l\in f^n(\mathcal{R})$ of length at most
\[C_3re^{\alpha_1n}\le C_4rD_-^{\alpha_1/\beta_0}\]
such that $V_i\cap V_{i+1}\ne\emptyset$, $x_0\in V_0$, $V_m$
contains a point $x\in R=R_m$, and for every $V_i$ there exists a point
$z_i\in V_i$ such that $[z_i, \cup R_i]\subset V_i$.  Moreover, we may assume that the
chain $V_i$ covers any given in advance three point $y_1, y_2, y_3\in\cup
R_i$.

It follows that the $d_+$-distance from $[y_1, y_2]$ to $[y_3, y_2]$ is bounded
from above by
\[C_5rD_-^{\alpha_1/\beta_0}\cdot e^{-\alpha_0 n}\le
C_6rD_-^{\alpha_1/\beta_0}D_-^{-\alpha_0/\beta_0}=
C_6rD_-^{\frac{\alpha_1-\alpha_0}{\beta_0}}.\]
It follows that
\[D_+\le C_6 r D_-^{\frac{\alpha_1-\alpha_0}{\beta_0}},\]
and, by the same argument,
\[D_-\le C_7 rD_+^{\frac{\beta_1-\beta_0}{\alpha_0}}.\]
Combining the inequalities, we get
\[D_+\le C_6 r
\left(C_7rD_+^{\frac{\beta_1-\beta_0}{\alpha_0}}\right)^{\frac{\alpha_1-\alpha_0}{\beta_0}}=
C_8r^{1+\frac{\alpha_1-\alpha_0}{\beta_0}}D_+^{\frac{(\alpha_1-\alpha_0)(\beta_1-\beta_0)}{\alpha_0\beta_0}},\]
hence
\[D_+^{\frac{\alpha_1}{\alpha_0}+\frac{\beta_1}{\beta_0}-\frac{\alpha_1\beta_1}{\alpha_0\beta_0}}\le
C_8r^{1+\frac{\alpha_1-\alpha_0}{\beta_0}}.\]
Note that
$\frac{\alpha_1}{\alpha_0}+\frac{\beta_1}{\beta_0}-\frac{\alpha_1\beta_1}{\alpha_0\beta_0}=
\frac{\alpha_1\beta_1}{\alpha_0\beta_0}\left(\frac{\alpha_0}{\alpha_1}+\frac{\beta_0}{\beta_1}-1\right)>0$,
hence
\[D_+\le C_8r^{p_+}\]
for
$p_+=\left(1+\frac{\alpha_1-\alpha_0}{\beta_0}\right)
\left(\frac{\alpha_1}{\alpha_0}+\frac{\beta_1}{\beta_0}-
\frac{\alpha_1\beta_1}{\alpha_0\beta_0}\right)^{-1}$.

Similarly,
\[D_-\le C_9r^{p_-}\]
for
$p_-=\left(1+\frac{\beta_1-\beta_0}{\alpha_0}\right)\left(\frac{\alpha_1}{\alpha_0}+\frac{\beta_1}{\beta_0}-\frac{\alpha_1\beta_1}{\alpha_0\beta_0}\right)^{-1}$.

In particular (taking $C_{10}=\max(C_8, C_9)$ and $p=\max(p_+, p_-)$) we
have that  $d_-([x_0, x], x_0)$ and $d_+([x, x_0], x_0)$ are
less than $C_{10}r^p$ for any $x\in\bigcup_{R\in B(r)}R$.

Let $\mu_+$ and $\mu_-$ be the measures satisfying the conditions
of Theorem~\ref{th:SRBmeasures} (note that $(\X, f)$ is mixing by Proposition~\ref{pr:splittingmixing}). Let $\mu$ be their direct product on
$\splitting=W_+(x_0)\times W_-(x_0)$. Since the measures $\mu_+$ and $\mu_-$ on the
leaves of $\X$ are invariant under holonomies, $G$ acts by measure
preserving transformations on $\splitting$. It follows that
there exist positive constants $A_1$ and $A_2$ such that
\[A_1|B(r)|\le\mu\left(\bigcup_{R\in B(r)}R\right)\le A_2|B(r)|.\]

By the proven above, the set $\bigcup_{R\in B(r)}R$ is contained in
the direct product of the balls of radius $C_{10}r^p$ with center in
$x_0$ in $W_+(x_0)$ and $W_-(x_0)$. By condition (1) of
Theorem~\ref{th:SRBmeasures} volumes of these balls are bounded from
above by $C(C_{10}r^p)^{\eta/\alpha_0}$ and
$C(C_{10}r^p)^{\eta/\beta_0}$ for some constant $C$. It follows that
$|B(r)|$ is bounded above by a polynomial in $r$.
\end{proof}

By the Gromov's theorem on groups of polynomial growth~\cite{gro:gr}, 
$G$ is virtually
nilpotent. Theorem~\ref{th:Smaleinfranil} now implies the following
description of Smale spaces with pinched spectrum.

\begin{theorem}
\label{th:pinchedinfranil}
Every connected and locally connected Smale space with pinched
spectrum is topologically conjugate to an infra-nilmanifold automorphism.
\end{theorem}

\section{Mather spectrum of Anosov diffeomorphisms}
\label{s:mather}

Let $f:\X\arr\X$ be a diffeomorphism of a compact Riemann manifold $\X$. It induces a linear
operator $f_*$ on the Banach space of continuous vector fields by
\[f_*(\vec X)(x)=Df\circ \vec X(f^{-1}(x)).\]
By a theorem of J.~Mather~\cite{mather:charanosov}, $f$ is an Anosov
diffeomorphism if and only if the spectrum of $f_+$ does not intersect the unit
circle. It belongs then to the set
\[\{z\in\C\;:\;\lambda_1<
|z|<\lambda_2\}\cup\{z\in\C\;:\;\mu_2<|z|<\mu_1\},\]
where $0<\lambda_1<\lambda_2<1<\mu_2<\mu_1$. The tangent bundle $T\X$
is decomposed into a direct sum $W^s\oplus W^u$ such that there
exists a constant $C>1$ such that for all
vectors $\vec v_+\in W^s$, $\vec v_-\in W^u$, and for every positive integer $n$ we have
\[C^{-1}\lambda_1^n\|\vec v_+\|\le\|Df^n\vec v_+\|\le C\lambda_2^n\|\vec
v_+\|\]
and
\[C^{-1}\mu_2^n\|\vec v_-\|\le\|Df^n\vec v_-\|\le C\mu_1^n\|\vec
v_-\|.\]
For more detail,
see~\cite{mather:charanosov,brin:dan1,brin:dan2}. Every Anosov
diffeomorphism is a Smale space,
see~\cite[Proposition~5.10.1]{brin:book}.
Stable and unstable leaves of $(\X, f)$ are manifolds, and the vectors of $W^s$
and $W^u$ are tangent to the stable and unstable leaves, respectively.

\begin{proposition} Let $\lambda_1, \lambda_2, \mu_1, \mu_2$ be as above.
Then $\log\mu_2$ and $\log\mu_1$ are unstable lower and upper
exponents of $(\X, f)$,
and $-\log\lambda_2$ and $-\log\lambda_1$ are stable lower and upper
exponents of $(\X, f)$.
\end{proposition}

\begin{proof}
It is easy to show that for every fixed $k_0$ the metrics $d_{k_0}$ on leaves
$W_+(x_0)$ and $W_-(x_0)$ are quasi-isometric to the
restrictions of the Riemannian metric of $\X$ onto $W_+(x_0)$ and $W_-(x_0)$,
with the quasi-isometry constants depending only on $k_0$ (and
$(\X, f)$). For every stable leaf $W_+$ of $(\X, f)$ and all $x, y\in
W_+$, $n\in\Z$,
\[d_0(f^{-n}(x), f^{-n}(y))=d_n(x, y).\]

It follows that there exist constants $C_1>1, \Delta>0$ such that
\[C_1^{-1}\wt d_+(f^{-n}(x), f^{-n}(y))-\Delta\le d_n(x, y)\le C_1\wt
d_+(f^{-n}(x), f^{-n}(y))+\Delta,\]
where $\wt d_+$ is the Riemannian metrics on the stable
leaves.

If $\gamma$ is a curve in the stable leaf connecting $f^{-n}(x)$ to
$f^{-n}(y)$, then $f^n(\gamma)$ is a curve connecting $x$ to $y$, and
\[\mathop{\mathrm{length}}(f^n(\gamma))=
\int\left\|\frac{d}{dt}f^n\circ\gamma(t)\right\|\;dt=
\int\left\|Df^n\circ\frac{d}{dt}\gamma(t)\right\|\;dt,\]
hence
\[C^{-1}\lambda_1^n\mathop{\mathrm{length}}(\gamma)\le\mathop{\mathrm{length}}(f^n(\gamma))\le
C\lambda_2^n\mathop{\mathrm{length}}(\gamma),\]
and
\[C^{-1}\lambda_1^n\cdot\wt d_+(f^{-n}(x), f^{-n}(y))\le\wt d_+(x, y)\le
C\lambda_2^n\cdot\wt d_+(f^{-n}(x), f^{-n}(y)).\]
It follows that
\[C_1^{-1}C^{-1}d_+(x, y)\cdot\lambda_2^{-n}-\Delta\le d_n(x, y)\le
C_1C\wt d_+(x, y)\cdot \lambda_1^{-n}+\Delta\]
for all stably equivalent $x, y$ and all positive $n$.
Then, by Propositions~\ref{pr:lowerexpchar} and~\ref{pr:upperexpchar},
$-\log\lambda_1$ and $-\log\lambda_2$ are upper and lower exponents.

The case of unstable leaves is proved in the same way.
\end{proof}

M.~Brin considers
in~\cite{brin:dan1,brin:dan2,brin:israel,brinmanning} Anosov
diffeomorphisms such that either
\begin{equation}
\label{eq:brin1}
1+\frac{\log\mu_2}{\log\mu_1}>\frac{\log\lambda_1}{\log\lambda_2}
\end{equation}
or
\begin{equation}
\label{eq:brin2}
1+\frac{\log\lambda_2}{\log\lambda_1}>\frac{\log\mu_1}{\log\mu_2}.
\end{equation}
Note that since $\frac{\log\mu_1}{\log\mu_2}$ and
$\frac{\log\lambda_1}{\log\lambda_2}$ are both greater than one, each
of the inequalities~\eqref{eq:brin1} and~\eqref{eq:brin2} implies
\begin{equation}
\label{eq:mathersp}
\frac{\log\lambda_2}{\log\lambda_1}+\frac{\log\mu_2}{\log\mu_1}
>1.
\end{equation}
For instance, in the case of~\eqref{eq:brin1}:
\[\frac{\log\lambda_1}{\log\lambda_2}+\frac{\log\mu_1}{\log\mu_2}
>1+\frac{\log\mu_1}{\log\mu_2}=\frac{\log\mu_1}{\log\mu_2}\left(\frac{\log\mu_2}{\log\mu_1}+1\right)
>\frac{\log\mu_1}{\log\mu_2}\cdot\frac{\log\lambda_1}{\log\lambda_2}.\]
Multiplying by
$\frac{\log\mu_2}{\log\mu_1}\cdot\frac{\log\lambda_2}{\log\lambda_1}$,
we get~\eqref{eq:mathersp}.

Note that if $a_0$ and $a_1$ are stable lower and upper
critical exponents, and $b_0$ and $b_1$ are unstable lower and upper
critical exponents of the Smale space,
then $-\log\lambda_2\le a_0$, and $-\log\lambda_1\ge a_1$,
 $\log\mu_2\le b_0$, and $\log\mu_1\ge b_1$,
so that
\[\frac{\log\mu_2}{\log\mu_1}\le\frac{a_0}{a_1},\qquad
\frac{\log\lambda_2}{\log\lambda_1}\le\frac{b_0}{b_1},\]
and therefore
\[\frac{a_0}{a_1}+\frac{b_0}{b_1}\ge
\frac{\log\mu_2}{\log\mu_1}+\frac{\log\lambda_2}{\log\lambda_1}. \]

Consequently, each of the conditions~\eqref{eq:brin1}, \eqref{eq:brin2}
implies the inequality
\[\frac{a_0}{a_1}+\frac{b_0}{b_1}>1,\]
which is the condition of Theorems~\ref{th:developable}
and~\ref{th:growth}. In particular, we conclude that
Theorem~\ref{th:pinchedinfranil} is a generalization of the results
of~\cite{brinmanning}.

\section{Co-dimension one Smale spaces}
\label{s:codimone}

We say that a Smale space is of \emph{co-dimension one} if its stable
or unstable leaves are homeomorphic to $\R$ with respect to their
intrinsic topology.

It was proved by Franks~\cite[Theorem~6.3]{franks} and Newhouse~\cite{newhouse:codimone} that every co-dimension one
\emph{Anosov diffeomorphism} is topologically conjugate to a linear Anosov diffeomorphism of a
torus. Here we prove this statement for all locally connected and
connected Smale spaces.

\begin{theorem}
\label{th:dim1}
A locally connected and connected co-dimension one Smale space is
topologically conjugate to a co-dimension one hyperbolic automorphism
of a torus $\R^d/\Z^d$.
\end{theorem}

\begin{proof}
Let us prove at first that every co-dimension one locally connected and connected Smale space $f:\X\arr\X$ is splittable and irreducible. We assume that the stable leaves of $\X$ are homeomorphic to $\R$. 

Let $d_+(x, y)$ and $d_-(x, y)$ be the metrics on the stable and unstable leaves of $\X$ associated with the natural log-scales and some exponents, as in the proof of Theorem~\ref{th:developable}. 
Let $d_n(x, y)$ be the distance inside the stable leaves of $\X$ defined in~\ref{s:lowerexponents}. Changing the log-scale $\ell$ used in the definition of $d_n(x, y)$ to a bi-Lipschitz equivalent one (e.g., changing the entourage $U$ in the definition of $\ell$) will change $d_n$ to a metric $d_n'$ satisfying $L^{-1}d_n(x, y)-C\le d_n'(x, y)\le Ld_n(x, y)+C$. In other words, the identity map will be a quasi-isometry.

In particular, up to quasi-isometry, the metric $d_n$ can be defined in the following way. Fix a finite cover $\mathcal{R}$ of $\X$ by small open rectangles, such that the stable direction of every rectangle $R_i\in\mathcal{R}$ is homeomorphic to $\R$ (equivalently, to an open interval in $\R$). We will denote $f^n(\mathcal{R})=\{f^n(R)\;:\;R\in\mathcal{R}\}$.   By the Lebesgue covering lemma there exists $\delta_0>0$ such that for every $x\in\X$ there exists $R\in\mathcal{R}$ such that the $\delta_0$-neighborhoods of $x$ in its stable and unstable leaves belong to $R$.

Then $d_n(x, y)$ is equal (up to linear lower and upper bounds) to the smallest number $k$ such that there exists a sequence of stable plaques $P_0, P_1, \ldots, P_k$ of rectangles $R_i$ from $f^n(\mathcal{R})$ such that $P_i\cap P_{i+1}\ne\emptyset$, $x\in P_0$, $y\in P_k$. 

Fix $\epsilon>0$. For every $n\ge 1$, let $D_n$ be the smallest value of $d_n(x, y)$ for stably equivalent $x, y$ and such that $d(x, y)>\epsilon$. Let $x$ and $y$ be some points realizing the minimum $D_n$. Let $R_0, R_1, \ldots, R_{D'}\in f^n(\mathcal{R})$ be a sequence of rectangles as in the previous paragraph, were $L^{-1}D_n-C\le D'\le LD_n+C$ for some fixed $L$ and $C$.

Note that it follows from the inequality $d_{n-\Delta}(x, y)\le \frac 12(d_n(x, y)+1)$ for some $\Delta>0$ (see the beginning of the proof of Proposition~\ref{pr:bLipschlexp}) and the fact that $f^k$ induces an isomorphism $\Gamma_n(W)\arr\Gamma_{n+k}(f^k(W))$ that there exist positive constants $k$ and $l$ such that $d_n(f^k(x), f^k(y))<d_n(x, y)$. Consequently, $d(f^k(x_n), f^k(y_n))<\epsilon$. It follows (as $f$ is bi-Lipschitz) that there exists a constant $C_1>0$ such that $d(x_n, y_n)<C_1\epsilon$ for all $n$.  We may assume that $C_1\epsilon<\delta_0$ by choosing $\epsilon$ small enough.

Consider the composition of the holonomies between the unstable plaques of the rectangles $R_i$ defined by the intersections $R_i\cap R_{i+1}$ and seen as a homeomorphism from a neighborhood $U_x\subset W_-(x)$ of $x$ to a neighborhood $U_y\subset W_-(y)$ of $y$. Suppose that $x'\in U_x$, and let $y'\in U_y$ be its image under the composition. Then we get a chain of intersecting stable plaques of the rectangles $R_i$ starting from the plauqe of $x'$  and ending in the plaque of $y'$. It follows that $d_n(x', y')\le L'd_n(x, y)+\Delta'$ for some fixed $L', \Delta'$. Applying an appropriate iteration $f^s$ (where $s$ depends only on $L'$ and $\Delta'$) we get $d_n(f^s(x'), f^s(y'))<d_n(x, y)$, which implies that $d(f^s(x'), f^s(y'))<\epsilon$. Consequently, there exists a constant $r>0$ (not depending on $n$) such that $d(x', y')<r$. 

We are using now arguments similar to the agruments of the proof of Theorem~\ref{th:developable}.
By the Lebesgue covering condition for the rectangles $R_i$ and the fact that $f^{-1}$ uniformly expands the unstable leaves, we get that  the rectangles $R_i$ contain $C_2\exp(\alpha n)$-neighborhoods in the unstable direction of their intersections with the stable leaf of $x$ and $y$. If $x'\in U_x, y'\in U_y$ are as in the previous paragraph, then there exists a rectangle $R\in f^{-m}(\mathcal{R})$ containing $x'$ and $y'$ in one stable plaque and such that the $\delta_1$-neighborhood of $x'$ in the ustable direction belongs to $R$ (where $m$ and hence $\delta_1$ do not depend on $n$, $x$, and $y$). Then the holonomy of the unstable leaf from $x'$ to $y'$ extends to this neighborhood (but may be not contained in $U_x$). The smallest number of steps of size at most $\delta$ from $x$ to a point $z\in W_-(x)$ is bounded from above by a function of the form $C_3d(x, z)^p+C_4$ (where $p$ depends on the ratio of the exponent of the metric and the upper exponent of the unstable direction, see the proof of Theorem~\ref{th:developable}). The holonomy inside the rectangles $R\in f^{-s}(\mathcal{R})$ are uniformly Lipschitz. It follows that when we apply the holonomies between the unstable directions coming from the intersections $R_i\cap R_{i+1}$, one by one, then the distances between the images of $x$ and $z$ are bounded from above by a function of the form $C_5d(x, z)^p+C_6$. It follows that $z\in U_x$ if $C_5d(x, z)^p+C_6<C_2\exp(\alpha n)$. Consequently, the holonomy from $x$ to $y$ can be extended to a ball with center in $x$ of exponentially big in $n$ radius.

Let $x_n, y_n$ be the points realizing the minimum $D_n$.
Choose a strictly increasing sequence $n_k$ such that each of the sequence $x_n$ and $y_n$ converge to points $x, y$. Since $d(x_n, y_n)<C_1\epsilon$ for all $n$, such a sequence exists and the points $x, y$ are stably equivalent with $\epsilon\le d(x, y)\le C_1\epsilon$. The holonomy from the unstable plaque of $x_{n_k}$ to the unstable plaque of $x$ is a bi-Lipschitz map with the same bi-Lipschitz constant $L$ defined on an exponentially big in $n_k$ ball inside $W_-(x_{n_k})$, the same is true for $y$ and $y_{n_k}$. By the proven above, the holonomy from $W_-(x_{n_k})$ to $W_-(y_{n_k})$ is defined on an exponentially big ball with center in $x_{n_k}$. It follows that the holonomy between the unstable leaves of $x$ and $y$ is everywhere defined. 

Note that we have not used the fact that the Smale space is of co-dimension one so far. We have shown that for every locally connected Smale space there exist two different stably equivalent points $x$ and $y$ such that local holonomies between their unstable plaques can be extended to a global holonomy of their unstable leaves $W_-(x)$ and $W_-(y)$. In the case when the stable leaves are one-dimensional, this is enough to show splittability of the Smale space. It is unclear, however, what we can deduce from this fact in the general case.

Namely, let $x$ and $y$ be as above.
Then for every point $x'\in W_-(x)$ there exists $N$ and a chain of rectangles $R_0, R_1, \ldots, R_D\in f^N(\mathcal{R})$ such that $x', x\in R_0$, $y\in R_D$, $R_i\cap R_{i+1}\ne\emptyset$, and the composition of all the holonomies between the unstable plaques of $R_i$ defined by the intersections $R_i\cap R_{i+1}$ is a holonomy from $U_-(x)\subset W_-(x)$ to $U_-(y)$ such that $x, x'\in U_-(x)$ and $y\in U_-(y)$. 

The stable directions of $R_i$ are open sub-intervals of the stable leaf of $x$ and $y$ such that $x$ and $y$ are contained in the initial and the final intervals, and every two neighboring intervals intersect. It follows that these intervals cover the interval $I\subset W_+(x)$ with the endpoints $x$ and $y$. Consequently, there exists a local product preserving map $\pi:I\times W_-(x)\arr \X$ identical on $I$ and $W_-(x)$. Denote by $\pi_n$ the map $f^{-n}\circ\pi\circ f^n:f^{-n}(I)\times W_-(f^{-n}(x))\arr \X$. It is also local product preserving and identical on $f^{-n}(I)\times W_-(f^{-n}(x))$. Note  that the diameter of $f^{-n}(I)$ inside $W_+(f^{-n}(x))$ grows exponentially with $n>0$.

Let $t$ be an interior point of the interval $I$. Let $m_k$ be a sequence of positive integers converging to infinity such that the limit $t'=\lim_{k\to\infty}f^{-m_k}(t)$ exists. The maps $\pi_{m_k}:f^{-m_k}(I)\times W_-(f^{-m_k}(t))\arr\X$ as $k\to\infty$ will converge to a splitting $W_+(t')\times W_-(t')\arr\X$.

Consequently, $f:\X\arr\X$ is irreducible  (see Proposition~\ref{pr:splittingmixing}).
Let $\mu_+$ be the measure on stable leaves described in
Theorem~\ref{th:SRBmeasures}.
The group $G$ of deck transformations of a splitting of $\X$ acts on a
stable leaf $W_+(x_0)\cong\R$ by the transformations $x\mapsto[g(x), x_0]$. This
action preserves the measure $\mu_+$. Let us identify $(W_+(x_0), \mu_+)$ and
$\R$ with the Lebesgue measure by a measure-preserving
homeomorphism. Since $G$ acts by measure-preserving transformations,
the corresponding action of $G$ on $\R$ is by transformations of the
form $x\mapsto\pm x+a$ for $a\in\R$.

The action of $G$ on $W_+(x_0)$ is free, since
otherwise an unstable leaf is not mapped homeomorphically onto its
image in $\X$. But this implies that $G$ acts on $\R$ by translations,
hence it is torsion-free abelian. Therefore, by Proposition~\ref{pr:finitegeneration} and
Theorem~\ref{th:virtuallyabelian}, $(\X, f)$ is topologically
conjugate to a hyperbolic automorphism of the torus $\R^d/\Z^d$ for
some $d$.
\end{proof}

Note that the proof of Franks-Newhouse theorem due to K.~Hiraide~\cite{hiraide:franksnewhouse} also uses the measure $\mu_+$ on the stable leave.

\end{document}